\documentclass[preprint]{imsart}

\doi{10.1214/154957804100000000}
\pubyear{0000}
\volume{0}
\firstpage{0}
\lastpage{0}

\usepackage{verbatim}
\usepackage{amssymb,amsmath,amsthm,bm}
\usepackage{graphicx}
\usepackage{epsfig}

\newtheorem{theorem}{Theorem}
\newtheorem{corollary}{Corollary}[section]
\newtheorem{lemma}[corollary]{Lemma}
\newtheorem{proposition}[corollary]{Proposition}

\newtheorem{exercise}{Exercise}

\newcommand{\Prob} {{\bf P}}
\newcommand{\Z}{{\mathbb Z}} 
\newcommand{\E}{{\bf E}}
\newcommand {\N}{{\mathbb N}}
\newcommand{\Es}{{\rm Es}}

\newcommand{\R}{{\mathbb{R}}}
\newcommand{\C}{{\mathbb C}}

\newcommand{\dist}{{\rm dist}}
\newcommand{\x}{{\bf x}}
\newcommand{\y}{{\bf y}}
\newcommand{\z}{{\bf z}}
\newcommand{\w}{{\bf w}}

\def \Im {{\rm Im}}
\def \Re {{\rm Re}}
\def \p {\partial}
\def \edges {{\mathcal E}}
\def \Half {{\mathbb H}}

\def \Disk {{\mathbb D}}
\def \Square {{\mathcal S}}

\def \state {{\mathcal X}}
\def \eset {\emptyset}

\def \saws {{\mathcal W}}
\def \paths {{\mathcal K}}
\def \loop {{\mathcal L}}
\def \loops {{\loop}}

\def \tree {{\mathcal T}}
\def \sgn {{\rm sgn}}

 \def \qprob {{  Q}}

\newenvironment{definition}[1][Definition]{\begin{trivlist}
\item[\hskip \labelsep {\bfseries #1}]}{\end{trivlist}}
\newenvironment{example}[1][Example]{\begin{trivlist}
\item[\hskip \labelsep {\bfseries #1}]}{\end{trivlist}}

\newcommand{{\pe}}  {\partial_e}
\newcommand {{\lodd}} {{\mathcal J}}
\newcommand{{\inrad}} {{\rm inrad}}

\newcommand {{\cent}} {{w_0}}
\newcommand {{\eb}}  {{\bf e}}

\newcommand {{\fixeta}} {{\overline \eta}}
\newcommand {{\laplace}} {{\Delta}}
\newcommand {{\edge}}  {{\bf e}}
\newcommand {{\dedges}} {{\boldsymbol{\mathcal E}}} 
\newcommand {{\uedges}}{{\mathcal E}}

\newcommand {{\vlocal}}  {n}
\newcommand {{\delocal}} {{\bf L}}
\newcommand {{\elocal}}  {L}

\newcommand {\currents}  { {\mathcal C}}

\def \linehere { {\hrule}}
\def \labove { \mtwo \linehere \linehere \linehere \ms   }

\def \lbelow {{\ms \linehere \linehere \linehere \mtwo}}

\def \mtwo {{\medskip \medskip}}
\def \ms {{\medskip}}
\newcommand {{\whoknows}} {{\mathcal L}}

\def \greenable {green }

\newenvironment{advanced}
{ \labove \begin{quote} \begin{small}}
{ \end{small}\end{quote} \lbelow }

\def \begad{\begin{advanced}}
\def \endad{\end{advanced}}

\newcommand \Nfin {{\N}_{ \text{\tiny fin}} }
\newcommand \forest {{\tree}^f}

\newcommand \graph {{\mathcal G}}
\newcommand{\ora} {\overrightarrow}

\newcommand \bfeta  {{\boldsymbol{\eta}}}
\newcommand \bfomega  {{\boldsymbol{\omega}}}

\newcommand \aset {{\mathcal A}}
\newcommand \simpleset {{\mathcal A}^{sc}}
\newcommand  \oddloops {{\mathcal O}}
\newcommand  \mbridge  {{\nu_b}}

\newcommand{\pair}  {{\mathcal P}}
\newcommand {\lra}  {{\leftrightarrow}}

\begin{document}

\begin{frontmatter}

\title{Topics in loop measures and the loop-erased walk }
\thankstext{t1}{Research supported by National Science Foundation grant
 DMS-1513036}
\runtitle{Loop measures}


\author{\fnms{Gregory F. Lawler }  \ead[label=e1]{???}}
\address{Department of Mathematics\\ University of Chicago\\
Chicago, IL 60637-1546}

\runauthor{G. Lawler}






\end{frontmatter}

\section{Introduction}

This is a collection of notes based on a course that I gave
at the University of Chicago in fall, 2016 on ``Loop measures
and the loop-erased random walk''.   This was not intended to
be a comprehensive view of the topics but instead a personal
selection of some key ideas including some recent
results.  This course was to be followed by a course
on the Schramm-Loewner evolution ($SLE$) so there was some
emphasis on the ideas that have become important in the
study of conformally invariant processes.
 I will first give some  history of the
 main results I discuss here; this can be considered
 a   personal perspective of  the developments 
 of (some of the) ideas.  I will follow with
 a summary of the topics   in this paper.

  \subsection{Some history}

I started looking at the loop-erased random walk in
my thesis \cite{Lthesis}
 spurred by a suggestion by my advisor,
Ed Nelson. My original motivation had been
to try to understand the self-avoiding walk.  Soon in
the investigation, I found out two things: the bad news
was that this process was different than the self-avoiding
walk, but the good news was that it was a very interesting
process with many of the attributes of other models
in critical phenomena. In particular, there was
an upper critical
dimension (in this case $d=4$) and (conjecturally) conformal
invariance in two dimensions.  My thesis handled the
easiest case $d > 4$.  The four dimensional case was 
significantly harder; I did not have the chance to discuss
this in this course even though there is some recent
work on the subject \cite{LSunWu}.

The interest in loop-erased random walk increased when
the relationship between it and uniform spanning trees was
discovered \cite{MDM,Pemantle}.  I believe there were several
independent discoveries of this; one thing I know is that I was
not one of the people involved!    I found out about
it  from  Robin Pemantle who was trying
to construct the infinite spanning tree and forest.  He was able
to use my results combined with the Aldous-Broder
algorithm to show that the limit of the uniform
spanning tree was a tree for $d = 4$ and a forest
for $d > 4$.  I discuss one version of this construction
(limit of wired spanning trees) in Section \ref{spansec}.
The argument here uses an algorithm  found by 
  David Wilson.
  
Although the loop-erased random walk was first considered
for simple random walk on the integer lattice, it immediately
extends to loop-erased Markov chains.  The study of Markov
chains often boils down to questions in linear algebra, and this
turns out to be true for the loop-erased walk.  As we will
see in these notes, analysis of the loop-erased walk naturally
leads to consider the following  quantity for a subMarkov
chain on a state space $A= \{x_1,\ldots,x_n\}$:
\[         F(A) = G_{A}(x_1,x_1) \, G_{A_2}(x_2,x_2)
 \cdots G_{A_n}(x_n,x_n) , \]
 where $A_j = A \setminus \{x_1,\ldots,x_{j-1}\}$ and $G$
 denotes the Green's function.  At first, this  
  quantity looks like it
 depends on the ordering of the vertices.  I
 first noticed that the quantity was the same for
 the reverse ordering $\{x_n,x_{n-1},\ldots,x_1\}$ when
 I needed to show the reversibility of the distribution
 of the LERW in  \cite{Lconn}.  The proof in that paper,
  which is not difficult and is Exercise \ref{exer.prop.sep27}
   in these notes, shows the quantity is the same
   for any permutation of the indices.  I did not notice
   this until conversations with Lon Rosen when
he was working on a paper \cite{Rosen}.  His calculations with
matrices led to a measure on self-avoiding that looked like
it could be the loop-erased measure; however, if that
were the case we would need   
 invariance under permutation, so this caused me to check
this out. 

This fact arose again at a conference in Cortona in 1997
\cite{Cortona}.  There are three things I remember about
that conference: first, Cortona is a very pretty town; second,
this is the first time I heard Oded Schramm discuss his ideas for
what would become the Schamm-Loewner evolution; and finally, I
was told David Wilson's beautiful
algorithm that uses loop-erased random
walk to construct a uniform spanning tree.  This algorithm
was certainly a surprise for me, but I remember saying then that
if it were true I could probably verify it quickly.  In this case,
I was right --- the key fact is the invariance of $F(A)$ under 
permutations of the indices.  I published the short
proof as part of a survey paper \cite{Lkesten}; see
\cite{Wilson} for Wilson's proof using ``cycle popping''.
I believe that  Marchal \cite{Marchal} was the first to identify
  $F(A)$ as
a determinant. 

There were two major papers published in 2000 on loop-erased walk.
In \cite{Kenyon}, Rick Kenyon used a relationship between dimers
and spanning trees to prove the conjecture that the growth
exponent for two-dimensional LERW is $5/4$.  He also showed
how to compute a chordal exponent exactly.  Oded Schramm \cite{schramm}
constructed what would later be proved to be the scaling limit
for loop-erased random walk; it is now called the Schramm-Loewner
evolution.  We will not discuss the latter in these notes.  We
will not do Kenyon's proof exactly, but as mentioned below, the
proof we discuss in these notes uses an important idea from that paper.
Kenyon's result computed both a ``chordal'' and a ``radial'' exponent.
A nice identity by Fomin \cite{Fomin} can be used to derive the
chordal exponent; this derivation is in these notes.  The radial
exponent, from which the growth exponent is deduced, is more difficult.

The current theory of loop measures started with the paper
\cite{LWsoup} where the Brownian loop measure was constructed and
used to interpret some of the computations done about $SLE$
in \cite{LSWrest}.  (It was later realized that work of Symanzik
had much of the construction; indeed, in the whole subject of loop
measures one continually finds parts of the theory that are older
than one realizes!)  The random walk version was first considered
in a paper with Jos\'e Trujillo Ferreras
\cite{LJose} where a strong coupling was given between the
random walk and Brownian loop measures.  That paper did not actually
give the correct definition of the random walk loop measure (it was not
important for that paper), but we (with contribution by John Thacker
as well) soon realized that a slightly 
different version was the one that corresponded to the loop-erased
walk.   The general theory of discrete time loop measures for
general (positive) weights was given in Chapter 9 of \cite{LLimic}.

At about the same time Yves Le Jan was extending   \cite{LWsoup} by
developing a theory of continuous time loop measures on discrete
sample spaces (the case of continuous time random walk appears
  in John Thacker's
thesis that was never published).  Le Jan used the continuous
time loop soup to construct the square of the Gaussian free field.
See \cite{LeJan} and references therein.
This is similar (but not exactly the same) as the Dynkin isomorphism,
which itself is a generalization of ideas of Brydges, Fr\"ohlich,
and Spencer.    Continuous times are needed to get a continuous
random variable.  When I was trying to learn Le Jan's work, I realized
that one could also construct the field using the discrete time
loop measure and then adding exponential random variables.  This idea
also appears later in Le Jan's work.  This construction only gives
the square of the field.  A method to find the sign of the field
was found by Lupu and we discuss a version of this here although
not his complete construction.  There is also a relationship
with currents --- here we give a somewhat novel treatment but
it was motivated by the paper of Lupu and Werner \cite{LupuW}.

Another recent improvement to the loop measure theory is the 
consideration of nonpositive weights.  Loop measures are very
closely tied to problems of linear algebra and much (but certainly
not all) of the theory of ``Markov chains'' can be extended
to negative weights.  There have been two recent applications of
such weights: one is an extension of Le Jan's result to some
Gaussian fields with some negative correlations \cite{LPerlman}
and the other is a loop measure proof of the Green's function
for two-dimensional LERW \cite{BLV,LGreen}.  The latter is an improvement of
Kenyon's result, but it uses a key idea from his paper.

\subsection{Summary of notes}
  
Most of Section \ref{defsec} sets up the notation for the
notes. Much of the framework is similar to that in \cite[Chapter 9]
{LLimic}, but it is done from scratch in order to show that
nonpositive weights are allowable.   When dealing with
nonpositive weights some care is needed; if the weight is
integrable (that is, the weight as a function
on paths  is   $L^1$ 
with respect to counting measure on paths), 
then most of the operations for positive weights
are valid.  For nonintegrable weights, some results will
still hold, but some will not because many of the arguments
involve interchange of sums over paths.

Loop erasure is the topic of Section \ref{erasingsec}.  Here
we only consider the deterministic transformation of loop erasure
and see the measure it induces on paths.  The expression
involves the quantity
$F(A)$.  The invariance of this quantity under
permutations of the indices is discussed as well as the fact
that it is a determinant of the Laplacian for the weight. 

Section \ref{markovsec} discusses the loop-erased walk obtained
for Markov chains.  There are three main cases: transient chains,
where loop-erasure is done on the entire path; chains on finite
state spaces, where loop-erasure is done on the path stopped
when it hits the boundary; and (some) recurrent chains, for which
the LERW on infinite paths can be defined as a limit of finite
state spaces.  One main example is simple random walk
in two dimensions.  The relationship between the loop-erased
walk and the Laplacian random walk is discussed.  Wilson's
algorithm to generate spanning trees is discussed in Section
\ref{wilsonsec}.  The fact that the algorithm generates
uniform spanning trees on graphs works is 
surprising; however, once one is told this verifying it takes
little time (this is often true for algorithms).  Combining
this with the interpretation of $F(A)$ as a determinant gives
Kirchhoff's matrix-tree theorem as an almost immediate
corollary.  The next subsection shows a nice application
of Wilson's algorithm to understand the uniform spanning
tree or forest in $\Z^d$; the algorithm is easily defined
for infinite graphs and it is not too difficult to show that
this gives the same tree or forest as that obtained by
a limit of ``wired'' spanning trees.   We only touch
on this subject: see \cite{BKPS} for a deeper description
of such trees and forests. 

The (discrete time, discrete space) loop measures are
introduced in Section \ref{loopsoupsec}.   It is easiest
to define for rooted loops first, where it is just the
usual measure with an extra factor.  The utility of the
measure comes from the number of different ways that
one can get the same measure on {\em unrooted loops}.
We also give more emphasis to another measure on
rooted loops that uses an ordering of the vertices.
It is the discrete analog of a Brownian bubble measure
decomposition of the Brownian loop measure.
This measure is often the most useful for calculations
and estimations.  In Section \ref{loopmeasuresec}, we
find the other expression for $F(A)$ in terms of the
exponential of the loop measure.  In the next
subsection, we define a soup (the terminology comes
from \cite{LWsoup}) which for a positive weight is a 
Poissonian realization from the loop measure.  We
extend the definition of the soup for nonpositive
weights by considering the distribution of the soup.
Some of the material in Sections \ref{growingsec}
and \ref{bubblesec} may be new.  Here the ``bubble soup''
(which is a union of ``growing loops'')
version of the loop soup is studied and the soup
is shown to be given by a negative binomial process
(for the number of ``elementary'' loops).  A particular
case of this, which was known, was that the loop soup
at intensity one corresponding to the loops erased
in the loop-erasing procedure.  This is made more
explicit in Section  \ref{walksoupsec}.

Section \ref{fieldsec} discusses the results of Le Jan
and Lupu about the Gaussian field.  Some of the treatment
here is new and, as in \cite{LPerlman}, applies to
some negative weight fields.  It also uses the relation
with currents \cite{ADCS,LupuW}.   After defining the field
and giving some basic properties, we study the measure
on currents generated by the loop soup at intensity $1/2$.
We compute the distribution exactly in Theorem
\ref{oct10.theorem}.  The main work is   a combinatorial
lemma proved in Section \ref{combinsec}.
 This is a measure on discrete currents.
We then get a continuous local time by adding independent
exponential random variables for each visit to a vertex
by the discrete soup.  Given Theorem \ref{oct10.theorem}
we get the joint distribution of the current and the 
continuous local times and by integrating out the currents
we get a   new proof of the Le Jan's theorem.
We then discuss Lupu's way of producing the signs for
the field. 

The final section deals with several questions dealing
with multiple loop-erased walks.  The first is understanding
the natural measure in terms of the loop measure.  We
then discuss Fomin's identity and discuss two nontrival
applications.  The first is the derivation of the chordal
crossing exponent first calculated rigorously by Kesten.
The second is a start of the asymptotics
 of the SLE Green's function
by Bene\v{s}, Viklund, and myself \cite{BLV,LGreen}.  We do
not give a complete proof of the latter result but we
do discuss how a loop measure with negative weight
  reduces the problem to several estimates
about random walks.

I thank the members of the class for their comments
and in particular Jeffrey Shen for pointing out
a number of misprints.

\section{Definitions and notations}  \label{defsec}

Loop measures and loop-erased walks were first considered
for random walks and, more generally, Markov chains.  One
of the first things that one learns about Markov chains
on finite state spaces is that much of the
initial theory quickly boils down
to questions of linear algebra.  However, the probabilistic
interpretation gives new insights and in some cases new
techniques, e.g.,  coupling.  

The theory of (sub)Markov chains is therefore a study
of (sub)stochastic matrices.  There are times when one
does not want to restrict one's study to matrices
with nonnegative entries; indeed, many models in mathematical
physics lend very naturally to complex weights on objects.
Much of the theory of loop measures also extends to complex
weights, so we will allow them in our setup.  A disadvantage
of this is that we will need to start with a lot of notation
and definitions.   First time readers may wish to consider the
case of nonnegative entries first when trying to learn the
material.   

To be a little careful, we will adopt the following terminology.
  If  $\Lambda$ is a countable
set, we will call $\phi: \Lambda \rightarrow \C$ a
{\em function} or a {\em weight}.  We will also call $\phi$ a {\em measure} on $\Lambda$
if either $\phi \geq 0$ or
\[    \|\phi\|:= \sum_{x \in \Lambda} |\phi(x)| < \infty . \]

\begin{itemize}

\item $A, \partial A$ finite sets,
$\overline A = A \cup \partial A$.  We  
call the elements in $A$ vertices or sites.  (There
will be times that we allow infinite sets, but we assume
finite unless stated otherwise.)

\item We let $\dedges_A = A \times
A$ denote the set of {\em directed
edges} in $A$.  If $(x,y)$ is a directed
edge,  we call $x$
the {\em initial}  and $y$ the {\em terminal}
 point of the directed
edge, respectively.   Let  
\[ \overline \dedges_A = \dedges_A
 \cup (A \times \partial A) \cup
 (\partial A \times A)\]
  be the set of directed edges in $\overline{A}$
 with at least one vertex in $A$.
 We will write bold-face $\edge$ for directed
 edges. 
 We say that edge $\edge_2$ {\em follows} edge
 $\edge_1$ if the terminal vertex of $\edge_1$
 is the initial vertex of $\edge_2$.  Note that
 we do allow {\em self-edges}, i.e., edges with
 the same initial and terminal point.
 
\item  Let $\edges_A$ denote the set of {\em (undirected)
edges} in $A$ which can be viewed as equivalence
classes of $\dedges_A$ under the equivalence
$(x,y) \sim (y,x)$.  Note that $\edges_A$ includes
self-edges from $x$ to $x$.  We define
$\overline \edges_A$ similarly.  The word ``edge''
will mean undirected edge unless otherwise specified.
  We write $e$
for undirected edges.

 \item  A function $q: \overline \dedges_A \rightarrow \C$
 is called a {\em weight (on edges)}.
Weights restricted to $\dedges_A$ are the same as 
 square matrices indexed by $A$.  A weight is
 {\em symmetric} if $q(x,y) = q(y,x)$ in which case
 it is a function on $\overline \edges_A$.  It
 is {\em Hermitian} if  $q(x,y) = \overline{q(y,x)}$.

 \item  We say that $p$ is a {\em positive} weight
 if $p(\edge) \geq  0$ for all $\edge$.
When we use $p, \Prob$ for a
weight, then the assumption will be that it is a positive
weight.  If we wish to consider complex weights, we will
use $q,{\bf Q}$.  Of course, positive weights are complex
so results about complex weights apply to positive weights.

  \item If $q$ is a weight, we will write
 $|\qprob|$ for the matrix $[|q(x,y)|]$.  Note
 that $|q|$ is a positive weight.

\item   We will call $p$ a {\em Markov chain (weight)} if 
$ [p(x,y)]$  are   transition probabilities of an
irreducible Markov chain
$X_j$
on $\overline A$.
Let $\tau = \tau_A = \min\{j: X_j \in \partial A\}$;  the
assumptions imply that for $x \in A$, 
  $\Prob^x\{\tau < \infty\} = 1$.  We write $\Prob$ for
the transition matrix restricted to $A$.  It is standard
 that there is a unique positive
eigenvalue $\lambda < 1$
 of $\Prob$  such that all other eigenvalues
have absolute value at most $\lambda$.

\item More generally, we say that $q $ is an {\em integrable}
weight on $A$ if the largest positive eigenvalue of $|\qprob|$
is strictly less than one.

\item  We say that $q$ is a \greenable weight, if the eigenvalues
of $\qprob$ are all strictly less than one in absolute value.
This is a weaker condition than integrability.
 
\item  {\bf Simple random walk}
\begin{itemize}
%

\item Let $\overline A$ be a connected graph
and $A$ a strict subset of vertices.  There are two
forms of {\em simple random walk on the graph} we will
consider.  Let $d_x$ denote the degree of $x$ and write
$x \sim y$ if $x,y$ are adjacent in the graph.

\begin{itemize}

\item   {\bf Type I}.
$p(x,y) = 1/d_x$ if $x\sim y$.  In this case, the invariant probability
$\pi$ 
is proportional to $d_x$.  The chain is reversible,
that is $\pi(x) \, p(x,y) = \pi (y) \, p(y,x)$, but
is not symmetric unless all the degrees are the same.

\item  {\bf Type II.} Let $n$ be a number greater than or equal to
the largest degree of the vertices in $\overline A$,
and let $p(x,y) = 1/n$ if $x \sim y$.  This is symmetric
and hence the invariant probability is the uniform distribution.

\end{itemize}

\item  {\em Simple random walk in $\Z^d$} is a particular example.
On the whole plane, it is both a Type I or Type II walk. 
If $A$ is a finite subset of $\Z^d$ and $\p A = \{x \in \Z^d:
\dist(x,A) = 1\}$, then it is often convenient to  view
the simple random walk on $\overline A = A \cup \p A$
 as a Type II walk as above with $n=2d$.

\end{itemize}

 \item  A {\em path} or {\em walk}
  in $A$  of length $n$ is a sequence of $n+1$ 
 vertices with repetitions allowed
 \[        \omega = [\omega_0,\omega_1,\ldots,\omega_n],\;\;\;\;
    \omega_j \in A .\]
 We allow trivial paths of length $0$, $\omega = [\omega_0]$.  Any
 path of length $n >0$
 can also be represented as  
 \[    \omega =\edge_1 \oplus \cdots \oplus \edge_n, \;\;\;\;
  \edge_j = (\omega_{j-1},\omega_j) \in \dedges_A.\]
  We call $\omega_0,\omega_n$ the {\em initial} and {\em terminal}
  vertices of $\omega$, respectively.  A path of length
  one is the same as a directed edge.  We write
  $|\omega| = n$ for the length of the path.
  
\item  If $\omega^1 =   \edge_1 \oplus \cdots \oplus \edge_n,$
$\omega^2 = \edge_{n+1} \oplus \cdots \oplus \edge_{n+m}$
and  $\edge_{n+1}$ follows $\edge_n$,
  we define the {\em concatenation}
\[  \omega^1 \oplus \omega^2
 = \edge_1 \oplus \cdots \oplus \edge_n \oplus
  \edge_{n+1} \oplus \cdots \oplus \edge_{n+m}.\]
  Conversely, any concatenation of $n$ edges such that
 $\edge_j$ follows $\edge_{j-1}$ gives a path.

 \item If $\omega = [\omega_0,\omega_1,\ldots,\omega_n]
  = \edge_1 \oplus \cdots \oplus \edge_n$ is a path we
  write $\omega^R$ for the {\em reversed} path
  \[  \omega^R = [\omega_n,\omega_{n-1},\ldots,
  \omega_0] = \edge_n^R \oplus \edge_{n-1}^{R} \oplus
  \cdots \oplus \edge_{1}^R.\]
  
  \item If $x,y \in A$, we let $\paths_A(x,y)$ denote the set
  of paths in $A$ starting at $x$ and ending in $y$. If $x=y$,
  we include the trivial path $[x]$.  We let
  \[ \paths_{A,x} = \bigcup_{y \in A} \paths_A(x,y),\;\;\;\;
   \paths_A = \bigcup_{x \in A} \paths_{A,x}=
  \bigcup_{x \in A}\bigcup_{y \in A} \paths_A(x,y).\]
  
  \item  We also write $\paths_A(x,y)$ when one or both of $x,y$
  are in $\partial A$.  In this case it represents paths
  \[  \omega =[\omega_0,\ldots,\omega_n] \]
  with $\omega_0 = x, \omega_n = y$ and $\omega_j \in A$
  for $0 < j < n$. If $x,y \in \partial A$, we also require
  that $n \geq 2$, that is, that there is at least one 
  vertex of $\omega$ in $A$.   We let   
  \[  \overline \paths_A = 
   \bigcup_{x \in \overline A}\bigcup_{y \in \overline
   A} \paths_A(x,y), \;\;\;\;
     \paths_{\partial A} =  \bigcup_{x \in \p A}\bigcup_{y \in \p
   A} \paths_A(x,y).\]
   
    \item  If $V \subset \overline A$, we write
  \[     \paths_A(x,V) = \bigcup_{y \in V} \paths_A(x,y).\]
  Similarly, we write $\paths_A(V,x), \paths_A(V,V')$.
  
  \item  We call a walk in $\paths_A(x,x)$ a {\em (rooted)
  loop} rooted at $x$.  This includes the {\em trivial} loop $[x]$
  of length zero. 
  We sometimes write $l$ instead of $\omega$ for loops; $l$
  will always refer to a (rooted) loop.

  \item  A weight $q$ gives a
  weight on paths  by
  \[   q(\edge_1 \oplus \cdots \oplus
  \edge_n) = q(\edge_1) \, q(\edge_2) \cdots
     q(\edge_n) , \]
   where paths of length zero get weight one.
  Note that
   \[   q(\omega^1 \oplus \omega^2) = q(\omega^1)
    \, q(\omega^2).\]

    \item If $q$ is positive or integrable, then $q$ is a measure
    on $\overline \paths_A$.
    It is easy to see that the $q$-measure of the set of walks
of length $n$ in $\paths_A(x,y)$ is the same as the $(x,y)$
entry of the matrix $\qprob^n$.
    
    \item If $q$ is weight and $\lambda \in \C$, then
     $\lambda q $  is also a weight.  We sometimes
     write $q_\lambda$ for the weight on paths 
    induced by $\lambda q$,  that is,
  \[            q_\lambda(\omega) = \lambda^{|\omega|}
    \, q( \omega), \]
  where we recall that $|\omega|$ is the length of
  $\omega$.  For any weight $q$ there exist
$\delta > 0$ such that $q_\lambda$ is integrable for
$|\lambda| < \delta$.
    
 \end{itemize}
 

\begin{itemize}

\item  For a Markov chain, the {\em Green's function}
is given by 
\[  G_A(x,y) = \sum_{n=0}^\infty \Prob^x\{X_n = y;
  n < \tau\} = \sum_{\omega \in \paths_A(x,y)}
      p(\omega). \]
 The second expression extends immediately to \greenable
 complex weights.
 \item In matrix form
\[      G = G_A = \sum_{n=0}^\infty \qprob^n, \]
from which we get $(I-\qprob) \, G = I$, that is,
$G = (I-\qprob)^{-1}$. 
 We will write $G_A^p$ or
$G_A^q$ if we wish to emphasize the
weight that we are using.  This expression only requires
the eigenvalues of $\qprob$ to all have absolute value
less than one.
 \end{itemize}
 
 \begad

For integrable $q$,   we can  
view sampling from $q$ as a two-step process: first  sampling
from $|q|$ and then specifying a rotation $q/|q|$.

 \endad

 \begin{itemize}
\item    More generally,  
 the {\em Green's generating function}
 is defined as a function of $\lambda$,
\[       G_A(x,y;\lambda) =  \sum_{\omega \in \paths_A(x,y)}
      \lambda^{|\omega|} \, q(\omega). \]
Note that
  \[     G_A^q(x,y;\lambda) = G_A^{\lambda q}
  (x,y).\]
  
  \item 
   We will say that $l \in \paths_A(x,x)$
  is an {\em elementary loop} if it is nontrivial ($|l| > 0$)
  and the only visits to $x$ occur at the beginning and terminal
  vertices of $l$.  We write  $\tilde \loop_x^1 = \tilde \loop_x^1(A)$
  for the set of elementary loops in $A$ 
  with initial vertex $x$. 

\item  For the Markov chain case,
let  $T_x= \min\{j \geq 1: S_j = x\}$ and
 $f_x = \Prob^x\{T_x < \tau\}$.  Then it is standard that
\begin{equation}  \label{firstreturn} G_A(x,x) = \frac{1}{1-f_x} = \sum_{l \in \tilde \loop_x^1} p(l ).
\end{equation}
This formula extends to complex   weights if  
\begin{equation}  \label{firstreturnbound}
  \sum_{\omega \in \tilde \loop_x^1}| q(\omega)| < \infty , 
   \end{equation}
   which is true, say, for integrable weights. 
   Any
$l \in \paths_A(x,x)$ of length at least one
can be written uniquely as
\[         l = \omega_1 \oplus \omega_2  , \;\;\;\;
  \omega_1 \in \tilde \loop_x^1, \;\;\; \omega_2 \in \paths_A(x,x) , \]
and hence
\[      G_A(x,x) = 1 + f_x \, G_A(x,x).\]


        \begad
        
   A little care is needed when $q$ is not
    integrable.  Let $V_j$ denote the set of loops
    in $\tilde \loop_x^1$ of length $j$. If $q$
    is \greenable and satisfies \eqref{firstreturnbound},
    then
    \[    \sum_{j=1}^\infty \sum_{k=1}^\infty
         |q(V_j) \, q_{n-j}(x,x)| < \infty.\]
Therefore,
\begin{eqnarray*}
  G(x,x) = \sum_{n=0}^\infty q_n(x,x) & = & 1 + \sum_{n=1}^\infty 
\sum_{j=1}^n q(V_j) \, q_{n-j}(x,x) \\
& = &    \sum_{j=1}^\infty q(V_j) \, \sum_{n=1}^\infty
        q_n(x,x) \\
        & = & 1 + f_x \, G(x,x).
        \end{eqnarray*}
        \endad
\begad

A number of standard results for which probabilists use stopping
times can be written in terms of products of
generating functions by suitable path splitting.  Such
 arguments are 
  standard in combinatorics and much mathematical physics
literature.  While the
probabilistic form is more intuitive, it is often useful to go
to the generating functions, especially when using nonpositive
weights.  

\endad
  
 \item  The {\em (discrete) Laplacian} is defined by
 \[     \laplace = \laplace_A^q =     (\qprob-I) = - G_A^{-1}. \]
 A function $h$ on $\overline A$ is called {\em ($q$)-harmonic}
 on $A$ if $\laplace h(x) = 0, x \in A$.
 
    \end{itemize}

 \begad
 
 \begin{itemize}
 
 \item  Analysts  often use   $-\laplace = I-\qprob$ which
 is a positive operator for positive $\Prob$.  Indeed, we
 will phrase many of our results in terms of $I-\qprob$.
 
 \item  In the case of random walks on graphs, $\laplace$
  is sometimes
 called the random walk Laplacian.   Combinatorialists often
 use the {\em combinatorial or graph Laplacian}
   which is $-n \laplace$  for
 the
 Laplacian for the Type II random walk on
 graph.
 Note that  this is the diagonal matrix of degrees
 minus the adjacency matrix.
 
 \end{itemize}
 
  \endad

 \begin{itemize}
 \item  The Poisson kernel for a Markov chain
 is defined by
 \[   H_A(x,z) = \Prob^x\{X_{\tau} = z\}, \;\;\;\;
   x \in A, \;\;z \in \partial A.\]
   In this case,
   \begin{equation}  \label{sep27.1}
       \sum_{z \in \p A} H_A(x,z) = p\left[\paths_A(x,
       \partial A)\right] =  1, 
       \end{equation}
We extend the definition   for complex weights $q$ by 
  \[   H_A(x,z) = q\left[\paths_A(x,z) \right].\]
The analogue of the first equality in  \eqref{sep27.1}
holds but it is not necessarily true that
$q[\paths_A(x,\p A)] = 1$.
   
 \item The {\em boundary Poisson kernel} is
 defined by
 \[     H_{\p A}(z,w) = q
 \left[\paths_A(z,w) \right], \;\;\; z,w \in \partial
    A.\]
    
    \begin{exercise}
    Suppose $x \in A$ and $z,w \in \partial A$.
   \begin{enumerate}
   \item  Show that
   \[   H_A(x,z) = \sum_{y \in A} G_A(x,y) \, q(y,z) .\]
   \item  Suppose that $q$ is symmetric and
   $\paths_ A(z,w;x)$ denotes  the set of paths in $\paths_A(z,w)$
  that include the vertex $x$ at least once.  Show that
  \[     q\left[\paths_{A}(z,w;x)\right]
     = \frac{H_A(x,z) \, H_A(x,w)}{G_A(x,x)}. \]
   \end{enumerate}
   \end{exercise}
    
  \begad
  
  The boundary Poisson kernel goes under a number of other
  names and is related to the Dirichlet to Neumann map.

  \endad

\end{itemize}

\section{Loop-erasure}  \label{erasingsec}

A path $\omega = [\omega_0,\ldots,\omega_n]$ is called
a {\em self-avoiding walk (SAW)} if all of the vertices
are distinct.  We will write $\eta  = [\eta_0,\ldots,
\eta_m]$ for SAWs.  We write $\saws_A(x,y)$ for the
set of $\omega \in \paths_A(x,y)$ that are self-avoiding
walks.  We write $\saws_A(x,V)$, etc., as well.

\begad

We will reserve the notation $\eta$ for self-avoiding walks
and use $\omega$ for general walks that can have self-intersections.

\endad

There is a deterministic procedure called {\em (chronological)
loop-erasing} that takes every $\omega \in \paths_A(x,y)$
and outputs a subpath $\eta =LE(\omega)
\in \saws_A(x,y)$.  One
erases the loops in the order that they appear.  This definition
makes this precise.

\begin{definition} 
Suppose that     $\omega=[\omega_0,\ldots,\omega_n]$ 
is a path. The {\em (chronological) loop-erasure} $\eta =
LE(\omega)$ is defined as follows.
\begin{itemize}
\item  Let $j_0 = \max\{k: \omega_k = \omega_0\},$   and
set $\eta_0 = \omega_0 = \omega_{j_0}$.

\item  Recursively, if $j_i < n$, let
$   j_{i+1} = \max\{k: \omega_k = \omega_{j_i+ 1}\},$
and set
\[   \eta_{j+1} = \omega_{j_i + 1} = \omega_{j_{i+1}}.\]

\item  Continue until $j_m = n$ at which point we
set
$      LE(\omega) = \eta = [\eta_0,\eta_1,\ldots,\eta_m].$

\end{itemize}

Note that $LE(\omega)$ is a self-avoiding subpath of $\omega$
 with the same initial and terminal vertices.

\end{definition}

\begad

In general, there are many self-avoiding subpaths of a path
$\omega$ with the same initial and terminal vertices.  The
loop-erasing procedure specifies a particular choice.

Topologists  use the word ``simple'' to mean
with no self-intersections.  Since this conflicts with our
terminology of simple random walk (which does not produce
a path with no self-intersections) we will use the
term ``self-avoiding'' to refer to such paths.  There is a
some possibility of confusion because ``self-avoiding walk''
is also used to refer to a particular measure on SAWs that
is different from   the ones we will consider.

\endad

\begin{exercise}  Give an example of a path $\omega$ such that
\[   LE\left[\omega^R\right] \neq \left[LE(\omega)\right]^R.\]
\end{exercise}

\begin{definition}
Given an integrable weight $q$ on $A$ which gives a measure $q$
on $\overline \paths_A$, the loop-erased measure $\hat q$
is the measure on $\overline \saws_A$ defined by
\[      \hat q(\eta) = \sum_{\omega \in \overline \paths_A,
LE(\omega) = \eta}  q(\omega).\]
\end{definition}

We can also consider the restriction
of  $\hat q$ to  $\saws_A(x,y)$
and note that
\[   \sum_{\eta \in \saws_A(x,y)} \hat q(\eta) = q
 \left[\paths_A(x,y)\right].\]
 The next proposition gives an expression for $\hat q(\eta)$
 in terms of $q(\eta)$ and the Green's function.
 
 \begin{proposition}  \label{sep27.prop1}
If
$    \eta = [\eta_0,\eta_1,\ldots,\eta_m] 
  \in  \saws_A,$ then 
 \[  \hat q(\eta) = q(\eta) \, \prod_{j=0}^m G_{A_j}
    (\eta_j,\eta_j) \]
  where $A_j = A \setminus\{\eta_0,\ldots,\eta_{j-1}\}$.
\end{proposition}

\begin{proof}
Suppose $\omega=[\omega_0,\ldots,\omega_n]$ is such
that $LE(\omega) = \eta.$  Define the indices $j_0,j_1,
\ldots,j_m
$ as in the definition of $LE(\omega)$.  This gives
a unique decomposition  
\[ \omega = l_0 \oplus [\eta_0,\eta_1] \oplus l_1
 \oplus [\eta_1,\eta_2] \oplus \cdots \oplus
 [\eta_{m-1},\eta_{m}] \oplus l_m,\]
where $l_j \in \paths_{A_j}(\eta_j,\eta_j)$.  Conversely,
any choice of $l_j \in K_{A_j}(\eta_j,\eta_j),
j=0,1\ldots,m$ produces an $\omega$ as above with
$LE(\omega) = \eta$.  Since
\[  q(\omega) =  q(\eta) \, q(l_0)\,  q(l_1) 
\dots \, q(l_m) , \]
we get
\[  \hat q(\eta)  =
  q(\eta) \, \left[ \prod_{j=0}^m \sum_{l_j
    \in K_{A_j}(\eta_j,\eta_j)}\;  q(l_j) \right]=
      q(\eta) \, \prod_{j=0}^m
        G_{A_j}(\eta_j,\eta_j).\]

\end{proof}

The case when one or both of the endpoints of $\eta$ is
in $\partial A$ is almost the same except that there is
no loop to be erased at the boundary point. We only
state the proposition. 

\begin{proposition}\label{sep27.prop1.alt}
 Suppose 
 \[     \eta = [\eta_0,\eta_1,\ldots,\eta_m] 
  \in  \overline \saws_A.\]
 Then 
 \[  \hat q(\eta) = q(\eta) \, \prod_{j=0}^m G_{A_j}^*
    (\eta_j,\eta_j) \]
  where $A_j = A \setminus\{\eta_0,\ldots,\eta_{j-1}\}$ and
  \[     G_{A_j}^*
    (\eta_j,\eta_j) = \left\{ \begin{array}{ll}G_{A_j} 
    (\eta_j,\eta_j), & \eta_j \in A, \\
    1, & \eta_j \in \partial A \end{array} \right..\]
\end{proposition}

The quantity
\[ 
  \prod_{j=0}^m G_{A_j} 
    (\eta_j,\eta_j) \]
 appears to depend on the ordering of the vertices
 $\{\eta_0,\ldots,\eta_m\}$.  Actually, as this
 next proposition shows,  it is independent
 of the order.  The proof is easy (once one decides that
 this is true!), and we leave
 it as an exercise. 

\begin{proposition}  \label{prop.sep27}
Let $\{x_1,\ldots,x_n\} \subset A$ and let
\[   F_{x_1,\ldots,x_n}(A) = \prod_{j=1}^n G_{A_j}
(x_j,x_j) , \]
where $A_j = A \setminus \{x_1,\ldots,x_{j-1}\}$. 
Then if $\sigma:\{1,\ldots, n\} \rightarrow \{1,\ldots,
n\}$ is a permutation,
\[    F_{x_{\sigma(1)},\ldots,x_{\sigma(n)}}
  (A) =   F_{x_1,\ldots,x_n}(A).\]
  \end{proposition}

\begin{exercise}  \label{exer.prop.sep27}
Prove Proposition \ref{prop.sep27}.  Hint: first prove it for
$n =2$ and then explain why this implies the general result.
\end{exercise}

Given the proposition, we can make the following definition.

\begin{definition}
If $B = \{x_1,\ldots,x_n\} 
\subset A$, then
\begin{equation}  \label{sep27.3}
      F_B(A) = \prod_{j=1}^n G_{A_j}
(x_j,x_j), \;\;\;\; A_j = A \setminus \{x_1,\ldots,x_{j-1}\}.
\end{equation}
By convention, 
  if $B \not\subset
A$, we define
$    F_B(A) = F_{B \cap A}(A). $
Also,
\[    F(A) = F_A(A). \]
\end{definition}

The proposition implies the rule
\begin{equation}  \label{funion}
   F_{B_1\cup B_2}(A) = F_{B_1}(A)
 \, F_{B_2}(A \setminus B_1).
 \end{equation}
It also  allows us to rewrite
  Propositions \ref{sep27.prop1} and 
\ref{sep27.prop1.alt}  as follows.  

\begin{proposition}  \label{lerwprob}
If 
$    \eta = [\eta_0,\eta_1,\ldots,\eta_m] 
  \in  \overline \saws_A$, then 
\[             \hat q(\eta) = q(\eta) \, F_\eta(A). \]
\end{proposition}

\begad

In the statement of the proposition we have used $\eta$ for
both the path and for the set of vertices in the path.  We will
do this often; hopefully, it will not cause confusion.

\endad

In Proposition \ref{funrooted}, 
we  will give   expression for $F_B(A)$ in terms of the
{\em loop measure} and the invariance under reordering will
be seen from this.   The next proposition gives a formula
for $F(A)$ that is clearly invariant under permutation.

 \begin{proposition}
\begin{equation}  \label{fdet}
 F(A) = \det G_A = \frac{1}{\det (I - \qprob)}. 
 \end{equation}
\end{proposition}  

\begin{proof}
If $A$ has a single element $x$ and $q = q(x,x)$, then
\[  G_A(x,x) = \frac{1}{1-q}, \;\;\; \qprob= [q] , \]
so the result is immediate.  

We now proceed by induction on the number of elements of $A$.
Assume it is true for sets of $n-1$ elements and let
$A = \{x_1,\ldots,x_n\}, A' = \{x_2,\ldots,x_n\}. $ From
the formula we see that
\[   F(A) = G_A(x_1,x_1) \, F(A'). \]
Let $v(y) = G_A( y,x_1)$ which satisfies
\[          v(y) = \delta(x_1-y) + \sum_{z \in A}
      q(y,z) \, v(z) . \]
In other words, the vector $\vec v$  satisfies
    \[   (I- \qprob) \vec v = \delta_{x_1}\]
    where $\delta_{x_1}$ is the vector with $1$ in the
    first component and $0$ elsewhere.  Using Cramer's
    rule to solve this equation we see that
    \[ G_{A}(x_1,x_1) = \frac{\det[M]}{\det(I - \qprob)},\]
  where $M$ is the matrix obtained from $I - \qprob$ by
  changing the first column to  $\delta_{x_1} $. By
  expanding along the first column, we see that
  \[       \det[M] = \det(I -\qprob') , \]
 where $\qprob'$ is $\qprob$ restricted to the entries
 indexed by $A'$.  Therefore, using the inductive hypothesis,
 \[  F(A) =   G_A(x_1,x_1) \, F(A') = \frac{\det[I - \qprob']}{\det[I - \qprob]
 } \, F(A')  =  \frac{1}{\det[I - \qprob]
 }   .\]

\end{proof}

 \begin{exercise}   Consider simple one-dimensional random
walk with $A = [n] := \{1,2,\ldots,n\},$ $ \partial A = \{0,n+1\}$.  Compute
$    F_{[m]}(A)$  for $1\leq m  \leq  n$.
\end{exercise}

\begin{exercise}   \label{exer4} Let $\overline A$ be the complete
 graph on $n$ vertices, and $A \subset \overline A$
 a set with $m < n$ vertices.  Assuming that we are
 doing simple random walk on $\overline A$, compute
 $F(A)$. 
 \end{exercise}

 \section{Loop-erased walk on Markov chains}  \label{markovsec}
 
 For this section, we will only consider loop-erased walks
 arising from Markov chain  transition probabilities
 $p$.  We assume that the
 reader knows basic facts about Markov chains. As before,
 we write 
 \[  \laplace f(x) = \sum_{y \in \state}
        p(x,y) \, [f(y) - f(x)] \]
   for the Laplacian of the chain.
   
 \subsection{Loop-erased walk from transient chains} 
 
 Let $S_n$ be an irreducible
  transient Markov chain on a
 countable set $\state$.  We
 we can define loop-erased random walk as a stochastic
 process by erasing loops from the infinite path. 
 Indeed, if
 \[  \omega = [\omega_0,\omega_1,\ldots]  , \;\;\;
  \omega_j \in \state , \]
  is an infinite sequence of points such that no
  vertex appears an infinite number of times, the
 loop-erasing algorithm  in Section \ref{erasingsec}
  outputs an infinite
 subpath $\eta = LE(\omega)$. 
This probability
measure on infinite self-avoiding paths
can also be viewed as a nonMarkovian process  $\hat S_n$
starting at the same point as the chain $S_n$.

We  will give  another description of the process
by specifying for each SAW $\eta = [\eta_0,\eta_1,
\ldots,\eta_n]$ the probability that the LERW starts
with $\eta$.  
If $A \subset \state$ is a bounded set, let
\[   \phi_A(z) = \Prob^z\{S_n \not\in    A
\mbox{ for all } n \geq 0\}.\]
It is well known that
 $\phi_A$ is the unique nonnegative
 function on 
 $\state$ satisfying
 \[   \phi_A(x) = 0 , \;\;\; x \in A, \]
 \[  \laplace \phi_A(x) = 0 , \;\;\; x \in \state  \setminus
 A, \]
 \[   \sup_{x \in  \state}
  \phi_A(x)= 1   .\]
 We define the {\em escape probability}  $\Es_A(x)$ to be
 \[  \Es_A(x) = \Prob^z\{S_n \not\in    A
\mbox{ for all } n \geq 1\}= \left\{ \begin{array}{ll}
    \phi_A(x) , & x \not \in A. \\
     \laplace \phi_A(x), & x \in A \end{array} \right. .\]

\begin{proposition} \label{sep27.prop2}
If $\eta = [x_0,\ldots,x_n]$
is a self-avoiding walk in $\state$  starting at   $x_0$,
then
\[   \Prob\left\{ [\hat S_0, \ldots \hat S_n] = \eta
     \right\} = p(\eta) \, F_\eta(\state )
      \, \Es_\eta(x_n) ,\]
\begin{equation}  \label{sep27.5}
 \Prob\left\{\hat S_{n+1} = z \mid
     [\hat S_0, \ldots \hat S_n] = \eta
     \right\} = \frac{p(x_n,z) \, \phi_\eta(z)}
            {\sum_{w \in \state} p(x_n,w) \, \phi_\eta(w)}
   . 
            \end{equation}
            \end{proposition}
            
            \begad

   The right-hand side of \eqref{sep27.5} is easily seen
 to be the conditional probability that the Markov chain
 starting at $x_n$ takes its first step to $z$ given
 that it never returns to $A$.
   
   \endad

\begin{proof}
Similarly as in the proof of Proposition
\ref{sep27.prop1}, if $\omega = [\omega_0,\omega_1,\ldots]$
is an infinite path   such that the loop-erasure
$LE(\omega)$ starts with $\eta$, then we can write
$\omega$ uniquely as
\begin{equation}  \label{loopdecomposition}
   \omega = l^0 \oplus [x_0,x_1] \oplus l^1
  \oplus [x_1,x_2] \oplus \cdots \oplus [x_{n-1},x_n]
   \oplus l^n \oplus \omega^+, 
   \end{equation}
   where $l^j$ is a loop rooted at $x_j$
    contained in $\state
   \setminus \{x_0,\ldots,x_{j-1}\}$
 and $\omega^+$ is an infinite path starting at
 $x_n$ that never returns to $\eta$.  In this
 case, $LE(\omega) = \eta \oplus LE(\omega^+)$. 
 The measure of possibilities for $l^0,\ldots,l^n$
 is given by $F_\eta(\state )$ and the measure of
possibilities for $\omega^+$ is $\Es_\eta(x_n)$.  
Given that $\omega^+$ does not return to $\eta$,
the first step of $LE(\omega^+)$ is the same
as the first step of $\omega^+$ and the
conditional probabilities for this step are
given by  \eqref{sep27.5}.
     
   \end{proof}
   
   The process $\hat S_n$ could have been defined using
   the transition probability
     \eqref{sep27.5}.  Since $\phi_\eta$ is the solution
    of the Laplace's equation $\laplace \phi_\eta = 0$,
 the process is sometimes called the {\em Laplacian
 random walk}.  More generally, we can define a process
 called the $b$-Laplacian random walk by
 using the transitions
\[  \Prob\left\{\hat S_{n+1} = x_{n+1} \mid
     [\hat S_0, \ldots \hat S_n] = \eta
     \right\} = \frac{[p(x_n,z)\, \phi_\eta(z)]^b}
            {\sum_{w \in \state}[p(x_n,w)\,  \phi_\eta(w)]^b}, \]
  where we set $0^b = 0$ even if $b \leq 0$.  For $b \neq 1$,
this process is much harder to study and little is known
about it rigorously.  The case $b=0$ is sometimes called the
{\em infinitely growing self-avoiding walk (IGSAW)}.  The IGSAW
chooses randomly among all possible vertices that will not
trap the chain.

\begad

One could also have chosen the numerator to be 
$p(x_n,z) \, \phi_\eta(z)^b$ and similarly
for the denominator.  Neither case is understood.
The only cases that I know where this has been studied,
$p(x_n,z) = p(x_n,w)$ for all $z,w$ for which this
probability is nonzero (such as simple random walk
on a lattice), so the two definitions would be
the same.

\endad
 
 The decomposition of $\omega$ into $LE(\omega)$ and
 the loops in $l_0,l_1,l_2,\ldots$ in Proposition \ref{sep27.prop2}
 extends to the infinite path.  If $x_n = \hat S_n$, then
 the path of the Markov chain is decomposed into
 \[  l^0 \oplus [x_0,x_1] \oplus l^1 \oplus [x_1,x_2]
   \oplus l^2 \oplus \cdots.\]
 As a corollary of the proof, we get
  the conditional distribution of $l_0,l_1,\ldots$
  given its loop-erasure.
 We state the result.  Recall that
 \[  G_A(x,x) = \sum_{l \in \paths_A(x,x)} p(l).\]
 
 \begin{proposition}  \label{inferasedloops}
 Given $\hat S_n=[x_0,x_1,x_2,\ldots]$ the distribution of    $l^0,l^1,\ldots$
  is that of independent
 random variables taking values respectively in $\paths_{\state_j}(x_j,x_j)$.  The random variable $l^j$ has the distribution
\begin{equation}  \label{inferasedloopsdist}
  \Prob\{l^j = l\} =  \frac{p(l)} {G_{\state_j}(x_j,x_j)}
         , \;\;\;\; l \in 
         \paths_{\state_j}(x_j,x_j).
         \end{equation}
 Here $\state_j = \state \setminus \{x_0,\ldots,x_{j-1}\}$.
  \end{proposition}

There is another way to view the distribution on loops in the last proposition.
For fixed $j$, let $S_k = S_k^j$  denote the Markov chain starting
at $x_j$ and let $\tau_j = \inf\{k: S_k \not \in \state _j\}$ (this can equal
infinity) and $\sigma_j = \max\{k < \tau_j : S_k = x_j\}$.  Then
it is easy to check that
the distribution of the loop $[S_0,S_1,\ldots,S_{\sigma_j}]$ is given
by \eqref{inferasedloopsdist}. 
This gives us a method to obtain a path of the Markov chain
by starting with a loop-erased path (or, equivalently, a realization
of the Laplacian walk with transitions as in \eqref{sep27.5})
and adding loops with the appropriate distribution.  We omit
the proof (we have done all the work already).

\begin{proposition}   
Suppose we have  independent Markov chains $\{S_k^x: x \in \state\}$,
each with transition matrix $\Prob$, 
with $S_0^x = x $.  Create a new path as follows.
\begin{itemize}
\item Start with  $\hat S_n = [\hat S_0,\hat S_1,\ldots]$,
  a Laplacian random walk satisfying
\eqref{sep27.5}    starting at $x_0$ independent of $\{S^x\}$.
\item  For each $j$, let $\state _j = \state \setminus \{\hat S_0,\ldots,
\hat S_{j-1}\}$.  Choose a loop at $\hat S_j$ with
distribution \eqref{inferasedloopsdist} by using the method
in previous paragraph.  Note that the loops $\{l^j\}$ are conditionally
independent given $\hat S_n$.

\end{itemize}
Then the path
\[  l^0 \oplus [\hat S_0, \hat S_1] \oplus l^1
 \oplus [\hat S_1,\hat S_2] \oplus l^2 \oplus \cdots \]
has the distribution of the Markov chain starting at $x_0 $.
\end{proposition}

\begad
One thing to emphasize about the last proposition is that the construction
has the following form.
\begin{itemize}
\item We first choose independently  $\hat S$ and the loop-making
Markov chains $\{S^x_j\}$.
\item  The Markov chain $S$ is then
constructed as a {\em deterministic} function of the realizations of these
processes.
\end{itemize}

\endad

 \subsection{Loop-erased walk in a finite set $A$}
 
 \begin{definition}
 Suppose $\overline A = A \cup \partial A$ with 
 $\partial A$ nonempty  and $\Prob$
 is an irreducible Markov chain on $\overline A$. 
   If $x \in A$, then loop-erased
 random walk (LERW) from $x$ to $\partial A$ is
 the   probability
 measure on $\saws_A(x,\partial A)$ 
   obtained by starting
 the chain at $x$, ending at the first time that the
 chain leaves $A$, and erasing loops chronologically.  
 \end{definition}
 
Equivalently, using
Proposition \ref{sep27.prop1.alt}, we see that LERW 
 is the probability measure $\hat p$ on
  $  \saws_A(x,\p A)  $ given by
 \[     \hat p(\eta) = p(\eta) \, F_\eta(A) . \]
 We can also describe this process by
 giving  the transition probabilities for
 the random process $\hat S_n$.  Let $\phi_\eta(y)
 = \phi_{\eta,A}(y)$
 denote the function on $\overline A$ satisfying
 \[            \phi_\eta(y) = 0 , \;\;\; y \in \eta,\]
 \[        \laplace \phi_\eta(y) = 0 , \;\;\;
    y \in A \setminus \eta, \]
  \[   \phi_\eta(y) = 1, \;\;\; y \in \partial A . \]
As before, we let
\[       \Es_{\eta}(y) =\Es_{\eta,A}(y) =   \left\{
\begin{array}{ll}   \laplace \phi_\eta(y), & y  \in \eta\\
   \phi_\eta(y), & y \not \in \eta \end{array}
    \right. . \]
    
   \begin{proposition}
 Suppose $\eta=[\eta_0,\ldots,\eta_k]
  \in \saws_A(x,A)$. Then the probability
 that the LERW from $x$ to $\p A$ starts with $\eta$ is
 \[         p(\eta) \, F_\eta(A) \, \Es_\eta(\eta_k).\]
\end{proposition}

\begin{proof}  Essentially the same as the proof of
Proposition \ref{sep27.prop2}.
\end{proof}

One can define the Laplacian $b$-walk on $A$ similarly
as to the transient case.

\begin{proposition}  \label{sept27.prop4} 
Suppose $S_n$ is an irreducible, transient Markov chain on
a countable state space $\state$ starting at $x \in \state$
  and $A^j$ is an increasing sequence of subsets
of $\state$ containing $x$, such that
\[     \state = \bigcup_{j=0}^\infty A^j.\]
Let $\eta$ be a (finite) SAW  in $\state$ starting at $x$
 and let $\hat p(\eta)$ and $\hat p^j(\eta)$ denote
the probability that LERW starting at $x$ to infinity 
and $\partial A^j$, respectively, start  with $\eta$.  Then,
\[   \hat  p(\eta) = \lim_{j \rightarrow \infty }\hat  p^j(\eta).\]
\end{proposition}

\begin{proof}  For fixed $\eta$, we need only show that
\[   \lim_{n \rightarrow \infty}  F_\eta(A^n) = F_\eta 
(\state) , \]
\[    \lim_{n \rightarrow \infty}  \Es_{\eta,A^n}(\eta_n) =  
\Es_{\eta,\state}(\eta_n), \]
both of which are easily verified (Exercise \ref{sep28.exer1}).
\end{proof}
%
%

We state the analogue of Proposition \ref{inferasedloops} which is
proved in the same way.  

\begin{proposition}  \label{erasedloops}
 Given $\eta $ the distribution of    $l^0,l^1,\ldots,l^{k-1}$
  is that of independent
 random variables taking values respectively in $\paths_{A_j}(x_j,x_j)$.  The random variable $l^j$ has the distribution
\begin{equation} \label{erasedloopsdist}
   \Prob\{l^j = l\} =  \frac{p(l)} {G_{A_j}(x_j,x_j)}
         , \;\;\;\; l \in 
         \paths_{A_j}(x_j,x_j).
         \end{equation}
 Here $A_j = A \setminus \{x_0,\ldots,x_{j-1}\}$.
  \end{proposition}

It is often  useful to consider LERW from a boundary point to a
subset of the boundary.  Suppose $A, \partial A$ are given and
$z \in \p A, V \subset \p A \setminus \{z\}$.  Then loop-erased
random walk from $z$ to $V$ in $A$ is the measure on paths
of total mass 
\[            H_{\p A}(z,V) := \sum_{w \in V}
   H_{\p A}(z,w) , \]
obtained from the measure $\hat p$ restricted to $\saws_A(z,V)$.
This also gives a probability measure on paths when we normalize
so with total mass one.  Let us consider this probability measure.
Note that if $z \in A$, then LERW from $z$ to $V$ is the same
as if we make $z$ a boundary point.  An important property that
the probability measure satisfies is the following:
\begin{itemize}
\item {\bf Domain Markov Property.}  Suppose $z \in \p A $,
$V \subset \p A \setminus \{z\}$.  Then the probability measure of
loop-erased random walk from $z$ to $V$ satisfies the 
following
{\em domain Markov property}: conditioned that the
path starts as $\eta = [\eta_0 = z,\ldots,\eta_k]$, the
remainder of the walk has the distribution of LERW from
$\eta_k $ to $V$ in $A \setminus \eta$.
\end{itemize}

\begad

There is a slight confusion in terminology that the reader must live with.
When referring to loop-erased random walk say from $z$ to $w$ in $A$ where
$z,w \in \p A$, one sometimes is referring to the measure on paths of
total mass $H_{\p A}(z,w)$ and sometimes to the probability measure obtained
from normalizing to total mass one.  Both concepts are very important and
the ability to go back and forth between the two ideas is fruitful in
analysis.

\endad

\begin{exercise}$\;$
\begin{enumerate}
\item  Verify the domain Markov property.
\item  Extend it to the the following ``two-sided'' domain
Markov property.  Take LERW from $z$ to $V$ in $A$ and condition
on the event that the beginning of the path is
$\eta = [\eta_0,\ldots,\eta_k]$; the  end of the path
is $\eta' = [\eta_0',\ldots,\eta_j']$ where $\eta_j' \in V$.
Assume that $\eta \cap \eta' = \eset$ and that $p(\eta_k,\eta_0')
 = 0$.  Show that the conditional
 distribution of the remainder of the path is the same
 as LERW from $\eta_k$ to $\eta_0'$ in $A \setminus (\eta \cup \eta')$.
 \end{enumerate}
 \end{exercise}

\begin{exercise} \label{sep28.exer1}  The following is used in the proof
of Proposition \ref{sept27.prop4}.  Suppose $X_n$ is
an irreducible, transient Markov chain on a countable
state space $\state$; $A_n$ is an increasing sequence
of finite sets containing $x$ whose union is $\state$.
Then  for all $x \in \state$, 
\[        \lim_{n \rightarrow \infty}
            G_{A_n}(x,x) = G(x,x) , \]
 \[   \lim_{n \rightarrow \infty} \Es_{A_0,A_n}(x)
   = \Es_{A_0}(x).\]

 \end{exercise}

 \begin{exercise}  Using the notation of Proposition
 \ref{erasedloops},
suppose $S$ is a random walk starting at $x_j$ and
 let $\tau = \min\{k: S_k \not \in A_j\}, \sigma = \max\{k
 < \tau:
   S_j = x_j\}$.  Show that the distribution of the loop
 $    [S_0,S_1,\ldots,S_\sigma]  $
   is the same as \eqref{erasedloopsdist}.
 \end{exercise}
 
 \subsection{Infinite LERW for recurrent Markov chains}  \label{infinitelerwsec}
 
 If $X_n$ is an irreducible, recurrent Markov chain on a countably
 infinite state space $\state$, then one cannot define LERW on 
 $\state$ by erasing loops from the infinite path.  However, 
 if one can prove a certain property of the chain, then one can
give a good definition.  This property will hold
for  two-dimensional simple random walk.
 
Let $x_0 \in \state$.
Suppose $A_n$ is an increasing sequence of finite subsets of
$\state$ with $x_0 \in A_0$ and whose union is $\state$.  
Let $\eta = [\eta_0 =x_0,\ldots,\eta_k]$ be a SAW in $\state$
starting at $x_0$.  In order to specify the distribution
of the infinite LERW it suffices to give the probability of
producing $\eta$ for each $\eta$. Using
   the previous section, we see that we would
 like to define this to be
 \[  \hat p(\eta) =  \lim_{n \rightarrow \infty} p(\eta) \, F_\eta(A_n)
  \, \Es_{A_n}(\eta_n). \]
Since $F_\eta(A_n) =  G_{A_n}(x_0,x_0)
 \, F_\eta (A_n \setminus \{0\})
  \sim G_{A_n}(x_0,x_0) \, F_{\eta}(\state
 \setminus \{x_0\}),$ we can see that
 \[   \hat p(\eta) = p(\eta) \, 
  F_{\eta}(\state
 \setminus \{x_0\}) \, \left[\lim_{n \rightarrow \infty}
      G_{A_n}(x_0,x_0) \, \Es_{\eta,A_n}(\eta_n) \right], \]
 assuming that the limit on the right-hand side exists.

 \begin{itemize}
 
 \item {\bf Property A.}  For every finite $  V \subset \state$
 and $y \in V$,
 there exists a nonnegative function $\phi_{V,y}$ that vanishes
 on $V$ and is harmonic (that is, $\laplace \phi_{V,y} = 0$)
 on $\state \setminus V$ satisfying the following.  Suppose
 $A_n$ in an increasing sequence of subsets of $\state$
 whose union is $\state$. Let $\phi_n$ be the function
 that is harmonic on $A_n \setminus V$; vanishes on $V$;
 and takes value $1$ on $\state \setminus (A_n \cup V)$.  
  Then for all $x$,
\begin{equation}  \label{sep28.1}
       \lim_{n \rightarrow \infty}  G_{A_n}(y,y)
    \, \phi_{n}(x) = \phi_{V,y}(x)  . 
  \end{equation}
  In particular, if $x \in V$,   
     \[        \lim_{n \rightarrow \infty}  G_{A_n}(y,y)
    \, \Es_{V,A_n}(x) =   \laplace \phi_{V,y}(x)  . \]
    \end{itemize}
    
\begin{definition}  If a recurrent irreducible Markov chain 
satisfies Property A, then we define the infinite LERW starting
at $x_0$ by
\[    \Prob\left\{[\hat S_0,\ldots,\hat S_n]
 = \eta \right\}    =   p(\eta) \,  [\laplace \phi(\eta_n)]
    \, F_\eta( \state \setminus \{x_0\}), \]
   where $\phi = \phi_{\eta,x_0}$ and $\eta = [\eta_0 = x_0,\eta_1,
   \ldots,\eta_n].$
\end{definition}

We will show later the nontrivial fact that two-dimensional
random walk satisfies Property A.   However, this property
is not satisfied by all recurrent chains as can be seen from
the next example.

\begin{exercise}$\;$

\begin{enumerate}
\item  Show that one-dimensional simple random walk does
not satisfy Property A.
\item  Show that if one takes $y=0$, $V = \{0\}$,  and
 $A_n = \{1-n,2-n,\ldots,n-2,n-1\}$,
then the limit in \eqref{sep28.1} exists and give
the limit.
\item  Do the same with $A_n = \{1-n,2-n,\ldots,2n-2,
2n-1\}$ and show that the limit exists but is different.
\end{enumerate}
\end{exercise}

One could also define infinite LERW with respect to a particular
sequence $\{A_n\}$ provided that the appropriate limit
exists for this sequence.

\begin{exercise}
Suppose  that $X_n$ is an irreducible, recurrent
Markov chain on a countably infinite state space $\state$, and
  $A_n$ is  an increasing sequence
of finite subsets of $\state$ whose union is $\state$.  Show
that if $x,y \in \state$,
\[   \lim_{n \rightarrow \infty} \frac{G_{A_n}(x,x)}
 {G_{A_n}(y,y)} = \frac{G_{\state \setminus \{y\}}
    (x,x)}{G_{\state \setminus \{x\}}(y,y)}.\]
    Is the recurrence assumption needed?
  \end{exercise}

\begin{exercise}  Assume that $X_n$ is an irreducible, recurrent
Markov chain on a countably infinite state space $\state$
that satisfies Property A.  Let $A_n$ be an increasing sequence
of finite subsets of $\state$ whose union is $\state$
and $V$ a finite subset of $\state$.
\begin{enumerate}
  \item  Show that there exists a single function
  $\phi_V$  and a
  positive function $c(\cdot)$ on $V$
such that  for $x \in  V$,
    $\phi_{V,x} = c(x)
 \, \phi_V$.
 \item  Show that the process is a Laplacian random
 walk in the sense that
 \[  \Prob\left\{\hat S_{n+1} = z \mid
     [\hat S_0, \ldots \hat S_n] = \eta
     \right\} = \frac{\phi_\eta(z)\, p(\eta_n,z) }
            {\sum_{|w-x_n| = 1} \phi_\eta(w)\, p(\eta_n,w)}\]
            
\item Assume as given that two-dimensional simple random
walk satisfies Property A.  Show that  
$\phi_{V,x} = \phi_{V,y}$ for all $x,y \in \Z^2$.

\end{enumerate}

\end{exercise}

\subsubsection{Random walk in $\Z^2$}

Here we will show that two-dimensional simple random walk
 satisfies Property A using known
facts about the random walk.  We will establish
this property with $y = 0$ (the other cases are done similarly)
and write $\phi_{V,y} = \phi_V$.  Let
$C_m = \{z \in \Z^2: |z| < m\}$  and if $S_j$ denotes a simple
random walk,
\[     \sigma_m = \min\{j \geq 0: S_j \not\in C_m\}. \]

The {\em potential kernel} (see \cite[Section 4.4]{LLimic})
is defined by
\[      a(x) = \lim_{n \rightarrow \infty}
 \sum_{j=0}^n  \left[\Prob\{S_j = 0\}
    - \Prob\{S_j = x\} \right].\]
 This limit exists, is nonnegative,  and satisfies
 \[    a(0) = 0 ,   \]
 \[  \laplace a(x) = \left\{\begin{array}
 {ll}  1 , & x = 0 \\
  0 , & x \neq 0 \end{array} \right., \]
 \[   a(x) = \frac 2 \pi \, \log |x| + c_0 + O(|x|^{-2}),
 \;\;\;\; x \rightarrow \infty.\]
where $c_0 = (2\gamma+\log 8)/\pi$  and $\gamma$
is Euler's constant.   We  set 
\[      \phi_V(x) = a(x) - \E^x\left[a(S_{\zeta}) \right], \]
where $\zeta = \zeta_V = \min\{j \geq 0: S_j \in V\}.$
It is known  
\cite[Proposition 6.4.7]{LLimic} that
\[ \phi_V(x)= \lim_{m \rightarrow \infty} G_{C_m}(0,0) \,
  \Prob^x\{S[0,\sigma_m] \cap V = \eset\}. \]

Let $A_n$ be an increasing sequence of finite subsets of $\Z^2$
containing $V$ 
whose union is $\Z^2$.  
Let $\tau_n = \min\{j: S_j \not \in A_n\}$; we need
to show that
\[   \phi_V(x)= \lim_{n \rightarrow \infty} G_{A_n}(0,0) \,
  \Prob^x\{S[0,\tau_n] \cap V = \eset\}, \]
   Let $m = m_n $ be 
  the largest integer with 
  $ C_{ m} 
\subset A_n $.

Let us first consider the case $V = \{0\}$.  Let 
$T = \inf\{j > 0: S_j = 0\}$.  As usual for Markov chains,
we have
\begin{eqnarray*}
  G_{A_n}(0,0)^{-1}  & = & \Prob^0\{T  > \tau_n\} \\
  & = & \Prob^0\{T > \sigma_{m}\} \, \Prob^0\{T > \tau_n
    \mid T > \sigma_{m} \}\\
    &= &  G_{C_m}(0,0)^{-1}  \,   \Prob^0\{T > \tau_n
    \mid T > \sigma_{m} \}
    \end{eqnarray*}
Let $h_m(z) = H_{C_{m}}(0,z)$ denote the hitting probability
of $\partial C_{m}$ by a random walk starting at the origin.
Using a last exit decomposition, we can see that
 $
   \Prob^0\{ S_{\sigma_m} = z \mid \sigma_{m } < T \}  = h_m(z).$
Therefore,
\[      \Prob^0\{T > \tau_n
    \mid T > \sigma_m \} = \sum_{z \in \p C_{m}}
        h_m(z) \,  \Prob^z\{T > \tau_n\} . \]
        Using \cite[Proposition 6.4.5]{LLimic},
 we see that for $|x| < m/2$,
 \[     \Prob^x\{
 S_{\sigma_m} = z \mid  \sigma_m< T \} =
  h_n(z) \, \left[1 + O\left(\frac{|x|}{m} \,
     \log \frac{m}{|x|}\right) \right], \]
 from which we conclude that
 \[ \Prob^x\{\tau_n < T\}
=  \Prob^x\{\sigma_{m} < T\} \, 
\Prob^0\{T > \tau_n
    \mid T > \sigma_{m} \} \, 
    \left[1 + O\left(\frac{|x|}{m} \,
     \log \frac{m}{|x|}\right) \right],\]
 and 
 \[  \lim_{n \rightarrow \infty}
   G_{A_n}(0,0) \, \Prob^x\{\tau_n < T\}
      =  \lim_{m \rightarrow \infty}
         G_{C_{m}}(0,0)
          \, \Prob^x\{\sigma_{m } < T\} = a(x)\]

 For more general $V$, let $\zeta = \zeta_V
 $ and let $\phi_n$ denote the function that is
  harmonic on $A_n \setminus V$ with boundary value $0$
  on $V$ and $1$ on $\Z^2 \setminus A_n$.
  Let $\psi_n$ be the corresponding function
  with $V=\{0\}$.
  Note that  
  \[  \phi_n(x) = \psi_n(x) - 
  \sum_{z \in V}
      \Prob^{x}\{S_{\zeta \wedge \tau_n} = z\}   \, \psi_n(z). \]
 Therefore,
\begin{eqnarray*}
 \lefteqn{\lim_{n \rightarrow \infty} G_{A_n}(0,0)
    \phi_n(x)}\\ & = &  \lim_{n \rightarrow \infty}\left[ 
G_{A_n}(0,0)\,  \psi_n(x) - 
  \sum_{z \in V}
      \Prob^{x}\{S_{\zeta \wedge \tau_n} = z\}   \,G_{A_n}(0,0)
      \,  \psi_n(z)\right]\\
    & = & a(x) -  \sum_{z \in V}
      \Prob^{x}\{S_{\zeta } = z\}  \,a(z) =  \phi_V(x).
    \end{eqnarray*}

\subsection{Wilson's algorithm}  \label{wilsonsec}

Suppose $\Prob$ is the transition matrix
of an   irreducible Markov chains on a finite state space
$\overline A = \{x_0,x_1,\ldots,x_n\}$
and let $A = \{x_1,\ldots,x_n\}$.   A {\em spanning tree} $\tree$
of 
(the complete graph) of $\overline A$ is a collection of $n$ (undirected)
edges such that $\overline A$ with those edges is a connected graph.
This implies that every point is connected to every other point
(this is what makes it {\em spanning}), 
and since there are only $n$ edges that there are no ``loops''
(this is what makes it a {\em tree}). 
Given a spanning tree $\tree$,
for each $x \in A$, there is a unique SAW $\eta \in \saws_A(x,x_0)$
whose edges lie in $\tree$. This gives us a directed graph (that we also
label as $\tree$ although it depends on the choice of
``root'' $x_0$) by orienting each edge towards the root.

The weight of $\tree$ (with respect to $x_0)$ is given by
\[           p(\tree; x_0) = \prod_{j=1}^n p(\edge_j), \]
where the product is over the directed edges in the tree.
We will now describe an algorithm to choose a spanning tree
with a fixed root $x_0$.

\begin{definition}  Given $A,\Prob,x_0$, {\em Wilson's
algorithm} to select a spanning tree is as follows.
\begin{itemize}
\item  
Take a LERW in $A $ starting at $x_1$ to $\p A  = \{x_0\}$.
   Include all the edges traversed
by the walk in the tree and let $A_2$ be the set of vertices
that have not been connected to the tree yet.  

\end{itemize}

and recursively,
\begin{itemize}

\item  If $A_k = \eset$, then we have a tree and stop.
\item Otherwise, let $j$ be the smallest index such that
$x_j \not \in A_k$.  Take a LERW   in $A_k$ from
$x_j$ to $A \setminus A_k$.  Add those edges to the tree
and let $A_{k+1}$ be the set of vertices that have not
been connected to the tree. 

\end{itemize}

\end{definition}

\begin{proposition}   \label{kirchhoff}
Given $\{x_0\}, \Prob$,
the probability that a particular spanning
tree $\tree$ is chosen in Wilson's algorithm is
$            p(\tree;x_0) \, F(A).$
In particular, 
\begin{equation}  \label{dec28.1}
 \sum_{\tree} p(\tree;x_0) = \frac{1}{F(A)}
 = \det(I-\Prob). 
 \end{equation}
 \end{proposition}
 
 \begin{proof}
 Given any $\tree$ we can decompose it in a unique way as
 follows.
 \begin{itemize}
 \item Let $\eta^1$ be the path in $\tree$ from $x_1$ to $x_0$.
 \item  Given $\eta^1, \ldots  ,\eta^k$, let $x_j$ be
 the vertex of smallest index (if any) that is not included
 in $\eta^1 \cup \cdots \cup \eta^k$.  Let $\eta^{k+1}$
 be the unique
  path from $x^j$ to $\eta^1 \cup \cdots\cup \eta^k$.
  (If there were more than one path, then there would be
  a loop in the tree.)
  \end{itemize}
Given this decomposition of $\tree$ into $\eta^1,\ldots,\eta^k$
we can see from repeated application of Proposition
\ref{lerwprob} that the probability of choosing $\tree$ is
\[  p(\eta^1) \, F_{\eta^1}(A) \, p(\eta^2)
 \, F_{\eta^2}(A \setminus \eta^1) \cdots p(\eta^k)
  \, F_{\eta^k}(A \setminus (\eta^1 \cup   \cdots  \cup  \eta^{k-1})).\]
 but $p(\eta_1) \cdots p(\eta_k) = p(\tree;x_0)$ and 
 \eqref{funion} shows that
 \[  F(A) =  F_{\eta^1}(A)\,  F_{\eta^2}(A \setminus \eta^1)\,
 \cdots  F_{\eta^k}(A \setminus (\eta^1 \cup  \cdots  \cup  \eta^{k-1})).\]
 The second equality in \eqref{dec28.1}
 follows from \eqref{fdet}.
 \end{proof}

 A particularly interesting case of this result is   random
 walk on a graph.  Suppose $(G,\edges)$ is a simple, connected
 graph with vertices $\{x_0,x_1,\ldots,x_n\}$.
  Let us do
 the ``Type II'' version of random walk on the graph.  Then
 for each spanning tree $\tree$ of $G$ we have
 \[     p(\tree;x_0) = n^{-n} .\]
 In particular each tree is chosen with equal probability and
 this probability is
 \[              n^{-n} \, F(A) =   \frac1{n^n \, \det(I - \Prob)}
  = \frac{1}{\det[n(I-\Prob)]}.\]
  Recall that $n(I-\Prob)$ is the graph Laplacian.  We have proved
  an old result due to Kirchhoff sometimes called the matrix-tree
  theorem.
  
  \begin{corollary}[Kirchhoff] 
  The number of spanning trees of a graph is given by the determinant
  of the graph Laplacian.
  \end{corollary}

  \begin{exercise}  Explain why doing Wilson's algorithm with a ``Type I''
  simple random walk on a graph generates the same distribution
  (that is, the uniform distribution)
  on the set of spanning trees.
  \end{exercise}

  \begin{exercise}  Use Proposition \ref{kirchhoff} and Exercise 
  \ref{exer4} to compute the number of spanning trees in a complete
  graph.
  \end{exercise}
  
  We will generalize this a little bit.  Given $\overline A = A \cup
  \partial A$, we define the graph obtained by {\em wiring} the boundary
  $\partial A$ to be  the graph with vertex set $A \cup \{\partial A\}$
  (that is,  considering $\partial A$
  as a single vertex) and retaining all the edges 
  with at least one vertex in $A$.     This
  gives some multiple edges between vertices in $A$  and $\partial A$, but
  we retain the multiple edges.  A {\em wired spanning tree} for $A$
  is a spanning tree for the wired graph.  Wilson's algorithm gives
  a method for sampling from the uniform distribution on wired
  spanning trees.  We can describe the algorithm recursively
  as follows.
  \begin{itemize}
  \item  Choose any vertex  $x \in A$ and let $\eta$ be LERW
  from $x$ to $\partial A$.  Let $A' = A \setminus \eta$.
  \item  If $A' \neq \eset$, choose a uniform wired spanning
  tree from $A'$.
  \end{itemize}
  What we emphasize here is that we can choose any vertex at which
  to start the algorithm, and after adding a SAW $\eta$ to the
  tree, we can choose any remaining vertex at which to continue.
  If we take the uniform wired spanning tree and restrict to
  the edges in $A$, then we call the resulting object the
  {\em uniform spanning forest} on $A$.  The terminology is
  perhaps not best because this is not the same thing as looking
  at all spanning forests of $A$ and choosing one at random.
  Since we will use this terminology we define it here.
  
  \begin{definition}
  Suppose $(\overline A,E)$ is a connected graph and $A$
  is a strict subset of $\overline A$.    
  \begin{itemize}  
  \item  The {\em uniform
  wired spanning tree} on $A$ is a spanning tree of
  $A \cup \{\partial A\}$ chosen uniformly over all spanning
  trees of the wired graph.
  \item The {\em uniform spanning forest} of $A$ is the uniform wired
  spanning tree of $A$ restricted to the edges for which both 
  endpoints are in $A$,
  \end{itemize}
  \end{definition}
  
  Wilson's algorithm applied to the simple random walk on the
  graph generates a uniform wired spanning tree and hence
  a uniform spanning forest.

  \subsection{Uniform spanning tree/forest in $\Z^d$}  \label{spansec}
  
  The uniform spanning tree in $\Z^d$  is 
  the limit as $n \rightarrow \infty$ of the uniform
  spanning forest
  on the discrete ball $C_n = \{x \in \Z^d:
  |x| < n\}$.  If $d=1$, the uniform spanning forest of $C_n$
  is all of $C_n$, so we will consider only
  $d\geq 2$.  We will use Wilson's algorithm
  to give  a different definition
  for the forest, but then we will prove it is also the
  limit of the uniform spanning forests on $C_n$.  The construction
  will yield a spanning tree of $\Z^d$ if $d=2,3,4$, but
  will only be a forest for $d \geq 5$.   The
  difference between $d \leq 4$ and $d \geq 5$ comes
  for a property about loop-erased random walk that
  we now discuss.  If $S_j$ is a simple
  random walk we write   
  \[  S[0,n] = \{S_j: 0 \leq j \leq n \}.\]

  \begin{proposition}  \label{oct5.prop1}
  If $S^1,S^2$ are independent simple
  random walks in $\Z^d$ starting at the origin,
    then
    \[  \Prob\left\{
  S^1[0,\infty) \cap S^2[0,\infty)
\mbox{ is infinite}\right\} =  \left\{ \begin{array}{ll}  1 , & d \leq 4 \\
    0, & d \geq 5 \end{array}\right. ,\]
   \[  \Prob\left\{
 S^1[0,\infty) \cap S^2[1,\infty)=
  \eset \right\}   \; \left\{ \begin{array}{ll}  = 0  , & d \leq 4 \\
    >0, & d \geq 5 \end{array}\right.  .\]
  \end{proposition}
   
   \begin{exercise}   Prove Proposition \ref{oct5.prop1}.
  You may want to  consider  first  the expectation of
 \[\#\left[S^1[0,\infty) \cap S^2[0,\infty)
  \right].\]
  \end{exercise}

  A little harder to prove is the following.
  
  \begin{proposition} \label{oct5.3}
   If $S^1,S^2$ are independent simple
  random walks starting at the origin then
  \[   \Prob\{\hat S^1[0,\infty) \cap S^2[1,\infty)
   = \eset \}
    \; \left\{ \begin{array}{ll}  = 0  , & d \leq 4 \\
    >0, & d \geq 5 \end{array}\right.  .\]
    \end{proposition}

  \begad
  
  We will not prove this.  The critical dimension is $d=4$.  
  The probability that two simple random walks in
  $\Z^4$ starting at neighboring points go distance
  $R$ without intersecting is comparable to $(\log R)^{-1/2}$.
  The probability that one of the walks does not
  intersect the loop-erasure of the other is comparable
  to $(\log R)^{-1/3}$. 
  
  \endad
    
   Using this proposition, we will now define the spanning
   tree/forest in the three cases.  In each case we will use
  what we will call the   {\em infinite 
  Wilson's algorithm}.  We assume that we start with
   an enumeration of $\Z^d=\{ x_1,x_2,\ldots,\}$ and  
 we have   independent simple random walks $S^j_n$ starting
   at $x_j$.   The algorithm as we state it will depend
   on the particular enumeration of the lattice, but it will follows
   from Theorem \ref{infspan} below that the distribution
   of the object is independent of the ordering.

  \subsubsection{Uniform spanning tree for $d=3,4$}
  \begin{itemize} 
  \item  Start by taking $S^1$ and erasing loops to get $\hat S^1[0,\infty)$.
  Include all the edges and vertices of $\hat S^1[0,\infty)$ in the tree.
  We call this tree (which is not spanning) $\tilde \tree_1$.
  \item  Recursively, if $x_j \in \tree_{j-1}$, we set
  $\tree_{j} = \tree_{j-1}$.
  Otherwise, consider the random walk $S^j$ and
  stop it at the first time $T^j$
   it reaches a vertex in $\tilde \tree_{j-1}$.  By Proposition
  \ref{oct5.3}, this happens with probability one. Take
  the loop-erasure
  $LE(S^j[0,T^j])$ and add those edges and vertices to the tree
  to form $\tilde \tree_j$.
  \end{itemize}
  
  This algorithm does not stop in finite time,  but it gives a spanning
  tree of the infinite lattice $\Z^d$.   To be more precise,
  suppose that $\overline C_m \subset \{x_1,\ldots,x_k\}$.
  Then every vertex in $\overline C_m$ in included in $\tilde \tree_k$,
  and it is impossible to add any more edges adjacent to a vertex
  in $C_m$.  Hence for all $n \geq k $,
  \[     \tilde \tree_k \cap C_m = \tilde \tree_n \cap C_m ,\]
  and hence we can set $\tree \cap C_m = \tilde \tree_k \cap C_m$.  Here we are writing
  $\tree \cap C_m$ for the set of edges in $\tree$ that have both
  vertices in $C_m$.

  This distribution on spanning trees
  is called the {\em uniform spanning tree} on $\Z^d, d=3,4$.
  
 \subsubsection{Uniform spanning tree for $d=2$}
 The uniform spanning tree for $d=2$ is defined  similarly.  The only
 difference is that in the first step, one takes the infinite
 LERW starting at $x_1$ as discussed in Section
 \ref{infinitelerwsec} and uses those edges to form $\tilde \tree_1$.
 The remaining construction is the same.
 
 \subsubsection{Uniform spanning forest for $d \geq 5$}
 The construction will be similarly to $d=3,4$ except that
 the $\tilde \tree_k$ will only be forests, that is, they
 will not necessarily be connected.
 \begin{itemize} 
  \item  Start by taking $S^1$ and erasing loops to get $\hat S^1[0,\infty)$.
  Include all the edges and vertices of $\hat S^1[0,\infty)$ in 
  the forest $\tilde \tree_1$.
  \item  Recursively, if $x_j$ is a vertex in the forest $\tilde \tree_{j-1}$,
  then we set $\tilde \tree_j = \tilde \tree_{j-1}$.
Otherwise, consider the random walk $S^j$ and
  stop it at the first time $T$
   it reaches a vertex in $\tilde \tree_j$.  It is possible that $T = \infty$.
  Erase loops
  from $S^j[0,T]$ and add those edges and vertices to the tree.
  If $T <\infty$,  this adds edges to one of the components
  of $\tilde \tree_{j-1}$. 
  If $T = \infty$, this adds the complete
  loop-erasure $\hat S^j[0,\infty)$  and hence gives
  a new connected component to the forest.
  \end{itemize}
The output of this algorithm is an infinite spanning forest $\forest$
  with an infinite number of components.

  \begin{exercise}  Show that the uniform spanning forest for
  $d \geq 5$ has an infinite number of components.
  \end{exercise}
  
 This was not the original definition of the uniform spanning tree/forest.
Rather, it was described as a limit of trees on finite subsets of $\Z^d$.
Let $C_n = \{x \in \Z^d: |x| < n\}$ and consider the {\em uniform spanning
forest} on $C_n$.  To be precise, we construct the uniform spanning tree
$\tree_n$  on
the wired graph  $C_n \cup \{\p C_n\}$ and let $\forest_n$ be
the forest in $C_n$ obtained by taking only  the edges
in $C_n$. 

For every finite set $A$, we write $A \cap \forest_n$
for  the set of edges in $\forest_n$ with
both vertices in $A$.    This gives a probability measure $\nu_{A,n}$ on 
forests in $A$.  We can also consider the probability measure
$\nu_{A}$ obtained from intersecting the infinite spanning tree with $A$.

\begin{theorem}  \label{infspan}
  If $\forest$ denotes the uniform spanning
forest (or tree) in $\Z^d$, then we can couple $\forest$ 
and $\{\forest_n: n \geq 1\}$ on the same probability space,
such that with probability one for each finite set $A$,
for all $n$ sufficiently large
\[         
    \forest_n \cap A = \forest \cap A . \]
  \end{theorem}
   
%

 \begin{proof}  If suffices to prove this for $A = C_m$,
  and we write $\tree_{n,m},\tree_{\infty,m}$
 for $\forest_{n} \cap C_m, \forest \cap C_m$,
 respectively.  We will do the $d\geq 3$ case leaving
 the $d=2$ case as an exercise.    Assume $d \geq 3$  
and choose any ordering of $\Z^d = \{x_1,x_2,\ldots\}$.  We 
assume we have a probability space on which are defined
 independent simple random walks $S^j_i$ starting at $x_j$. Given
 these random walks  the  spanning forest $\tree$ is output
using Wilson's algorithm above (it is a forest for $d \geq 5$
and a  tree for $d=3,4$, but we can use a single notation).
   For each $n$, we  construct
the uniform spanning forest on $C_n$ on the same probability
space, using Wilson's algorithm with
the same random walks and the same ordering of the points.
The only difference is that the random walks are stopped upon
reaching $\p C_n$. We recall that the distribution of this
forest  is independent of the ordering.  
If $m < n$, we will write $\forest_{\infty,m},
\forest_{n,m}$ for the forests  
restricted to $C_m$.   

We fix $m$.
Given the realization of the random walks $S^j$ we find
$N$ as follows. We write $\tilde \tree_k$ for the (non-spanning)
forest obtained from the infinite Wilson's algorithm
stopped once all the vertices $\{x_1,\ldots,x_k\}$
have been added to $\tree_k$. We write $\tilde \tree_{k,n}$
for the analogous forest for the walks stopped at
$\partial C_n$.
\begin{itemize}
\item 
Choose $k$ sufficiently large so that 
 $C_{m+1}\subset\{x_1,\ldots,x_k\}$.
In particular, every vertex in $C_{m+1}$
 has been added to $\tilde \tree_k$.   
 We partition $\{x_1,\ldots,x_k\}$
 as $V_1 \cup V_2$ where $V_1$ are the points
 $x_j$ such that $S^j[0,\infty)
  \cap \tree_{j-1} \neq \eset$.
    By definition, $x_1 \in V_2$;
   if $d=3,4$, then $V_2 = \{x_1\}$,
  but $V_2$ can be larger   for $d \geq 5$.
\item  Choose
$n_1$ sufficiently large so that for each $x_j
 \in V_1$,
The path $S^j$ hits $\tilde \tree_{j-1}$ before reaching
$\partial C_{n_1}$.  
\item  Choose $n_1 < n_2 <N $ such that each for each $j \in V_2$, the
random walk $S^j$ never returns to $C_{n_1}$ after reaching
$\partial C_{n_2}$ and never returns to $C_{n_2}$ after
reaching $\partial C_{N}$.  Note that this implies that
for every $n \geq N$, the intersection of
$C_{n_1}$  and  the loop-erasure of $S^j$ stopped
when it reaches $\p C_n$  is the same as 
intersection of
$C_{n_1}$  and  
  the loop-erasure of the infinite path.
%
%
\end{itemize}
Then one readily checks that for $n \geq N$,
$\tree_{n,m} = \tree_{\infty,m}$.

%
%
%
%
%
%
\end{proof}

\begad

While this proof was not very difficult, it should be
pointed out that no estimates were given
for the rate of convergence.  Indeed, the numbers $n_1,n_2,N$
in the proof can be very large.

\endad

Given the infinite spanning tree or forest, we can also 
consider the intersection of this with a discrete ball $C_n$.
This gives a forest in $C_n$.  For $d=2,3$, the largest
components of this forest have on order $n^{d}$ points.
However, for $d=4$, there are of order $\log n$ components
of order $n^4/\log n $ points.  In other words, the uniform
spanning tree in $\Z^4$ does not look like a tree locally.

\begin{exercise}  Prove  Theorem \ref{infspan}
for $d=2$.
\end{exercise}

  \section{Loop measures and soups}  \label{loopsoupsec}
  
  \subsection{Loop measure}  \label{loopmeasuresec}
  
  Recall that a loop rooted at $x$ in $A$ is an element of
  $\paths_A(x,x)$.  We will say that $l \in \paths_A(x,x)$
  is an {\em elementary loop} if it is nontrivial ($|l| > 0$)
  and the only visits to $x$ occur at the beginning and terminal
  vertices of $l$.  We write  $\tilde \loop_x^1 = \tilde \loop_x^1(A)$
  for the set of elementary loops in $A$ rooted at $x$.   Recall that
if $q$ is an integrable weight,
\begin{equation}  \label{dec28.3}  
  f_x = \sum_{l \in\tilde \loop_x^1 } q(l), \;\;\;\;
      G_A(x,x) = \frac1{1 - f_x}.  
 \end{equation}
  Any nontrivial loop $l \in \paths_A(x,x)$ can be written uniquely
  as
\begin{equation}  \label{dec28.2}
   l = l^1 \oplus \cdots \oplus l^k , 
   \end{equation}
  where $k$ is a positive integer and $l^1,\ldots,l^k \in  \tilde \loops_x^1 $.
  We write $\tilde \loops_x^k $ for the set of loops of the form \eqref{dec28.2}
  for a given $k$, and we write  $ \tilde \loops_x^0 $ for the set containing
  only the trivial loop at $x$.  Let $ \tilde \loops_x
  = \tilde \loop_x(A)$ be the
  set of nontrivial loops, so that we have partitions
  \[   \paths_A(x,x) = \bigcup_{k=0}^\infty \tilde \loops_x^k(A) , \;\;\;\;
     \tilde \loops_x(A) = \bigcup_{k=1}^\infty \tilde \loops_x^k(A).\]
   Note that $q(\tilde \loops_x^k) = f_x^k.$
    We will define a measure on nontrivial
  loops, that is, on
  \[    \tilde \loops(A) := \bigcup_{x \in A} \tilde \loops_x(A).\]

  \begin{definition}
  If $q$ is a weight on $A$, then the {\em (rooted) 
  loop measure} $\tilde m = \tilde m^q_A$  is defined 
  on $\tilde \loops(A)$ by 
  \[   \tilde m (l)  = \frac{q(l)}{|l|}  . \]
 \end{definition}
 
 The loop measure is a measure on $\tilde \loops(A)$ and hence
 gives zero measure to trivial loops.
 It may   not be immediately clear why one would make this definition.
 The usefulness of it comes when we consider the
 corresponding measure on unrooted loops.  An  unrooted loop
 is   an oriented loop that has forgotten where the loop starts.  
 
 \begin{definition}
 An {\em (oriented) unrooted loop} $\ell$ is an equivalence class
 of rooted loops of positive length under the equivalence relation
 generated by
 \[   [l_0,l_1,\ldots,l_n] \sim [l_1,l_2,\ldots,l_n,l_1]
  \sim [l_2,l_3,\ldots,l_n,l_1,l_2] \sim \cdots.\]
  \begin{itemize}
 \item  Let $\loop(A)$ denote the set of unrooted loops in $A$.
 \item If $B \subset A$, let $\loop(A;B)$ denote the set
 of unrooted loops in $A$ that include at least one vertex in $B$.
 If $B = \{x\}$ is a singleton, we also write $\loop(A;x)$. 
 \end{itemize}
  \end{definition} 
  
  We will write $\ell$ for unrooted loops and $l$ for rooted
  loops.   We write $l \in \ell$ if $l$ is a representative
 of the unrooted loop $\ell$. 
    Note that $|l|$ and $q(l)$ are 
 the same for all representatives of an unrooted loop $\ell$
 so we can write $|\ell|$ and $q(\ell)$.  
 
  For each unrooted loop,
 let $s_\ell$ denote the number of distinct representatives
 $l$ of $\ell$.  If $s_\ell = |\ell|$ we call $\ell$ {\em
 irreducible}; we also call a rooted loop $l$ irreducible
 if its corresponding $\ell$ is irreducible.  More generally,
 if $|\ell| = n, s_\ell = s$, then each representative
 of $\ell$ can be written as
 \[                l =    \underbrace{l' \oplus \cdots
   \oplus l'}_{n/s} \]
where $l'$ is an irreducible loop of length $s$.
  For example, if $\ell$ is the unrooted loop with
representative $[x,y,x,y,x]$ we have $s_\ell = 2$ and
the two irreducible loops are $[x,y,x]$ and $[y,x,y]$.  Note
that $s_\ell$ is always an integer dividing $n$.

\begad

For rooted loops we have two different notions: elementary
and irreducible.  The words are similar but refer to different
things.  Elementary loops are irreducible but irreducible
loops are not necessarily elementary.

The notion of elementary loops is used only for rooted loops
while irreducibility is a property of an unrooted loop.

\endad

\begin{definition}  The {\em unrooted loop measure} $  m
=  m_A^q$ 
is the measure on  $\loops(A)$ induced by the rooted loop
measure.  More precisely, for every  $\ell
 \in \loops(A)$,
\[    m(\ell)  
   = \sum_{l \in \ell}\tilde m (l) =   s_\ell \, \frac{q(\ell)}
     {|\ell|}. \]
 \end{definition}
 
 \begin{definition}$\;$
 \begin{itemize}
 \item If $l$ is a loop and $x \in A$, we let $\vlocal(l;x)$ be the
 vertex local time, that is, the  number of
 times that the loop visits $x$.  To be precise, if
 \[   l =[l_0,l_1,\ldots,l_n], \] 
 then
 \[    \vlocal(l;x) = \#\{j: 1 \leq j \leq n : l_j = x \}. \]
  In particular,   $\vlocal(l;x) = k $ if $l \in 
  \tilde \loop_x^k(A)$.
%
 
 \item  If $\ell$ is an unrooted loop, we similarly
 write $\vlocal(\ell;x).$  
 
 \end{itemize}
 
 \end{definition}

The next proposition is important.  It relates the unrooted
loop measure restricted to loops that visit $x$ to a measure
on loops rooted at $x$.

  \begin{proposition}  \label{oct1.lemma1}
  Let $m' = m'_{A,x}$ denote the measure on $
  \tilde \loop_x = \tilde \loop_x(A)$ that
  gives measure
  \[                   \frac{q(l)}{\vlocal(l;x)} \]
  to each $l \in \tilde\loop_x $.  In other words, if
  $V \subset \tilde \loop_x $, then
  \[  m '(V) = \sum_{k=1}^{\infty}
     k^{-1} \, q\left
     [V \cap \tilde \loop_x^k  \right].\] Then
  the induced measure on unrooted loops is $m$ restricted
  to $\loop(A;x)$.
  \end{proposition}
  
  \begin{proof}  Let $l$ be a representative of $\ell$
  in $\tilde\loop_x $ and let $s = s_\ell, n = |\ell|$.
  Then we can write
  \[    l = \underbrace{l' \oplus \cdots \oplus l'}_{n/s}, \]
  where $l'$ is an irreducible loop in $\tilde \loops_x $.
  The loop $l'$  is the concatenation of $n(l';x)$ elementary
  loops.
  Note that $\vlocal(l;x) = (n/s)\, n(l';x),$ and there
  are $\vlocal(l';x)$ distinct representatives of $\ell$ 
  in $\tilde\loop_x  $.
  
  \end{proof}
  
  Recall that if $B = \{x_1,\ldots,x_n\}$, then
  \[  F_B(A) = \prod_{j=1}^n G_{A_j}(x_j,x_j), \]
  where $A_j = A \setminus \{x_1,\ldots,x_{j-1}\}$.  In
  Proposition \ref{prop.sep27} (actually in 
  Exercise \ref{exer.prop.sep27}), it was shown that this is independent
  of the ordering of the points of $B$.  In the next
  proposition we give another expression for $F_B(A)$
  in terms of the unrooted loop measure
  that is clearly independent of the ordering.

  \begin{proposition}\label{funrooted}  Suppose that
  $q$ is an integrable weight on $A$.
  
  \begin{enumerate}
  
  \item If $x \in A$, 
  \[    \exp\left\{ m \left[\loop(A;x)\right]\right\}
   = G_A(x,x).\]
  
  \item  If $B \subset A$, then
  \[     F_B(A) = \exp\left\{m[\loops(A;B)]\right\}.\]
  
  \item  In particular,  
  \[   \exp\left\{m[\loops(A )]\right\} =F(A) =  \det G_A = \frac{1}{\det[I - \qprob]}. \]
  \end{enumerate}
  \end{proposition}
  
  \begin{proof}$\;$
  \begin{enumerate}
  \item  By the previous lemma, the measure  $m$ restricted
  to $\loop(A;x)$ can be obtained from $m'=m'_{A,x}$ by ``forgetting
  the root''.  Using \eqref{dec28.3}, we get 
\[
 m\left[\loop(A;x)\right]  =  \sum_{k=1}^\infty k^{-1} \,  q
 \left[\tilde \loops_{x}^{k}(A)\right] 
  =   \sum_{k=1}^\infty k^{-1}
    \, f_x^k   \hspace{.5in}\]\[\hspace{1.5in}
     =  - \log[1-f_x] = \log G_A(x,x). 
 \]
 \item  If $B = \{x_1,\ldots,x_n\}$ and $A_j = A \setminus
 \{x_1,\ldots,x_{j-1}\}$ we partition $\loops(A;B)$ as
 \[ \loops(A;B) = \loops(A;x_1) \cup \loops(A_1;x_2)
   \cup \cdots \cup \loops(A_n;x_n) , \]
 and use part 1 $j$ times. 
 \item  This is part 2 with $B = A$ combined with \eqref{fdet}.
 \end{enumerate}
 
 \end{proof}

 The last proposition might appear surprising at first.  The first
 equality can be rewritten as
 \[  \exp\left\{\sum_{\ell \in \loop(A;x)} m^q(\ell) \right\}
     = \sum_{l \in \paths_A(x,x)} q(l). \]
 On  the right-hand side we have a measure of a set of paths and on the 
 left-hand
 side we have the {\em exponential} of the measure of a set of
 paths.  However,
 as the proof shows, this relation follows from the Taylor
 series for the logarithm,
 \[          -\log(1-q) = \sum_{k=1}^{\infty} \frac{q^k}{k}.\]
 As a corollary of this result, we see that a way to sample
 from the unrooted loop measure $m $ on $A$ is to first choose
 an ordering $A = \{x_1,\ldots,x_n\}$  and then sample
 independently from the measures on rooted loops $\tilde{m}_{x_j,A_j} $
 where $A_j = A \setminus \{x_0,\ldots,x_{j-1}\}$.  Viewed
 as  a measure
 on {\em unrooted} loops, this is independent of the ordering
 of $A$.

 \subsection{Soups}
 
 We use the word soup for the more technical term
 ``Poissonian realization'' from a measure.
 If $\state$ is a set, then a {\em multiset} of $\state$
 is a subset where elements can appear multiple times.  A more
 precise formulation is that a multiset of $\state$ is
 an element of $\N^\state$, that is, a function
 \[      N: \state \rightarrow \N , \]
 where $N(x) = j$  can be interpreted as saying
  that the element $x$ appears $j$
 times in the multiset.  Here $\N = \{0,1,2,\ldots
 \}$.  Let $\Nfin^\state$ denote the {\em finite}
 multisets, that is, the $N$ such that $\#(x:
 N(x) > 0\}$ is finite.
 
 \begin{definition}  If $\mu$ is a positive measure on a
 countable state space $\state$, then a soup is a
 collection of independent Poisson processes
 \[        \overline N_t := \{N^{x}_t: x \in \state \}  , \]
 where $N^x$ has rate $\mu_x = \mu(x)$. 
 
 \end{definition}
 
 If $\mu$ is a finite measure, then ${\bf N}_t
 \in \Nfin^\state$ with probability one and 
 the distribution of the
 soup at time $t$ is
 \[      \Prob\{{\bf N}_t = \overline N\}
   =   e^{-t \|\mu \|}\prod_{x \in \state}
     \frac{(t\mu_x)^{N(x)}}{N(x)!} .\]
Although the product is formally an infinite product,  since
$\overline N \in \Nfin^\state$, all but a finite
number of terms equals one.  We can give an alternative
definition of a loop soup in terms of the distributions.
This definition will not require the measure to be
positive, but it will need to be a complex measure.
In other words, if $\mu_x$ denotes the measure of $x$,
then we need
\[     \|\mu  \| := \sum_{x \in \state}
   |\mu_x| < \infty . \]

\begin{definition}  If $\mu$ is a complex measure
on a countable set $\state$, then the  soup
is the  collection of complex measures $\{\nu_t\}$
on $\Nfin^\state$ given by
\begin{equation}  \label{complexsoup}
     \nu_t\left[ \, \overline N\, \right]
   =   e^{-t \mu(\state)}\prod_{x \in \state}
     \frac{(t\mu_x)^{N(x)}}{N(x)!} . 
\end{equation}
    \end{definition}

The generalization to complex measures is straightforward but
it is not clear if there is a probabilistic intuition.  Let us consider
the simple case of a ``Poisson random variable with parameter 
$ \lambda \in \C$''.  This does not make literal
sense but one can talk about
its ``distribution'' which is the complex measure $\nu$ on
$\N$ given by
\[             \nu(k) = e^{-\lambda} \, \frac{  \lambda ^k}{k!},\;\;\;\;k=
0,1,2,\ldots. \]
As in the positive case, $\nu(\N) = 1$; however, the total variation
is larger,
\[               \|\nu\| = \sum_{k=0}^\infty
   \left|  e^{-\lambda} \, \frac{ \lambda^k}{k!}\right| = e^{  
      |\lambda|- \Re(\lambda) }. \]
  Using this calculation, we can see that $\nu_t$ as defined
 in \eqref{complexsoup} is a complex measure on $\Nfin^\state$ 
of total variation
\[    \|\nu_t\| = \prod_{x \in \state} \exp\left\{t[
 |\mu_x| - \Re(\mu_x)]  \right\}=
\exp\left\{
t(\|\mu\| - \Re[\mu(\state)])\right\}. \]

\subsection{The growing loop at a point}  \label{growingsec}

In this subsection, we fix $x \in A$ and an integrable
 weight $q$ and consider
loops coming from the measure $m' = m'_{A,x}$ as in 
Proposition \ref{oct1.lemma1}.   We first consider
the case of positive weights $p \geq 0$. Recall that
the  measure $m'$ is given by
\[  m'(l^1 \oplus \cdots \oplus l^k)
         = k^{-1} \, p(l^1) \cdots p(l^k), \;\;\;\;
          l^j \in 
          \tilde \loop^1 := \tilde \loops^1_x(A)  .\]
 Sampling from $m'$ can be done in a two-step method,
 \begin{itemize}
 \item Choose $k \in \N$ from the measure
        \[       \nu(k) = \frac{1}{k} \, f^k\;\;\;\;
         \mbox{where } f= f_x  = \sum_{l \in\tilde  \loops^1} p(l) , \]
   \item Given $k$, choose $l^1,\ldots, l^k$ independently
   from the probability  measure $p/f$ on $\tilde \loops^1$.
  \end{itemize}

 At time $t$, the soup  outputs a (possible empty)
 multiset of loops in $\tilde \loops_x  $.  If
 we concatenate them in the order they appear, we get
 a single loop in $\paths_A(x,x)$, which we denote
 by $l(t)$.  We can also write $l(t)$ as 
  a concatenation
 of elementary loops. 
 If no  loops have
 appeared in the soup, then the concatenated
 loop is defined to be the  trivial
 loop $[x]$. 

\begin{definition}
The process $l(t)$ is
the {\em growing loop} (in $A$ at $x$ induced
by $p$) at time $t$. 
\end{definition}

  \begad
 
  The growing loop at time $t$ is a concatenation of
 loops in $\tilde \loops_x(A)$; if we only view the
 loop $l(t)$ we cannot determine how it was formed
 in the soup. 
 
 \endad
 
 The growing loop can also be defined
as the  continuous time Markov
 chain with state space $\paths_A(x,x)$ which starts
 with the trivial loop and whose transition rate of
 going from $\tilde l$ to $\tilde l \oplus l$ is
 $m'(l)$.  The next proposition computes the distribution
 of the loop $l(t)$.

 \begin{proposition}  $\;$  \label{prop.growingloop}
 \begin{itemize}
 \item The distribution of the 
 growing loop at $x$ 
 at time $t$,  is
 \[  \mu_t(l)   
  = \frac{1}{G_A(x,x)^t} \, 
  \frac{\Gamma(k +t)}
  {k! \, \Gamma(t)} \, p(l)\;\;\;\;
       l \in\tilde \loops_x^k(A).\]
   \item In particular, the distribution
   at time $t=1$ is given by
  \[  \mu_1(l)   
  = \frac{p(l)}{G_A(x,x) } \, 
   \;\;\;\;
       l \in\tilde \loops_x^k(A).\]
       \end{itemize}
  \end{proposition}

 \begad
The expression involving the
Gamma function is also written as 
 the general binomial coefficient defined by
  \[ \binom{k+t-1}{k}
    := \frac{(k+t-1) \, (k+t-2) \, \cdots
      \, (k+t - k)}{k!} = \frac{\Gamma(k +t)}
  {k! \, \Gamma(t)} .\]
 We choose to use the Gamma function form because we
 will occasionally use properties of the Gamma function.
 
 \endad
 
 \begin{proof}
  We can    
 decompose  the growing loop at time $t$ into   a number
 of elementary loops $l^j \in \tilde \loops^1 $. 
 Let $K_t$ be the number of elementary loops in $l(t)$.
   Given $K_t$, the elementary loops $l^1,
 \ldots,l^k$ are chosen independently from the
 measure $p/f$.

 
  To  compute $\mu_t$
 we first consider the distribution 
 on $\N$ for the number of elementary loops at $t$.
 Given the number of such loops, the actual loops are 
 chosen independently from $p/f$.  In other words,
 the distribution $\mu_t$ at time $t$ can be written
 as
 \[   \mu_t[l^1 \oplus \cdots \oplus l^k]
   = \Prob\{K_t = k\}\;  \frac{p(l^1) \cdots p(l^k)}
          {f^k} , \;\;\;\; l^1,\ldots,l^k \in \tilde\loops^1.  \]
The process $K_t$ is sometimes called the
 {\em negative binomial process} with parameter
 $f$. 
It can  also be viewed as the L\'evy process with L\'evy measure
 $f^k/k$,
which can be written    as a   compound
 Poisson process  
 \[   K_t = Y_1 + \cdots +Y_{N_t},\]
 where $N_t$ is a Poisson process with parameter
 $m' (\tilde \loops  ) = -\log(1-f)$, and $Y_1,Y_2,\ldots$ are independent
 random variables with distribution
 \[  \Prob\{Y_j = k\} = \frac{1}{-\log(1-f)} 
 \, \frac{f^k}{k}.\]
 The distribution of $K_t$ is given by (see remark below)
\begin{equation}  \label{negbin}
\Prob\{K_t = r\}  = \frac{\Gamma(r +t)}
  {r! \, \Gamma(t)}\,f^r \, (1-f)^t = \binom{r+t-1}{r}
   \,f^r \, (1-f)^t .
  \end{equation}
 Therefore if $l = l^1 \oplus \cdots \oplus l^k$ with
 $l^j \in \tilde \loops^1$,  
 \[ \mu_t(l) = \left[\frac{\Gamma(k +t)}
  {k! \, \Gamma(t)}\,f^k \, (1-f)^t\right] \,
   \frac{p(l^1) \cdots p(l^k)}
          {f^k} = \frac{\Gamma(k +t)}
  {k! \, \Gamma(t)} \, (1-f)^t \, p(l).\]
%
%
Recalling that
 $(1-f) = 1/G_A(x,x)$, we get the result.
\end{proof}

 \begad
 Here we discuss some facts about the negative
 binomial process $K_t$ with parameter
 $p\in(0,1)$.  At time $1$, $K_1$ will have a geometric
 distribution with parameter $p$,
 \[   \Prob\{K_1 = k\} = p^{k}(1-p), \;\;\; k=0,1,2,\ldots,\]
and hence
\[     \E\left[e^{isK_1}\right] = \sum_{k=0}^\infty e^{iks} \, p^k \,
  (1-p) = \frac{1-p}{1-pe^{is}}.\]
  Since $K_t$ is a L\'evy process, we see that 
 the characteristic function of $K_t$ must be 
  \[   \left[\frac{1-p}{1-pe^{is}}\right]^t.\]
To check that \eqref{negbin} gives the distribution
for $K_t$, we compute the characteristic function.
  Using the binomial expansion (for positive, real $t$),
  we see that
  \[  (1-p)^{-t} = \sum_{k=0}^\infty \frac{\Gamma(k+t)}
     {k! \, \Gamma(t)} \, p^k,\]
  which shows that 
  \[   \nu_t(k) := \frac{\Gamma(k+t)}
     {k! \, \Gamma(t)} \, p^k\, (1-p)^t ,\]
 is a probability distribution on $\N$.  Moreover, if $K_t$
 has distribution $\nu_t$,
 \[  \E\left[e^{isK_t}\right] 
 = \sum_{k=0}^\infty \frac{\Gamma(k+t)}
     {k! \, \Gamma(t)} \, p^k\, (1-p)^t \, e^{iks} =
         \left[\frac{1-p}{1-pe^{is}}\right]^t.\]
     
     \endad
 
 \begin{exercise}  
 Let $f \in (0,1)$ and let $\mu_t$ denote
 the probability distribution on $\N$ given
 by
 \[   \mu_t(k) =
 \Prob\{K_t = k\}  = \frac{\Gamma(k +t)}
  {k! \, \Gamma(t)}\,f^k \, (1-f)^t.\]
   Show that for each $k$,
  \[   \lim_{t \downarrow 0}
      t^{-1} \, \mu_t(k) = 
  \frac{1}{-\log(1-f)} 
 \, \frac{f^k}{k}.\]
 
 \end{exercise}

 We can extend the last result to show a general principle
 \begin{itemize}
 \item  The distribution of the loops erased in a LERW is
 the same as that of the appropriate soup at time
 $t=1$.
 \end{itemize}

\begin{corollary}
Suppose $\omega \in\paths_A(x,y)$
and $l$ denotes the loop erased at $x$ in the
definition of $LE(\omega)$.  Then the distribution
of $l$ is the same as the distribution of the growing
loop at time $t=1$.
\end{corollary}

\begin{proof}  This follows immediately by comparison
with Proposition \ref{erasedloops}.
\end{proof}

The growing loop distribution is also defined for complex
integrable weights $q$ although some of the probabilistic
intuition disappears. Let $f = f_x$ be
as before although now $f$ can be complex.  Integrability
implies that $|f| < 1$, so we can define the negative
binomial distributions
\[   \nu_t(k) =  \frac{\Gamma(r +t)}
  {r! \, \Gamma(t)}\,f^r \, (1-f)^t.\]
Since $|f| < 1$, this gives a complex measure on
$\N$ with
\[  \nu_t(\N) = \sum_{j=0}^\infty 
\frac{\Gamma(r +t)}
  {r! \, \Gamma(t)}\,f^r \, (1-f)^t = 1  .\]
  We can then define
\begin{equation}  \label{dec13.1}
  \mu_t(l) = \frac{1}{G_A(x,x)^t} \, 
  \frac{\Gamma(k +t)}
  {k! \, \Gamma(t)} \, q(l)\;\;\;\;
       l \in\tilde \loops_x^k(A), 
       \end{equation}
   and we can check that $\mu_t$ is a complex measure
  on $\tilde \loops_x(A)$ with $\mu_t[
  \tilde \loops_x(A)] = 1$.
If $q$ is \greenable but
not integrable, the formula \eqref{dec13.1}
 defines a function on $[0,\infty) \times \tilde \loops_x^k(A)$
 although $\mu_t$ is   not necessarily   a measure.

 \begad
 
 As in the case of positive weights, we can view the measure
 $\mu_k$ for integrable $q$
  in two steps: first, choose $k$ according to (the complex
 measure) $\nu_t$
 and then, given $k$, choose independent
 $l^1,\ldots,l^k$ from the
 measure $q/f$.  This latter measure gives measure one to $\tilde \loops^1_x$,
 although it is not a probability measure since
 it is  not  a  positive measure.

 \endad

 \begin{exercise}

Let $z$ be in the open unit disk of the complex
plane and \[ q(t,r) = 
\frac{\Gamma(r +t)}
  {r! \, \Gamma(t)}\,z^r \, (1-z)^t, \;\;\;\;
  r =0,1,2,\ldots \]
Verify directly that $q(t,r)$ is the solution
of the system 
\[ \p_t \, q(t,r) = \log (1-z) \, q(t,r)
 + \sum_{k=1}^{r}
  q(t,r-k) \, \frac{z^k}{k} .\]
  with initial condition
  \[  q(0,r)  = \left\{ \begin{array}{ll}
   1, & r  =0,\\
   0, & r \geq 1.  \end{array}  \right.\]
 You may wish to derive or look up properties
 of the logarithmic derivative of the Gamma
 function,
 \[   \psi(x) = \frac{\Gamma'(x)}{\Gamma(x)}.\]
  
\end{exercise} 
\subsection{Random walk bubble soup}  \label{bubblesec}

We continue the discussion of the previous subsection to
define what we call the {\em (random walk)
bubble soup}.  We start with a finite set
$A$ and an ordering of the points $A = \{x_1,\ldots,
x_n\}$.  As before, we write $A_ j =A \setminus
\{x_1,\ldots,x_{j-1}\}$. 

\begin{definition}
The {\em (random walk) bubble soup} (for the
ordering $x_1,\ldots,x_n$) is an increasing
collection of multisets from $\tilde \loops(A)$
obtained by taking the union of the independent
soups from $m'_{x_j, A_j}$.

\end{definition}

\begad

The colorful terminology ``bubble soup'' come from the
relation between this discrete construction and a
construction of the Brownian loop soup in terms
of  (boundary) ``bubbles''.

\endad

By concatenation we
can also view the bubble soup as an $n$-tuple of
growing loops ${\bf l}(t) = (l^1(t),\ldots,l^n(t)) ,$ 
where $l^j(t)$ is the loop growing at $x_j$ in $A_j$.
These loops grow independently (although, of course, their
distribution depends on the ordering of $A$).  
More generally, if $B = \{x_1,\ldots,x_k\} \subset A$
is an ordered subset of $A$, we can 
define the bubble soup restricted to $\tilde \loops(A;B)$
 as a collection of
growing loops ${\bf l}(t) = (l^1(t),\ldots,l^k(t)) .$ 
The following is an immediate consequence of 
Proposition \ref{prop.growingloop} and the relation
\[     \det G= \prod_{i=1}^n G_{A_i}(x_i,x_i).\]

\begin{proposition}  The distribution of the bubble soup
at time $t$ is given by
\[\mu_t({\bf l}) =  \frac{ q({\bf l}) }{[\det G]^t} \, 
 \left[\prod_{i=1}^n  \frac{\Gamma(j_i +t)}
  {j_i! \, \Gamma(t)} \right] 
    , \]
    where ${\bf l} = (l^1,\ldots,l^n)$; $q({\bf l})
    = q(l^1) \cdots q(l^n)$; and 
  $j_i$ is   the number of elementary loops
in $l^i$, that is,
$   l^i \in \tilde \loops_{x_i}^{j_i}(A_i) .$
In particular,
\[   \mu_1({\bf l}) =  \frac{q({\bf l})}{ \det G  } ,\]
\begin{equation}  \label{onehalfsoup}
 \mu_{1/2}({\bf l}) = 
\frac{ q({\bf l})}{\sqrt{\det G}} \, 
 \left[\prod_{i=1}^n  \frac{\Gamma(j_i +\frac 12 )}
  {j_i! \, \sqrt \pi } \right]  .
  \end{equation}
\end{proposition}
 
Note that we
can write 
\[  \mu_t({\bf l}) = c(t,{\bf l}) \,  \frac{ q({\bf l}) }{[\det G]^t}, \]
where $c(t,{\bf l})$ is a combinatorial term, independent
of $q$,
\[   c(t,{\bf l}) = 
\prod_{i=1}^n  \frac{\Gamma(j_i +t)}
  {j_i! \, \Gamma(t)}.\]

\begin{proposition}  Suppose $p$ comes from a Markov chain.
Suppose $x \in A$, and the loop-erased walk from $x$ to $\partial
A$ is   
\[   \eta = [\eta_0 = x_0,x_1,\ldots,x_n, \eta_{n+1}].\]
Then given that the loop-erasure is $\eta$, the distribution
of the loops erased is the same as the growing loops restricted
to $\eta$ using the ordering $\{x_0,\ldots,x_n]$.
\end{proposition}

We can state this a different way.

\begin{proposition}
Suppose $p$ is coming from a Markov chain
and $x_0 \in A$. 
\begin{itemize}
\item  Let $\eta =[\eta_0,\eta_1,\ldots,\eta_n]$ be
LERW from $x_0$ to $\partial A$.  That is, $\eta$
is chosen from the probability distribution on
$\saws_A(x_0,\p A)$,
\[    \hat p(\eta) = p(\eta) \, F_\eta(A).\]
\item  Given $\eta$, take a  realization of the
bubble soup using the ordering $\{\eta_0,$ $\ldots,
$ $\eta_{n-1}\}$.
Let $  {\bf l}(1)= [l^0,\ldots,l^{n-1}]$ be the
loops that intersect $\eta$ at time $t=1$.

\end{itemize}
Then the path
\[  \omega = l^0 \oplus[\eta_0,\eta_1] \oplus l^1
 \oplus [\eta_1,\eta_2] \oplus l^2 \oplus \cdots
  \oplus  l^{n-1} \oplus [\eta_{n-1},\eta_n] \]
 has the distribution of the Markov chain started
 at $0$ ending at $\partial A$.  In other words, for each
 $\omega \in \paths_A(x_0,\partial A)$, the probability
 that this algorithm outputs $\omega $ is $p(\omega)$.
 
 \end{proposition}

\subsection{Random walk loop soup}  \label{walksoupsec}

\begin{definition}
The {\em (random walk) loop soup} is a soup from
the measure $m$ on unrooted loops $ \ell \in \loops_A$.  
\end{definition}

If $q$ is a positive measure, then the soup can 
be viewed as an independent collection of Poisson
processes $\{X^\ell_t: \ell \in \loops_A\}$,  where
$X^\ell$ has parameter $m(\ell)$. 
 If $q$ is complex,
the soup is defined only as the collection of
complex measures $\nu_t$ on $N^{\loops_A}$. 
The definition of the unrooted loop soup does
not require an ordering of the points on $A$.

However, if a realization of the loop soup is
given along with an ordering of the 
vertices $A = \{x_1,\ldots,x_n\}$ we can get
a soup on rooted loops with a little more 
randomness.
Indeed, suppose that 
an ordering of the vertices $A = \{x_1,\ldots,x_n\}$
is given.  If $\ell \in \loops_A$, we choose a
rooted representative of $\ell$ as follows:
\begin{itemize}
\item  Find the smallest $j$ such that the vertex
$x_j$ is in $\ell$.
\item  Consider all $l \in \ell$ that are rooted
at $x_j$ and select one (uniformly).
\end{itemize}
This gives a collection of rooted loops.  At any
time $t$ we can construct  a loop
 in $\paths_A(x,x)$
by concatenating all the loops in $\paths_A(x,x)$
that have been output by time $t$, doing the
concatenation in the order that the loops arrive.

\begin{proposition}  The random walk loop soup
considered as a collection of growing loops as above
has the same distribution as the bubble loop soup.
\end{proposition}

\begin{proof}  This is not difficult to show given
the fact that the measure $m'_{A,x}$, considered
as a measure on unrooted loops $\loops_x(A)$, is
the same as $m$ restricted to $\loops_x(A)$.
\end{proof}

\begin{proposition}
Suppose $p$ is coming from a Markov chain
and $x_0 \in A$. 
\begin{itemize}
\item  Let $\eta =[\eta_0,\eta_1,\ldots,\eta_n]$ be
LERW from $x_0$ to $\partial A$.  That is, $\eta$
is chosen from the probability distribution on
$\saws_A(x_0,\p A)$,
\[    \hat p(\eta) = p(\eta) \, F_\eta(A).\]
\item Let $\{X_t^\ell: \ell \in \loops(A)\}$ denote
an independent realization of the random walk loop soup.
Let us view the realization at time $t=1$ as a finite
sequence of loops
\[ [\ell^1,\ell^2,\ldots,\ell^M ] \]
where the loops are ordered according to the
 time they were added
to the soup.

\item Take a subsequence of   these loops, which we
also denote by $[\ell^1,\ell^2,\ldots,\ell^M ]$, by
considering only those loops that intersect $\eta$.

\item  For each $\ell^k$ let $j$ be the smallest index
such that $\eta_j \in \ell^k$.  Choose a rooted
representative $\tilde l^k$ of $\ell^k$ rooted at $\eta_j$.
If there are several representatives choose uniformly
among all possibilities.  Define loops $l^j, j=0,\ldots,n-1$
to be the loop rooted at $\eta_j$ obtained by concatenating
(in the order they appear in the soup)  all
the loops $\tilde l^k$ that are rooted at $\eta_j$.
 
\end{itemize}
Then the path
\[  \omega = l^0 \oplus[\eta_0,\eta_1] \oplus l^1
 \oplus [\eta_1,\eta_2] \oplus l^2 \oplus \cdots
   l^{n-1} \oplus [\eta_{n-1},\eta_n] \]
 has the distribution of the Markov chain started
 at $0$ ending at $\partial A$.  In other words, for each
 $\omega \in \paths_A(x_0,\partial A)$, the probability
 that this algorithm outputs $\omega $ is $p(\omega)$.
 
 \end{proposition}
 
 \section{Relation to Gaussian field}  \label{fieldsec}

There is a strong relationship between the loop soup
at time $t=1/2$ with a Gaussian field that we
will discuss here.
We will consider only integrable,
Hermitian
  weights $q$, that is $q(x,y) = \overline{q(y,x)}$.
 If $q$ is real, then this implies that $q$ is symmetric.  
This implies that for every path $\omega$,
$q(\omega^R) = \overline{q(\omega)}.$
Every Hermitian weight can be written as
\[        q(x,y) = p(x,y) \, \exp\{i\Theta(x,y)\},\]
where $p$ is positive and symmetric and $\Theta$
is {\em anti-symmetric}, $\Theta(x,y)
= -\Theta(y,x)$.  If $q$ is an integrable Hermitian
weight on $A$, then
\[     f_{x} = \sum_{l \in \tilde \loops_x^1(A)}
    q(l) \in \R , \]
 since $q(l) + q(l^R) \in \R$,  and hence
 \begin{equation}  \label{oct9.1}
    G_A(x,x) = \frac1{1-f_x}  \in \left(
 \frac 12,\infty \right).
 \end{equation}

\begin{proposition}  If $q$ is an integrable,
Hermitian
weight on $q$, then the Green's matrix $G$ is
a positive definite Hermitian matrix.
\end{proposition}

\begin{proof}  Since $I-\qprob$ is Hermitian,
it is clear that $G = (I-\qprob)^{-1}$ is Hermitian.
It suffices to prove that $ I-\qprob$
is positive definite.  Since $I-\qprob$ is Hermitian, 
Sylvester's criterion   states that
it suffices to show that for each $V \subset A$,
that $\det (I-Q_V) > 0$ where $Q_V$ denotes $\qprob$
restricted to the rows and columns associated to
$V$. If $V = \{x_1,\ldots,x_k\}$, then
\eqref{fdet} gives
\[  \frac{1}{\det (I -Q_V)}
    = G_V(x_1,x_1) \, G_{V_1}(x_2,x_2)
     \, \cdots \, G_{V_k}(x_k,x_k) ,\]
where $V_j = V \setminus \{x_1,\ldots,x_{j-1}\}.$
This is positive by \eqref{oct9.1}.
\end{proof}


\subsection{Weights on undirected edges}

\begin{definition} $\;$
\begin{itemize}
\item  A {\em (real, signed) weight} on undirected
edges $\edges_A$ is a function
  $\theta: \edges_A \rightarrow \R$. 
  \item  If $\theta$ is a weight on $\edges_A$,
then there is a symmetric weight $q = q_\theta$
on directed edges given by 
\begin{equation}  \label{thetadef2}
   q(x,y) = \left\{\begin{array}{ll}
       \theta_{xy}/2, & x \neq y \\
        \theta_{xx} , & x = y \end{array} \right. . 
        \end{equation}.
\item  Conversely, if $q$ is a symmetric
weight on $\dedges_A$, we define $\theta$ by
\begin{equation}  \label{thetadef}
  \theta_e = \theta_{e,q} = \left\{\begin{array}{ll}
q(x,y) + q(y,x) = 2q(x,y) , & x \neq y \\
  q(x,x) , &x = y .
  \end{array}  \right. 
  \end{equation}
\item  We say that $\theta$ is {\em integrable}
or {\em \greenable}if
the corresponding $q$ is integrable  or greenable,
respectively.
\item If $f: A \rightarrow \C$,
we also write $f$ for the function
 $f: \edges_A \rightarrow \C$ by $f_e = f(x) \, f(y)$
where $e$ connects $x$ and $y$.
\end{itemize}
\end{definition}

Clearly it suffices to give either $\theta$ or $q$,
and we will specify symmetric weights either way.
Whenever we use $\theta$ it will be a function
on undirected edges and $q$ is a function on directed
edges.  They will always be related by \eqref{thetadef2}
and \eqref{thetadef}.  If we give $\theta$, we will write
just $q$ for $q_\theta$.  In particular, if $\theta$
is integrable, we can discuss the Laplacian $\laplace =
 \qprob-I$ and the Green's function $G = (I-\qprob)^{-1}$ where
$\qprob = [q(x,y)]_{x,y \in A}$.  We will write
\[      D = \det (I-\qprob) = \frac{1}{\det G}.\]
 
\subsection{Gaussian free field}

\begin{definition}  Given a (strictly)
positive definite symmetric
real matrix $\Gamma$ indexed by a finite set $A$,
the centered {\em Gaussian   field (with Dirichlet
boundary conditions)} is a centered multivariate
normal random vector $\{Z_x: x \in A\}$ indexed
by $A$ with covariance matrix $\Gamma$.
\end{definition}

The density of $Z_x$ is given by
\[    f(x) = \frac{1}{(2\pi)^{\#(A)/2}
  \, \sqrt{\det \Gamma} }
    \, \exp\left\{-\frac 12  \langle f,
      \Gamma^{-1} f \rangle \right\}, \]
 where $\langle \cdot \rangle$ denotes the dot
 product
 \[   \langle f,
      \Gamma^{-1} f \rangle  =
        \sum_{x,y \in A}
          f(x) \, \Gamma^{-1}(x,y) \, f(y) .\]
We will consider the case $\Gamma = G, \Gamma^{-1}
    = I - Q = -\laplace$ where $q$ is a \greenable
    weight.  Then we have
\[ 
   -  \langle f,
      \Gamma^{-1} f \rangle
          =  -  \sum_{x \in A}
           \,  f(x)^2 
            +   \sum_{x \in A}
              \sum_{y \in A}
       q(x,y) \, f(x) \, f(y) 
    =- \sum_{x \in A}
           \,  f(x)^2 
            +   \sum_{e \in \edges_A}
       \theta_e \,f_e    .\]
    Here $\theta_e$ is as defined in 
    \eqref{thetadef}. From this we see that the Gaussian distribution with
covariance $G$  has Radon-Nikodym derivative with
respect to  independent standard Gaussians of
\[       (\det G)^{-1/2} \, \exp\left\{
 \frac 12\sum_{e \in \edges_A}
       \theta_e \,f_e \right\} .\] 
       
 \begin{definition}  
 The Gaussian field generated
 by a  \greenable weight $\theta$ on $\edges_A$ 
 is a random vector $\{Z_x: x \in A\}$
 indexed by $A$ whose density  
 is
 \begin{equation}  \label{fielddensity}
          \phi(\bar z) \,  \sqrt {D}  \, \exp\left\{
 \frac 12 \sum_{e \in \edges_A}
      \theta_e \, z_e \right\}. 
      \end{equation}
 where $\phi = \phi_A$ is the density of a standard
 normal random variable indexed by $A$.
 This is the same as the centered Gaussian
 field with covariance matrix $\Gamma = (I-\qprob)^{-1}$.
     
\end{definition}

We will consider a two step process for sampling from
the Gaussian field.  We first sample from the square
of the field, and then we try to assign the signs.
Recall that if $N$ is a standard normal random variable,
then $N^2$ has a $\chi^2$-distribution with one
degree of freedom.    In particular, if $T = N^2/2$, then
$T$ has density
$          (\pi t)^{-1/2} \, e^{-t}. $
The next proposition gives the analogous computation
for the Gaussian field weighted by $\theta$.

\begin{proposition}  Suppose 
$\{Z_x: x \in A\}$ is the  Gaussian field generated
 by an integrable weight 
 $\theta$ on $\edges_A$ and let $T_x = Z_x^2/2$.
 Then $\{T_x: x \in A\}$ has density.
 \begin{equation}  \label{may16.5}   \left[ \prod_{x \in A} \sqrt{\pi \, t_x}
  \right]^{-1} 
    \, \exp \left\{-\sum_{x \in A} t_x \right\}
     \,\Phi 
     \end{equation}
 where
 \[ \Phi =  \Phi_q =    \sqrt {D} 
     \,   \E\left[\exp \left\{ 
     \sum_{e \in \edges_A}
      \theta_e  J_e \, \sqrt{t_e} \right\}\right].\]  
  Here $\{J_x, x \in A\}$ are independent
  random variables $\Prob\{J_x = 1 \} = \Prob\{J_x = -1\}
   = 1/2$. 
 \end{proposition}

 \begad
 
 In other words, $\Phi$ is the Radon-Nikodym derivative of $\{T_x\}$
 with respect to the density obtained for standard normals
 ($\theta \equiv 0$).
 
 \endad
 
\begin{proof}  This is obtained by change of variables being a little
careful because the relation $\bar z \rightarrow \bar t$ is
not one-to-one. Let $n = \#(A)$,
and let 
$J_x = \sgn(z_x), y_x = |z_x|$ so that  $z_x=   J_x \,  y_x$. 
Then we can write the density \eqref{fielddensity} as
\[             \phi(\bar y) \,  \sqrt {D} 
 \, \exp \left\{\frac 12 \sum_{e \in \edges_A} \theta_e \, J_e
  \, y_e \right\}. \]
  We now do the change of variables $t_x = z_x^2/2 = y_x^2/2$,
  $dt_x = y_x \, dy_x = \sqrt{2t_x} \, dy_x $  to see that
  for a fixed value of $\bar J$, the density of $\bar T$
   restricted to $\bar z$ with
 the signs of $\bar J$ is
\begin{equation}  \label{dec28.4}
   \left[\prod_{x \in A}
   \frac{1}{\sqrt{\pi\, t_x}}\right] \,  \exp \left\{-\sum_{x \in A}
   t_x\right\} \,  \sqrt {D} 
     \,  \left[2^{-n} \,\exp \left\{  \sum_{e \in \edges_A} \theta_e \, J_e
  \, \sqrt{t_e} \right\}\right].\end{equation}
  If we now sum over the $2^n$ possible values for $\bar J$, we get the
  result.
\end{proof}

Note that \eqref{dec28.4} gives  
 the conditional distribution of the signs of the field given
 the square of the field.
 
 \begin{corollary}  \label{signcor}
  Suppose 
$\{Z_x: x \in A\}$ is the  Gaussian field generated
 by a \greenable weight 
 $\theta$ on $\edges_A$ and let $T_x = Z_x^2/2, 
 J_x = {\rm sgn}(Z_x)$.  Then the conditional distribution
 on $\{J_x\}$ given $\{T_x\}$ is proportional to
 \[         \exp \left\{ 
     \sum_{e \in \edges_A}
      \theta_e  J_e \, \sqrt{t_e} \right\}.\]
 \end{corollary}
  
  \subsection{The measure on undirected currents}
  
 We will write $\edges = \edges_A$ for the set of undirected edges.
 For each $e \in \edges, x \in A$, we let $n_e(x)$ be the
 number of times the edge  touches  $x$.  More precisely,
  $n_e(x) = 2$ if $e$
 is a self-loop at $x$; $n_e(x) = 1$ if $e$ connects $x$ to
 a different vertex; and $n_e(x) = 0 $ otherwise.  If
 $\bar k  = (k_e: e\in \edges) \in \N^\edges$, then $
 \bar k $ generates a local time
 on vertices by
\begin{equation}  \label{nx}
         n_x = n_x(\bar k) =
 \frac 12 \sum_{e \in \edges} k_e \,
 n_e(x) .
 \end{equation}
Note that
\[    \sum_{x \in A} n_x = \sum_{e \in \edges} k_e.\] 
 We say that $\bar k$ is an {\em (undirected) current} if
 $n(x)$ is an integer for each $x$.  Equivalently, $\bar k$
 is a current if for each $x \in A$ the number of edges in
 $\bar k$ that go from $x$ to a different vertex is even.
 Let $\currents = \currents_A$ denote the set of undirected
 currents in $A$.
 
 Given a weight $\theta$ on $\edges$, there is a 
 corresponding   symmetric weight
  $q$ on $\dedges_A$ given by \eqref{thetadef2}.
  The
  loop soup for an integrable weight  $q$ viewed at
 time $t$ induces a measure on $\N^{\edges}$ that is supported
 on $\currents$.  (This process is not reversible without adding
 randomness --- one cannot determine the realization of the loop
 soup solely from the realization of the current.)  
 At time $t=1/2$, this has a particular nice form.  If $\bar k
 \in \currents$ we define
 \[              \theta(\bar k) = \prod_{e \in \edges} \theta_e^{k_e}.\]

 \begin{theorem}  \label{oct10.theorem}
 If $\theta$ is an integrable weight on $\edges$  and $\mu = \mu_{1/2}$ denotes
 the distribution at time $t=1/2$
  of the corresponding loop  soup considered
 as a measure on $\currents$, then for each $\bar k \in \currents$,
 \[   \mu (\bar k) =  \sqrt {D} \, \left[\prod_{x \in A}
   \frac{\Gamma(n_x + \frac 12)}{\sqrt \pi}\right] \, 
 \left[\prod_{e \in \edges} \frac{1}{k_e!} \right] \, \theta(\bar k).\]
 Here $n_x = n_x(\bar k)$ is the vertex local time as in
 \eqref{nx}.
 \end{theorem}
 
\begin{proof}
The  proof
  will use a combinatorial identity that will be proved
  in Section \ref{combinsec}. 
   Here we will 
 show how to reduce it to this identity.  We write $q$ for the corresponding
 weight on directed graphs as in \eqref{thetadef2}. We choose
 an ordering of $A = \{x_1,\ldots,x_n\}$ and consider the growing loop
 (bubble)
 representation of the soup as in Section \ref{growingsec}. Let
 ${\bf l} = (l_1,\ldots,l_n)$ be the output of the growing loop
 at time $t=1/2$.  Let $\pi$ denote the function that sends each ${\bf l}$
 to the corresponding current $\bar k$; note that $\pi$ is not one-to-one.
 By \eqref{onehalfsoup} the measure on ${\bf l}$ of the soup is
\[   \sqrt {D}  \, 
 \left[\prod_{i=1}^n  \frac{\Gamma(j_i +\frac 12 )}
  {j_i! \, \sqrt \pi } \right]\, q({\bf l}) .\]
  If $\pi({\bf l}) = \bar k$, then
  \[      q({\bf l}) = 2^{-S(\bar k)} \, \theta(\bar k), \]
where 
$        S(\bar k) = \sum_{e\in\edges^0} k_e,$
and $\edges^0$ denote the edges in $\edges$ that are not self-edges.
Therefore, the induced distribution on $\currents$ gives measure
\begin{equation}  \label{sept21.1}
        \frac{\theta(\bar k)\,  \sqrt {D} }{2^{S(\bar k)}  }
 \, \sum_{\pi({\bf  l}) = \bar k}
 \left[\prod_{i=1}^n  \frac{\Gamma(j_i +\frac 12 )}
  {j_i! \, \sqrt \pi } \right]
  \end{equation}
 to $\bar k$.  So we need to show that
 \[   \sum_{\pi({\bf  l}) = \bar k}
 \left[\prod_{i=1}^n  \frac{\Gamma(j_i +\frac 12 )}
  {j_i!   } \right] =  2^{S(\bar k)} \left[\prod_{x \in A}
    {\Gamma(n_x + \frac 12)} \right] \, 
 \left[\prod_{e \in \edges} \frac{1}{k_e!} \right].\]
This is done in Theorem \ref{grapheqtheorem}. We note that
since \eqref{sept21.1} also represents the measure of the
current $\overline k$ from the (unordered) loop measure, that
the sum on the left-hand side is indpendent of the ordering
of the vertices. 
    \end{proof}
    
    \begad
  In the last proof, $\pi$ is used both for a function on loops
  and for the number $3.14\cdots$.  This will also
  be true in Section \ref{combinsec}.  We hope
  that  this is not
confusing.
  \endad

    \subsection{A graph identity}  \label{combinsec}
    
  Here we prove a combinatorial fact that is a little more
  general than we need for Theorem \ref{oct10.theorem}.  We will change the notation slightly
  although there is overlap with our previous notation.
  Let $\graph = (A,E)$ be a finite but not necessarily
  simple graph.  The edges $E $ are undirected, but we allow 
  self-edges and multiple edges; let $E_0$ be the set
  of edges that are not self-edges.   We write $A = \{x_1,\ldots,x_n\},
  A_j = A \setminus \{x_1,\ldots,x_{j-1}\}$
  and we let $\bar \omega = (\omega^1,\ldots,\omega^n)$
  be an ordered $n$-tuple of loops where $\omega^j$ is a loop in $A_j$
  rooted at $x_j$.   To be more precise, a loop rooted at $x_j$
  in $A_j$ is a sequence of points
  \[   \omega^j = [\omega^j_0 = x_j ,\ldots,\omega^j_m = x_j] ,
  \;\;\;\; \omega^j_i \in A_j \]
  as well as a sequence of undirected edges
  \[      \omega^j = e_1 \oplus \cdots \oplus e_n , \]
  such that the endpoints of $e_i$ are $\omega^j_{i-1}$
  and $\omega^j_i$.  As before, we write $n_e(x) = 2,1,0 $,
  if $e$ is a self-edge at $x$; is in $E_0$
  and has $x$  as a
  vertex; and does not touch $x$, respectively.  A current
  $\bar k = \{k_e, e\in E\}$  is an element of $\N^{\edges}$ with
  the property that the number of edges going out of each vertex
  is even.  To be precise, if
  \[         n_x = n_x(\bar k) = \frac 12 \sum_{e \in E}  k_e 
  \, n_e(x) , \]
  then $n_x$ is an integer for each $x$.  For each $\bar \omega$ there
  is a corresponding current, which we denote by $\pi(\overline \omega)$,
  obtained by counting the total number of transversals of each edge. 
 We write $N_j  = N_j(\omega_j)$ for the number of elementary loops
 in $\omega_j$, that is, 
 \[       N_j = \#\{i\geq 1: \omega_i^j = x_j\}. \]
 We also let 
 \[     S(\bar k) = \sum_{e \in E_0} k_e.\]
\begin{theorem}   \label{grapheqtheorem}
 If $(A,E)$ is a graph,
 then for every $\bar k \in \currents$,
 \begin{equation}  \label{graphequality}
  2^{-S(\bar k)}  \sum_{\bar \omega , \, 
     \pi(\bar \omega) = \bar k} \;\prod_{j=1}^n 
      \frac{\Gamma(N_j + \frac 12)}
          {N_j! } = \left
    [\prod_{x \in A}  {\Gamma\left(n_x + \frac 12\right)} 
    \right] \;\;
      \left[
      \prod_{e \in E}  \frac 1{k_e!}\right]. 
      \end{equation}
  \end{theorem}
  
  We will do this by induction by treating various cases.
 We will use the fact from the last section that
 \[ \sum_{\bar \omega , \, 
     \pi(\bar \omega) = \bar k} \;\prod_{j=1}^n 
      \frac{\Gamma(N_j + \frac 12)}
          {N_j! }  \]
  is independent of the ordering of the vertices. 
  
  \subsubsection{Trivial case: $A = \{x\}, E = \eset$}
  In this case there is only one current and one possible 
  walk $\omega$.  Then $N_1 = 0, n_x = 0$ and both sides
  of \eqref{graphequality} equal $\Gamma (\frac 12)
   = \sqrt \pi$.

\subsubsection{Adding a self-edge}  \label{selfedge}

Suppose that \eqref{graphequality}
holds for a graph $\graph = (A, E)$
with $A = \{x_1,\ldots,x_n\}$ and consider a new graph
$\tilde \graph =(A,\tilde E)$ by adding one self-edge
$ \tilde e $ at $x_1$. We write $\currents, \tilde \currents$
for the currents for $\graph$ and $\tilde \graph$
respectively. 
   We write   
$\tilde k \in \tilde \currents $
as $\tilde k  = (\bar k, k   )$  where $\bar k \in \currents $
and $k = k_{\tilde e}$.
Let us write $n_x, \tilde n_x$ for the corresponding quantities
in $\graph, \tilde \graph$, respectively.    We also
write $\whoknows$ and $\tilde \whoknows$ for the corresponding
collections of ordered pairs $\bar \omega$.

 Let $U,V$ denote the left and right-hand
sides of \eqref{graphequality}, respectively, for $\graph$
and $\bar k$,
 and
let $\tilde U, \tilde V$ be the corresponding quantities for
$\tilde \graph$ and $\tilde k = (\bar k,k)$.  
We will show that $U = V$ implies
that $\tilde U = \tilde V$.

Let $r = n_{x_1}$
 and hence $\tilde n_{x_1}
  =r +   k_ {\tilde e}= r + k$. Note that 
$  \tilde n_{x_j } =   n_{x_j}$ for $j \geq 2$ and 
$  S(\tilde k) = S(\bar k) .$   In particular,
\begin{equation}  \label{may21.6}
  \tilde V = \frac{\Gamma(r + k + \frac 12)}
  {\Gamma(r + \frac 12) \, k!} \;   V.
  \end{equation}

If $ \bar \omega = (\omega_1,\ldots,\omega_n) \in  \whoknows$
with $\pi(\bar \omega) = \bar k$
we have  $N_1 = n_{x_1} = r$.
(this uses the fact that $x_1$ is the first vertex
in the ordering). To  obtain an 
$\bar \omega' \in \tilde \whoknows$ with $\pi (\bar \omega')
 = \tilde k$ we
  replace $\omega_1$
with $\tilde \omega_1$ which is constructed by placing the edge $ e$
$k$ times into $\omega_1$.  
  The number of ways to add the edge $e$ in $k$ times
is
\[         \binom{N_1 + k}{k} = \frac{(r + k)!}{r! \, k!}.\]
Note that $\tilde N_1 = N_1 + k$.  Using $S(\tilde k) = 
S(\bar k)$, we see that
\begin{equation}  \label{may21.7}
  \tilde U = \frac{(r + k)!}{r! \, k!}\, \frac{\Gamma( r + k + \frac 12) 
}{\Gamma(r + \frac 12)} \,\frac{r!}{(r+k)!} \,  U
 = \frac{1}{k!} \, \frac{\Gamma( r + k + \frac 12) 
}{\Gamma(r + \frac 12)} \, U 
\end{equation}
Comparing \eqref{may21.6} and \eqref{may21.7}, we see
that $\tilde U = \tilde V$.

\subsubsection{Edge duplicating}  \label{edgesplit}

Suppose that \eqref{graphequality}   holds for a given
 graph $\graph = (A,E)$ and
 we take an edge $e \in E$  and add another  edge
 $e_1$
 with the same endpoints.  If $ e$ is a self-edge, this is
 the same as adding a self-edge and we can use the previous
 argument; hence, we will assume that $e$ connects distinct
 vertices.  Let $\tilde \graph = (A,\tilde E)$ with
 $\tilde E = E \cup \{e_1\}$.  We write $\currents, \tilde{\currents},
 \whoknows, \tilde \whoknows$ for the corresponding quantities
 as before. 
 If $\tilde k \in \tilde \currents $, then we can obtain
 a current $\bar k \in \currents $ by letting 
 \[   \bar k_ e = \tilde k_ e + \tilde k_{ e_1}.\]
and letting $\bar k$ agree with $\tilde k$ on 
$E \setminus \{ e\}$.  Suppose that $  k_ e = k$.
Then there are $k+1$ possible $\tilde k$ that give
$\bar k $.   We write $\tilde k^j \in \tilde \currents  
  $
 for the current that agrees with $\bar k$ on
 $E \setminus \{ e\}$ and has 
  $\tilde k^j_ e = j, \tilde k^j_{ e_1} = k-j$, 

Let us fix $\bar k$ as above and let $U,V$ be the left and
right hand sides of \eqref{graphequality} for $  \graph$
and $\bar k$.   We will choose an ordering $A = \{x_1,\ldots,
x_n\}$ for which $x_1$ is an endpoint of $e$.
For
$j=0,1,\ldots,k$, let  $\tilde U_j,\tilde V_j$ be the
left and right-hand sides  of \eqref{graphequality} 
for $\tilde \graph$ and
$\tilde k^j$.  Note that $\tilde n_{x_i} = n_{x_i}$ and
$N_i = \tilde N_i$ for $i=1,\ldots,n$.  We will show that
if $U = V$, then $\tilde U_j = \tilde V_j$ for each
$j$.

First note that
\[   \tilde V_j =  \binom{k}{j} \, V . \]
 If $\bar \omega
\in \whoknows$ is a walk with $\pi(\omega) = \bar k$, then
$N_1 = k$ and we traverse $ e$ $k$ times.  In $\tilde{\whoknows}$
for each of these traversals we can either keep $ e$ or
we can replace $ e$ with $ e_1$.  There are 
\[           \binom{k}{j} \]
ways in which we can retain $ e$ for $j$ times and change
to $ e_1$ at $k-j$ times.  Therefore
\[    \tilde U_j = \binom{k}{j} \, U . \]

\subsubsection{Converting a self-edge}  \label{converting}
  
  Suppose that $\graph = (A,E)$ is a graph with $A = \{x_1,\ldots,x_n\}$
   for which
  \eqref{graphequality} holds.  Suppose that $e \in E$ is
  a self-edge at $x_1$.
     Let $\tilde \graph = (\tilde A, \tilde E)$
  be a new graph obtained by converting the self-edge to an edge
  to a new vertex, that is,
  \[            \tilde A = \{x_1,\ldots,x_n,y\}\]
  and $\tilde E = [E \setminus \{e\}] \cup \{  e'\}$ where
  $  e'$ connects $x_1$ and $y$.  
  We write $\currents, \tilde{\currents},
 \whoknows, \tilde \whoknows$ for the corresponding quantities
 as before. 
 We will show that
  \eqref{graphequality} holds for $\tilde \graph$.
  Let $\tilde k \in \tilde \currents $ and let
  $2k = \tilde k_y$.  Note that $\tilde k_y$ must be even since
  $y$ has no other edges adjacent to it. 
    Let
  $\bar k $    be
  the current in $\currents $
  that agrees with $\bar k$ on $E \setminus \{ e\}$ and has
  $\bar k_{ e}=  k$.  We will show that if \eqref{graphequality}
  holds for $\graph$ and $\bar k$, then it also holds
  for $\tilde \graph$ and $\tilde k$.  As before let $U,V$
  be the left and right-hand sides of $\eqref{graphequality}$
  for $\graph,\bar k$ and $\tilde U, \tilde V$ the corresponding
  quantities for $\tilde \graph, \tilde k$.

  Note that $n_x$ agrees with $\tilde n_x$ on $A$ with
  $\tilde n_y = k$.  Also, 
  $   S (\tilde k) = S(\bar k) + 2k$.
  This and a standard identity for the Gamma function give
\begin{equation}  \label{may21.8}
    \tilde V = \frac{k!}{(2k)!} \,  {\Gamma (
  k + \frac 12) } \, V
      =  \frac{ 
   \sqrt \pi }{ 2^{2k}  }\, V. 
   \end{equation}
  Comparing $U$ and $\tilde U$ is not difficult.  There is
  a one-to-one correspondence between walks $\omega_1$ that visit
  $e$ $k$ times and walks $\tilde \omega_1$ that visit $e'$
  $2k$ times.  We just replace each occurrence of $e$ with $e'
  \oplus e'$. 
 Therefore, 
 \[   \sum_{\bar \omega \in \whoknows, \, 
     \pi(\bar \omega) = \bar k} \prod_{j=1}^{n+1} 
      \frac{\Gamma(N_j + \frac 12)}
          {N_j! } = \sqrt \pi  \sum_{\tilde \omega \in \tilde \whoknows, \, 
     \pi(\tilde \omega) = \tilde k} \prod_{j=1}^n 
      \frac{\Gamma(N_j + \frac 12)}
          {N_j! } .\]
  Here we setting $N_{n+1} = 0$ as the value corresponding to the
  new vertex $y$.
 Therefore,
\begin{equation}
\label{may21.9}
 \tilde U = \frac{2^{S(\bar k)}}{2^{S(\tilde k)}} \, \sqrt \pi 
    \, U = 2^{-2k} \,\sqrt \pi \, U . 
    \end{equation}
    Comparing \eqref{may21.8} and \eqref{may21.9} gives
    $\tilde U = \tilde V$.

  \subsubsection{Merging vertices}
  
Suppose that  $\graph = (A,E)$  is a graph
where $A = \{x_1,x_2,\ldots,x_n,y_1,y_2,\ldots,y_s\}$
with $s \leq n$  such that there 
there are edges $  e_j, 1 \leq j \leq s$ connecting $x_j,y_j$;
 and for $1 \leq i < j \leq s$,  edges
 $e_{ij}$ connecting
$x_i,x_j$. We also assume that there are no more
edges adjacent to $y_1,\ldots,y_s$ but there may be more
edges connecting $x_1,\ldots,x_n$.
 
 Our new graph $\tilde \graph$ will combine
  $y_1,y_2,\ldots,y_s$ into a single vertex that we call
  $y$.  We keep the edges $e_1,e_2,\ldots,e_s$ (that now
 connect $x_j$  and  $y$) and we remove
  the edges $e_{ij}$.  The remaining edges of $E$
  (all of which connect points in $x_1,\ldots,x_n$)
  are also in $\tilde E$.  We write $\currents,\tilde \currents,
  \whoknows, \tilde \whoknows$ as before.
   
 We choose the orderings of $A = \{x_1,\ldots,x_n,$
 $ y_1,$ $\ldots, $ $y_s\}$
 and $\tilde A = \{x_1,\ldots,$ $x_n,$ $y\}$.
There is a one-to-one relationship between
 the $\omega \in \whoknows$ and 
   $\tilde \omega \in \tilde \whoknows$
  by replacing each traversal  of  the  
  edge $e_{ij}$   starting at $x_i$ with
   $  e_i \oplus   e_j $ and each traversal of
   $e_{ij}$   starting at $x_j$
   with $  e_j \oplus  e_i $.

 Suppose $\bar \omega \in \whoknows$  giving the
current $\bar k$ such that
  \[    n_{y_j} = a_j ,\;\;\;\; k_{ e_j} = 2 a_j ,\;\;\;\;
      k_{ e_{ij}} = b_{ij}  , \]
      and let 
      \[
        B_i = \sum_{j \neq i} b_{ij} , \;\;\;\;  B 
         = \frac{1}{2} \sum_i B_i , \;\;\;\;
          K  = B+ \sum{a_j}.\]
   Then $\bar \omega$ induces  $\tilde{\omega} \in \tilde{\whoknows}$
  and  
     $\tilde k \in \tilde \currents$
  with
 \begin{equation}  \label{may21.10}
     \tilde k_{ e_j} =k_j := 2 a_j + B_j,
\end{equation}
      \[ 
     S(\tilde k ) = S(\bar k) + B,
\;\;\;\;  n_y =   K     . \]
  If there are any edges $ e$ that are not of the form
  $ e_j$ or $ e_{ij}$, then $\bar k_ e = \tilde k_ e$.
  This correspondence $\bar k \mapsto
  \tilde k$ is not one-to-one.  Let us write $U(\bar k),
  V(\bar k)$ for the left and right-hand sides of \eqref{graphequality}
  for $(\graph, \bar k)$ and $\tilde U, \tilde V$ for the
  corresponding quantities for $(\tilde \graph, \tilde k)$.
  We will show that if $U(\bar k) = 
  V(\bar k)$  for each $\bar k$, then $\tilde U =
  \tilde V$.
  
  Note that
   \begin{eqnarray*}
    \tilde U & = & \sum \frac{2^{S(\tilde k)}}{2^{S(\bar k)}}
      \,   (\sqrt \pi)^{1-s} \, U(\bar k)\\
      & = & \sum\frac{1}
        {2^B\,(\sqrt \pi)^{s-1}} \, U(\bar k) \\
        &= &\sum\frac{1}
        {2^B\,(\sqrt \pi)^{s-1}} \, V(\bar k)\\
       & = & \tilde V \sum\frac
                             {2^{B}\, K!\,  [k_1!
                              \cdots k_s!]} 
                   {(2K)!\  a_1! \cdots
                             a_s! \, \prod_{ij} b_{ij}!} , 
   \end{eqnarray*}
  where in each case the sum is over all $a_j, b_{ij}$
  satisfying \eqref{may21.10}.  The result $\tilde U
  = \tilde V$ follows from the following combinatorial
  lemma.

\begin{lemma}  Suppose $K$ is a positive
integer and   $k_1,\ldots,k_n$ are positive
integers with $k_1 + \cdots + k_n = 2K$.
Then,
\begin{equation}  \label{may7.1}
\sum    \frac{  2^B\, K!}
   {  \, a_1! \cdots a_n! \, \prod_{ij}
    b_{ij}! } = \frac{(2K)!}{k_1!
    \, k_2! \cdots k_n!}.
    \end{equation}
  where
  \[  B= \sum_{1 \leq i < j \leq n} b_{ij}, \]
  and the sum is over all nonnegative integers $a_1,
  \ldots,a_n$ and $\{b_{ij} :1 \leq i < j \leq n\} $
  with
  \begin{equation}  \label{apr13.1}
     k_j = 2a_j + \sum_{i \neq j} b_{ij}.
     \end{equation}
  In the last formula, we write $b_{ji} = b_{ij}$
  if $j > i$.  
  \end{lemma}

 \begin{proof} The right-hand side of \eqref{may7.1}
  is the number of sequences
 $(m_1,\ldots,m_{2K})$  with $m_i \in \{1,\ldots,n\}$ such
 that integer $j$ appears exactly $k_j$ times.  We can
 write each such sequence as a sequence of $K $ ordered
 pairs
 \[  (m_1,m_2), \; (m_3,m_4), \;\ldots, \;(m_{2K-1},m_{2K}).\]
 Let $a_j$ denote the number of these pairs that equal $(j,j)$
 and if $i < j$, let $b_{ij}$ denote the number of
 these that equal $(i,j)$ or $(j,i)$. Then the condition
 that the integer $j$ appears exactly $j$ times in the first
 sequence
 translates to \eqref{apr13.1} for the sequence of ordered
 pairs.  The factor $2^B$ takes into consideration the
 fact that $(i,j)$ or $(j,i)$ are counted by the $b_{ij}$.
 
 \end{proof}

\subsubsection{General case}  We proceed by induction on
the number of vertices.  Suppose the result is true for
all graphs of $n$ vertices.  If $\graph = (A,E)$ is a 
simple (no self-edges or multiple edges) graph
of $n+1$ vertices, write $A = A'  \cup \{y\}$ 
where $A = \{x_1,\ldots,x_n\}$ and
$E = E' \cup E_y$ where $E_y$ is the set of edges that
include $y$.  By induction the result holds for $E'$ and
we can obtain $E$ from $E'$ by first adding edges to
vertices $\{y_j\}$ for each $j$ for which there is an
edge in $E$ connecting $x_j$ and $y$ and then 
merging the vertices to get the graph $\graph$.  This
handles simple graphs of $n+1$ vertices,
 but then we can add multiple
edges and self-edges as above.

\subsection{Square of the Gaussian free field}

If $\theta$ is an integrable weight on $\edges = \edges_A$ with corresponding
directed weight $q$, we can consider the loop soup associated
to $q$.  Here we consider this only as measure on
currents  
$\{k_e: e \in A\} \in \currents$ and hence also on vertex local times
$\{n_x: x \in A\}.$  We will use this measure plus some
extra randomness to construct the square of the Gaussian field
with weight $\theta$.

Let us first consider $\theta \equiv 0$ for which the measure
on currents is supported on the trivial current.  In this
case, the field $\{Z_x: x \in A\}$ should be the standard Gaussian
and hence $\{Z_x^2: x \in A\}$ are independent $\chi^2$ random
variables with one degree of freedom.  Equivalently, we can say
that if $R_x = Z_x^2/2$, then
$\bar R = \{R_x: x \in A\}$ are independent Gamma random
variables with parameters $\frac 12 $ and $1$, that is, each with density
\[                    \frac{1}{\sqrt{\pi t}} \, e^{-t} . \]

More generally, given a realization of the current $\bar k$
and hence of the  vertex
local times $\{n_x\}$, at each vertex $x$ we  put the
sum of $n_x$ independent exponential random variables of
rate $1$.  In other words, we will consider a random
vector $\bar Y = \{Y_x: x \in A\}$ such that $Y_x$ are
independent (given $\bar k$) with a Gamma density with
parameters $n_x$ and $1$.   If $\bar T = \bar R + \bar Y$,
then (given $\bar k$), $\{T_x\}$ are independent Gamma
random variables with parameters $n_x + \frac 12$ and $1$,
that is, the joint density for $\bar T$ is
\[             \left[ \prod_{x \in A}\frac{ 
              t_x^{n_x - \frac 12}}{\Gamma(n_x + \frac 12)}
              \right] \, \exp\left\{-\sum_{x \in A} t_x \right\}.\]
This can also be written as
\[   \left[   \prod_{x \in A} \Gamma(n_x + \frac 12)\right]^{-1}
 \, \left[\prod_{x \in A} t_x^{-1/2} \right]
 \, \left[\prod_{e \in \edges}   t_e^{k_e/2} \right]
 \, \exp\left\{-\sum_{x \in A} t_x \right\}.\]            
If we combine this with Theorem \ref{oct10.theorem}, we  
get the following.

\begin{proposition} Suppose $\theta$ is an integrable  weight
on a set $A$ with $n$ elements.  Let $\bar k$ denote an undirected
current and let $\mu = \mu_{1/2}$ denote the measure on $\bar k$
induced by the loop soup at time $1/2$.  Let $\bar T =   \bar R
+ \bar Y$
  as above.
Then the joint density on $(\bar k, \bar t)$ is given by
\begin{equation}  \label{may16.1}  \frac{ \sqrt {D} }{
\pi^{n/2}  } \,   
 \left[\prod_{e \in \edges} \frac{(\theta_e  \sqrt {  t_e})^{k_e}}{k_e!} \right] \, 
 \left[\prod_{x \in A} t_x^{-1/2} \right]
 \, \exp\left\{-\sum_{x \in A} t_x \right\}.
 \end{equation}
 
 \end{proposition}

  If $\bar \rho = \{\rho_ e:  e \in \edges \} \in \C^E $ and
 $\bar k \in  {\N}^E$, we set
 \[    \Psi(\bar k, \bar \rho) = 
  \prod_{ e \in \edges }
     \frac{\rho_ e^{k_ e}}{k_ e!}.\]
   We can write the density
 \eqref{may16.1}  as
 \[ \frac{ \sqrt {D} }{
\pi^{n/2}  } \, 
       \Psi(\bar k, \bar \rho) 
                \, \left[\prod_{x \in A} t_x^{-1/2}\right]
  \,    \exp\left\{-\sum_{x \in A} t_x\right\},\;\;\;\;
  \rho_ e = \theta_ e \, \sqrt{t_ e}.  \] 
Let $\hat g$ denote the marginal density
 on $\bar t$ which can be given by summing over all possibilities
 for $\bar k$,   
\begin{equation}  \label{may16.2}
   \hat g(\bar t) :=  \frac{ \sqrt {D} }{
\pi^{n/2} \,   [\prod_{x \in A} \sqrt{t_x} ]}  \,  \left[\sum_{\tilde k \in \currents
 }  \Psi(\bar k, \bar \rho)
 \right]   
                \,  
                \exp\left\{-\sum_{x \in A}t_x \right\}
  .
 \end{equation}
      We will now show
   the relationship between the distribution of $\bar t$ and the
   square of the Gaussian free field which was first found by Le Jan
   (see \cite{LeJan} and references therein).
    We will use the
   following lemma (see, e.g.,  \cite[Section 2.1]{ADCS}).
  
   \begin{lemma}  Suppose $\{J_x: x \in A\}$
   are independent $\pm 1$ coin flips, and for $ e \in E$,
   let $J_ e = J_xJ_y$ if $x,y$ are the endpoints of $ e$.
   For any $\bar \rho \in  {\C}^E$,
\begin{equation}  \label{may16.3}
  \E\left[
\exp\left\{ \sum_{ e \in \edges }  J_e\, \rho_ e \right\} \right]
= \sum_{\bar k \in \currents } \ \Psi(\bar k, \bar \rho).
\end{equation}
\end{lemma}

\begin{proof}     If $\bar k \in 
 {\N}^\edges $,   
 we let $n_x = n_x(\bar k)$ be as before.  Then
 we can 
expand
\[ \exp\left\{  J_x \, J_y \,\rho_ e \right\}
  = \sum_{k_ e=0}^\infty \frac{J_x^{k_ e} \, J_y^{k_ e} \rho_ e^{k_ e }}{k_ e!}.\]
 \[
\exp\left\{ \sum_{ e \in \edges }  J_e\,\rho_ e \right\} = \sum_{\bar k
\in \edges ^{\N}}
 \left[\prod_{x \in A} J_x^{n_x(\bar k)}\right] \, \Psi(\bar k, \bar \rho),\]
\[    \E\left[
\exp\left\{ \sum_{ e \in \edges }  J_e\, \rho_ e \right\} \right]
     = \sum_{\bar k\in \edges ^{\N}}
 \left[\prod_{x \in A} \E (J_x^{n_x (\bar k)})\right] \, \Psi(\bar k, \bar \rho).\]
If
  $n_x(\bar k)$ is  odd  for some $x \in A$,
  we get $\E (J_x^{n_x(\bar k)})=0$.  Otherwise,
 $\E (J_x^{n_x(\bar k)})= 1$ for all $x$.  This gives the lemma.

 \end{proof} 
 
 \begad

The formula \eqref{may16.1} is valid in the non-integrable case although
it is not the density for the measure.  However, the calculations of this
section are valid.  In particular,  for fixed $\bar t$, the conditional
measure on $\bar k$ is a measure 
and \eqref{may16.3} shows that the   total mass is positive.

\endad
 
 \begin{theorem}  Under the assumptions above, the  marginal density
 for  $\bar t= \{t_x\}$
 is the same as that of $\{Z_x^2/2 \}$ where $\bar Z$ is the centered
 multivariate normal
 distribution indexed by $A$ with covariance matrix $G$.
 \end{theorem}
 
 \begin{proof}   
 In \eqref{may16.5} we showed that the density of $(Z_x^2/2)$ is 
  \[  g(\bar t) = \frac{ \sqrt {D} }{ \pi ^{n/2} \,
    \ [\prod_{x \in A} \sqrt{ t_x}\,]} \, \exp\left\{-  \sum_{x \in A}
    \, t_x\right\}
    \, \E\left[
\exp\left\{ \sum_{ e \in \edges } J_{ e}  \, \rho_ e \right\} \right].\, \]
 where  $\{J_x, x \in A\}$ are independent $\pm 1$ coin flips;
  $J_{ e} = J_xJ_y$ if $ e$ connects $x$ and $y$; and 
 $\rho_ e =\theta_ e\,\sqrt{ t_ e} $.
By using \eqref{may16.3}, we get \eqref{may16.2}.
 \end{proof}
 
The form of the joint density also gives us the
conditional density for the currents $\bar k$ given
$\bar t$.

\begin{proposition}  Given $\bar t$, the conditional
distribution for the currents $\bar k$ is proportional
to 
\begin{equation}  \label{oct17.1}
   \Psi(\bar k, \bar \rho) = 
  \prod_{ e \in \edges }
     \frac{\rho_ e^{k_ e}}{k_ e!}, \;\;\;\; \rho_e = \sqrt{t_e} \, \theta_e.
     \end{equation}
\end{proposition}

\begad

The form \eqref{oct17.1} may look like the distribution of
independent Poisson random variables $\{k_e: e \in\edges \}$
where $k_e$ has intensity $\rho_e$.  However, this distribution
is restricted to $\bar k \in \currents$.  
The distribution is that of
independent Poisson distribution {\em conditioned that  the
  $\bar k$ is a current}. 

\endad 
 
 \subsection{Finding the signs}
 
 Given $\bar t = \{t_x: x \in A\}$, the values of the field are
 given by
 \[         Z_x = J_x \, \sqrt {2t_x} ,\]
 where $J_x = \pm 1$ is the sign of $Z_x$.
  A way to specify the signs $\bar J$ is to give the 
  ``positive set''  
 $V = \{x \in A : J_x = 1\}$. 
 
For a positive weight, we can give an algorithm due to Lupu \cite{Lupu}
 to get the Gaussian
field with signs from a realization of the loop soup combined with some
extra randomness.

\begin{itemize}
 \item  Obtain a sample $(\bar k, \bar t)$ of currents and  
 times as above.  This gives the edge weights $\rho_e$.
 \item  Open any edge $e$ with $k_e \geq 1$.
 \item  For each edge $e$, make $e$ open with probability
$    1 - \exp\{- \rho_e\}$.
   An  edge $e$
 is open if it has been opened for either of the two reasons.
We say two vertices are connected if there is a path between
them using  open edges.
 \item  For
 each connected cluster $U$, take an independent random variable
 $J_U  = \pm 1$  with equal probability.
 \item  For $x \in U$, set $Z_x = J_U \, \sqrt{2t_x} . $
 \end{itemize}
 
 \begin{theorem}  The distribution of $\bar Z =(Z_x)$ is that of
 a centered multivariate random variable  indexed by
 $A$ with covariance matrix $G.$
 \end{theorem}

 \begin{proof}  We do it on induction on the cardinality of $A$.
 It is clearly true for a one-point set.   Assume that it
 is true for all sets of cardinality at most $n-1$ and suppose
 that $\#(A) =  n $.  Let $V \subset A$  with
 $1 \leq \#(V) \leq n-1$.
   Let $R_V = \{\bar z \in \R^A: z_x > 0, x \in V
 \mbox{ and } z_x < 0, x \in A \setminus V \}.$  We will write
 elements of $R_V$ as $\bar z = (\bar z_+ , -\bar z_-)$ where $\bar z_+
  \in (0,\infty)^V, \bar z_- \in  (0,\infty)^{A \setminus V}$.
   We will consider
 the density of  $\bar Z$  restricted to $R_V$.  Let
 $G_V, G_{A\setminus V}$ denote the Green's function restricted
 to the edges in $V, A \setminus V$, respectively, and let
\begin{equation}  \label{may16.4}
   D_V = \frac{1}{\det G_V} , \;\;\;\; D_{A \setminus V}
    = \frac 1 {\det G_{A \setminus V}}. 
    \end{equation}
    We
 write $\edges^* = \edges^*(V,A)$ for the set of edges in $\edges_A$
 that have one vertex in $V$ and one vertex in $A \setminus V$.
 \begin{itemize}
 
 \item  In order for the algorithm to output $\bar z \in R_V$
 it is necessary that every edge in $\edges^*$ is closed.  This
 requires it to be closed using both criteria.
 \begin{itemize}
 \item 
  To be closed using the first criterion, the loop soup  must
 not contain any loop in $L$, the set of
 unrooted loops that intersect  both $V$ and $A \setminus V$.
We know that 
 \[   D  = \exp\left\{\sum_{\ell \in \loops(A)}
   m(\ell) \right\} = D_V \, D_{A \setminus V} \, e^{-m(L)}\]
 where $m$ denotes the unrooted loop measure.  Therefore,
 the probability
 that no such loop is chosen in the soup of intensity $1/2$
 is
 \[   e^{-m(L)/2} =   \sqrt{\frac{D}{D_V \, D_{A \setminus V}}}.\]
 \item  Given the realization of $\bar t$, the probability that no
 loop is open using the second criterion is
 \[      \exp\left\{- \sum_{e \in \edges^*} \rho_e \right\}. \]
 
 \end{itemize}
 
 \item Given that the loop soup contains
 no loop in $L$, the algorithm
 acts independently on $V$ and $A \setminus V$.  By the induction
 hypothesis,  
the density of the output of the algorithm, restricted to $R_V$,
is given by 
  \[    f_{V}(\bar z_+) \, f_{A \setminus V}(-\bar z_-) 
       \,  e^{-m(L)/2} \,  \exp \left\{-  \sum_{e \in \edges^* } \rho_e \right\}. \]
Using \eqref{fielddensity} and
  \eqref{may16.4}, we see that this is 
 the same as $f_A((\bar z_+, -\bar z_-))$.
  
  \item This argument computes the density
  $f_A(\bar z)$ for any $\bar z \in R_V$
  such that   $V$ is  not $\eset$ or $A$.
  However, symmetry shows that $f_A(\bar z) = f_A(- \bar z)$
  for $\bar z \in  R_A  $, and since the total integral of the density
  must equal one, we get the result for $V = \eset$
  and $V = A$ as well.

 \end{itemize}

 \end{proof}

 \begin{example}
 Suppose $A = \{x,y\}$, assume that $\theta_{xx} =
 \theta_{yy} = 0$, and let $ a = q^2 = [\theta_{xy}/2]^2$.
 There is only one elementary loop at $x$, $l=[x,y,x]$
 with $q(l) = q^2$.  In particular, the loop
 measure $m_{q} = m_{-q}$.  
 \[  G(x,x) = G(y,y) = \frac{1}{1-a}, \]
 \[  G(x,y) = q \, G(y,y) = \frac{q}{1-a} \]
 \[  \det G = \frac{1}{(1-a)^2} - \frac{a}
   {(1-a)^2} = \frac{1}{1-a}.\]
The distribution of $n_x$ at time $t=1/2$ for the
growing loop at $x$ is 
\[     \nu(k) =  \frac{\Gamma(k + \frac 12)}{k!\,
\Gamma(\frac 12)} a^{k} \, (1-a)^{1/2}.\]
The joint density of $(N_x,T_x)$ is
\[   
  \frac{\Gamma(k + \frac 12)}{k!\,
\Gamma(t)} a^{k} \, (1-a)^{1/2}
\,  \frac{1}{\Gamma(k + \frac 12)} \, t^{k - \frac 12}
 \, e^{-t}.\]
Summing over $k$ we get the density of $(T_x)$ is
\[  \frac{e^{-t}\, \sqrt{1-a}}
   {\sqrt{t \pi}} \,  \sum_{k=0}^\infty \frac{(at)^k}{k!}
      =  \frac{\sqrt{1-a}}
         {\sqrt{ t \pi}}  \, e^{-(1-a) t}, \]
  which is the density of $ Z^2/2 $  where
  $Z$ is a centered normal with variance $1/(1-a) $.

  To get normal random variables with covariance matrix
  $G$, we can let $U,V$ be independent $N(0,1)$ and let
  \[     Z_x = \frac{U}{\sqrt{1-a}}
     , \;\;\;\;  Z_y =  \frac{ q}{\sqrt{1-a}} \, Z_x
         + \frac{\sqrt{1-a}}{\sqrt{1-a}} \, V.\]
 Note that the joint distribution of $(Z_x^2,Z_y^2)$ in this
 case is independent of the sign of $q$.  However the distribution
 of $(Z_x,Z_y)$ does depend on the sign.
 \end{example}
 
 \section{Measures on multiple walks}

 \subsection{An example: nearest neighbor, symmetric walks on $\Z^2$}
\label{simpleex}
 
 Before doing the general theory, we will consider simple
 random walk in $\Z^2$.  We will make use of some planarity
 properties as well as conformal invariance in the scaling
 limit.  For now we set up some notation.
 We will use complex notation, $\Z^2 = \Z + i \Z$; in
 particular, we write just $k$ for the point $k + 0\cdot i$.
 We will write $p$ for the usual random walk weight,
 $p_\edge = 1/4$  for every nearest neighbor directed edge.
  Equivalently, in the notation
of Section \ref{fieldsec}, we can set $\theta_e=1/2$ for each
undirected, nearest neighbor edge.
 We will consider random walk restricted to finite, connected
 open sets $A$ often stopped at the boundary $\p A$.
 
 \begin{itemize}
 
 \item 
 Associated to every finite $A$ there is a domain
$D_A \subset \C$ obtained by replacing each lattice point
with a square of side length one centered at the point.  To
be more precise, we let
\[   \Square = \left\{r+is \in \C: |r|,|s| \leq \frac 12 \right\}, \]
and we define $D_A$ to be the interior of
\[         \bigcup_{z \in A} \left[z + \Square\right]. \]
The boundary
of $D_A$ is a union of unit segments what are edges in the
dual lattice $\Z^2 + (\frac 12 + \frac i2) .$
 
 \item 
 We set 
\[     w_0 = \frac 12 - \frac i 4  \in D_A . \]  
  The {\em (downward) zipper} at $w_0$ is the vertical line
in $\C$ starting
at $w_0$ and going downward until it first reaches
$\p D_A$.  If $p$ is a positive, symmetric, nearest
neighbor integrable weight
on $A$, we
define the corresponding zipper measure $q$ by
saying that $q_e = J_e \, p_e$ where 
$J_e = -1$ if $e$ crosses the zipper and $J_e = 1$
otherwise.   Equivalently, let $k$ be the smallest positive integer,
such that either $-ki$ or $1-ki$ is not in $A$.  Then,
  $J_{e} = -1$ if $e$ connects $-ji$ and $1-ji$ with
$0 < j < k$ and $J_\edge = 1$ otherwise.
\item We say that connected $A$ is {\em simply connected} if $\Z^2 \setminus A$
is also a connected subgraph of $\Z^2$; this is equivalent to saying
that $D_A$ is a simply connected domain. 
\item  We let $\aset$ be the set of finite subsets
of $\Z + i\Z$ containing $0$ and $1$, and we set $\simpleset$
be the collections of such sets that are simply connected.

\item  Here is a topological fact.  Suppose $A \in \simpleset$
and $\{a,b\}$ are distinct boundary edges.  Then we can order
$(a,b)$ such that the following is true.
\begin{itemize}
\item Any $\eta \in \saws_A(a,b)$ using the directed edge $\ora{01}$
crosses the zipper an even number of times.  In particular,
$q(\eta) = p(\eta)$.
\item  Any $\eta \in \saws_A(a,b)$ using the directed edge
$\ora{10}$ crosses the zipper an odd number of times. In particular,
$q(\eta) = - p(\eta)$.
\end{itemize}
We will call this the positive ordering of $\{a,b\}$ (with respect
to this zipper).
\item We let $\oddloops_A$ be the set of unrooted loops in $A$
that intersect the zipper an odd number of times.  Note that
if $A \in \simpleset$,
\[         m_q(\ell) = -m_p(\ell) , \;\;\; \ell \in \oddloops_A , \]
\[      m_q(\ell) = m_p(\ell), \;\;\;  \ell \in \loops_A
\setminus  \oddloops_A.\]
\item If $A \in \simpleset$, 
let $f:D_A \rightarrow \Disk:=\{z \in \C: |z| < 1\}$ be the unique conformal
transformation with $f(0) = 0 , f(a) =1 $ and define $\theta$
by $f(b) = e^{2i\theta}$.  The existence and uniqueness
of the map follows from the Riemann mapping theorem. We set
\[   r_A = r_A(0) = |f'(0)|^{-1} , \;\;\;\;
 S_{A,a,b} = S_A(0;a,b) = \sin \theta .\]
 Using the Koebe $1/4$-theorem from complex analysis we can see that
\begin{equation}  \label{koebe}
        \frac{r_A}{4} \leq \dist(0, \p D_A) \leq {r_A}. 
        \end{equation}
\end{itemize}

\begad

We have made an arbitrary choice of the zipper.  We could take
any curve such that in the scaling limit it gives a simple
curve from the origin to the boundary.  It makes the discussion
a little easier by making a specific choice.  It is
important that the zipper comes up all the way to $\ora{01}$.

We have defined a collection of sets that contain the
ordered edge $\ora{01}$.  This is an arbitrary choice.
If we are interested in sets containing the nearest
neighbor edge $\ora{zw}$
where $w = z + e^{2i\theta}$, we can take $A \in \aset$,
translate by $z$, and then rotate by $\theta$.

\endad

If $D \subset \C$ is a bounded domain containing
the origin, we can make a lattice approximation to $D$.
To be specific, for each positive integer $n$, we let
$  A_{D,n}$ be the connected component containing
the origin of all $z \in \Z + i \Z$ such that
$\Square + z \subset nD$.  We then set $D_n =
  n^{-1} \, D_{A_{D,n}}.$
We get the following properties.
\begin{itemize}
\item  $D_n \subset D$. 
\item  For every $z \in \p D_n$, $\dist(z,\p D)
  \leq \sqrt 2 /n$.
 \item  If $D$ is simply connected, then $D_n$ is
 simply connected.
 \item  If $\zeta \in D  $, 
 then for all $n$ sufficiently large, $\zeta \in
 D_n$.

  \end{itemize}
 
 \subsection{General theory}
  
 If $q$ is a weight on $\overline A$, and hence a measure
 on $\overline \paths_A$, we extend this to a measure on
 ordered pairs of paths using product measure.  The
 next definition makes this precise.
 
 \begin{definition} $\;$
 \begin{itemize}
 \item If  $\x = (x_1,\ldots,x_k), \y = (y_1,\ldots,y_k)$ are
$2k$   points in $\overline A$, we let $\paths_A(\x,\y)$
be the set of ordered $k$-tuples
\[      \bfomega = (\omega^1,\ldots,\omega^k) , \;\;\;\;
    \omega^j \in \paths_A(x_j,y_j).\]
   \item The measure $q$ on $\paths_A(\x,\y)$ is the
 product measure
 \[    q(\bfomega) = q(\omega^1) \cdots q(\omega^k). \]
 \item We write $H_A(\x,\y)$ for the total mass 
which is given by the  product
\[         H_A(\x,\y) = \prod_{j=1}^k H_A(x_j,y_j).\]
    
\end{itemize}
\end{definition}

 There are 
 several reasonable choices for
  extensions of the loop-erased measure $\hat H_A(x,y)$,
   but the following
 has proved to be the most useful.

 \begin{definition} $\;$ 
\begin{itemize}
\item If  $\x = (x_1,\ldots,x_k), \y = (y_1,\ldots,y_k)$ are
$2k$ distinct points in $\overline A$, let  $\saws_A(\x,\y)$
denote the set of ordered $k$-tuples of SAWs
\[    \bfeta = (\eta^1,\ldots,\eta^k) , \;\;\;
  \eta^j \in \saws_A(x_j,y_j) , \]
 that are mutually avoiding
 \[               \eta^i \cap \eta^j = \eset, \;\;\;\;
  1 \leq i < j \leq k.\]

    \item The loop-erased measure $\hat q$
  is defined on $\saws_A(\x,\y)$ by
  \[    \hat q(\bfeta) = q(\bfeta) \, F_\bfeta(A) , \]
  where $\log F_\bfeta(A)  = 
  \log F_{\eta^1 \cup \cdots \cup \eta^k}(A) $
 is the  loop measure of loops in $A$ that intersect
 at least one of the SAWs $\eta^1,\ldots,\eta^k$.
 
 \item  More generally, the measure is defined on
 $\saws_A(x_1,y_1) \times \cdots \times \saws_A(x_k,y_k)$ by
 \[    \hat q(\bfeta) = 1\{\bfeta \in\saws_A(\x,\y)\}
 \, q(\bfeta) \, F_\bfeta(A) . \]

 \item We write $\hat H_A(\x,\y)$ for the total mass
 of the measure
 \[    \hat H_A(\x,\y)=
  \sum_{\bfeta \in \saws_A(\x,\y)} \hat q(\bfeta).\]

\item If $\sigma$ is a  permutation of $\{1,\ldots,k\}$
and $\x = (x_1,\ldots,x_k)$ we write
$ \x^\sigma = (x_{\sigma(1)},\ldots,x_{\sigma(k)})$. 
Note that
\[  H_{A}(\x^\sigma ,\y^\sigma)  = H_A(\x,\y) ,\;\;\;\;\;
 \hat  H_{A}(\x^\sigma ,\y^\sigma)  = \hat  H_A(\x,\y).\]
(It is important that we use the same permutation $\sigma$
to permute the coordinates of $\x$ and $\y$.) 

 \end{itemize}
\end{definition}

In particular, $\hat q$ is   defined on $\saws_A(\x,\y)$ neither
as the
product measure nor as the product measure restricted to mutually
avoiding paths.  
 The definition allows the points to be interior
or boundary points.  It will be easier to consider only the
case of boundary points and the next proposition, which is really
no more than an immediate observation which we do not prove, shows that
one can change an interior point to a boundary point and we need
only multiply the entire measure by a loop term.

\begin{proposition} 
If  $\x = (x_1,\ldots,x_k), \y = (y_1,\ldots,y_k)$ are
$2k$ distinct points in $\overline A$,  then for all  
$\bfeta \in \saws_A(\x,\y)$
 \[         \hat q (\bfeta) = q(\bfeta)
  \, F_{B} (A) \, F_{\bfeta}(A \setminus B) , \]
where $B= \{x_1,\ldots,x_k,y_1,\ldots,y_k\}$.
In particular,
\[   \hat H_A(\x,\y) = F_{B}(A)
       \hat H_{A \setminus B}(\x,\y).\]
\end{proposition}

\begin{example}  If $A \in \simpleset$
 and $x= (x_1,\ldots,x_k), \y =
(y_1,\ldots,y_k)$ are $2k$ distinct boundary points that
appear in order on $\p A$, 
 there is  only  one permutation
$\sigma$ of $\{1,\ldots,k\}$ such that
\[    \hat H_A(\x,\y^\sigma) \neq 0.\]
\end{example}

 From the definition, we immediately get the
Radon-Nikodym derivative with respect to product measure.

 \begin{proposition}$\;$
 \begin{itemize}
 \item  If $\bfeta
  \in \saws_A(\x,\y)$,  then
  \[    \frac{\hat q(\bfeta)}{\hat q(\eta^1) \, \hat q(\eta^2)
    \, \cdots \, \hat q(\eta^k)}  \hspace{2.5in}\]\[ =  
               1\{\bfeta \in \saws_A(\x,\y)\}
                \, \exp\left\{  m[\loops (A;\bfeta)] -
                  \sum_{j=1}^k m[\loops (A;\eta^j)]\right\}. \]
 In particular, if $k=2$, then
\[   {\hat q(\bfeta)} =   
  \hat q(\eta^1) \, \hat q(\eta^2)\, \exp\left\{-m[\loops(A;\eta^1)
   \cap \loops(A;\eta^2)]\right\}. \]
   \item The marginal density on $\eta^1$ is absolutely continuous
 with respect to $\hat q$ on $\saws_A(x_1,y_1)$ with
 Radon-Nikodym derivative
 \[         \hat H_{A\setminus \eta^1}(\x',\y') \]
 where $\x'= (x_2,\ldots,x_k), \y'=(y_2,\ldots,y_k)$. 
 \item \begin{equation}  \label{oct25.1}
  \hat H_A(\x,\y) = \sum_{\bfomega \in 
      V} q(\bfomega), 
      \end{equation}
 where $V$ denotes the set of $\bfomega = (\omega^1,\ldots,
 \omega^k) \in \paths_A(\x,\y)$ such that for $j=2,\ldots,k$,
 \[    \omega^j \cap \left[LE(\omega^1) \cup \cdots \cup LE(
\omega^{j-1}) \right] = \eset . \]
 
 \end{itemize}
 \end{proposition}

 \begad
 We could have used \eqref{oct25.1} as a definition of 
 $ \hat H_A(\x,\y)$ but it is not obvious from this definiton
 that  $\hat H_A(\x^\sigma,\y^\sigma) =  \hat H_A(\x,\y)$
 for permutations $\sigma$.
 \endad
 
 The next proposition shows that we can give the probability
 that a loop-erased walk uses a particular edge in terms of
the total mass of pairs of walks (the past and the future
as viewed by the edge).

%

 \begin{proposition}  Suppose $x_1,y_1,x_2,y_2$ are distinct
 points in $\overline A$ with $x_2,y_1 \in A$.  Let $V$ denote
 the set of SAWs in $\saws_A(x_1,y_2)$ that include the directed
 edge $\edge = \ora{x_2y_1}$.  Then
 \[    \hat q(V) = q_\edge \, F_{\edge}(A) \, 
     \hat H_{A'}(\x,\y), \]
 where $A' = A \setminus \{x_2,y_1\}$.
 \end{proposition}
 
 \begin{proof}  If $\eta \in V$ we write $\eta = \eta^1 \oplus
 \edge \oplus \eta^2$; let $\bfeta = (\eta^1,\eta^2)$;
  and note that $q(\eta) = q(\bfeta ) \, q(\edge)
 $.  The loops that intersect $\eta$ can be partitioned
 into those that visit $\{x_1,y_2\}$ and those that do not giving
 \[      F_\eta(A) = F_\edge(A) \, F_{\bfeta}
    (A'), \]
    and hence 
 \[   \hat q(\eta) =   q(\edge) \, F_{\edge}(A) \, q(\bfeta)
  \, F_{\bfeta}
    (A ').\]
 We now sum over all possible $\bfeta$.
 
 \end{proof}

 \begin{proposition}[Fomin's identity, two paths]  
Suppose $x_1,x_2,y_1,y_2$ are distinct points in $\p A$,
 and let $\sigma$ denote the nontrivial permutation
 of $\{1,2\}$.   Then 
\begin{equation}  \label{oct25.2}
    \hat H_{A}(\x,\y) - \hat H_A(\x,\y^\sigma) =
     H_A(\x,\y) - H_A(\x,\y^\sigma) . 
    \end{equation}
  \end{proposition}
  
  \begin{proof}
 Let $V$ denote the set of $\bfomega=(\omega^1,\omega^2)
 \in \paths_A(\x,\y)$ such that $LE(\omega^1) \cap \omega^2
 \neq \eset$ and let $V_\sigma$ be the corresponding
 set in $\paths_A(\x,\y^\sigma)$.  We will give a bijection
 $\pi: V \rightarrow V_\sigma$ such that for each $\bfomega$,
 $q(\bfomega) = q(\pi(\bfomega)).$  Let $\eta = LE(\omega^1)$.
If $ \bfomega  \in V  $, let  $r$ be the smallest index such that $\eta_r \in \omega^2$;
 let $k$ be the largest
 index such that $\omega^2_k  = \eta_r$; and let $j$
 be the largest index such that $\omega^1_j = \eta_r$.
We write $\omega^i = \omega^{i,-} \oplus \omega^{i,+},$ where
 \[    \omega^{1,-} = [\omega_0^1,\ldots,\omega_j^1], \;\;\;\;
    \omega^{2,-} =  [\omega_0^1,\ldots,\omega_k^2].\]
 We then define $\pi({\bfomega})$ by
 \[     ( \omega^{1,-}  \oplus  \omega^{2,+},
    \omega^{2,-} \oplus \omega^{1,+}).\]
    It can  readily be checked that this gives  the
    necessary bijection.
    
 The bijection shows that the terms on the right-hand side
 of \eqref{oct25.2} corresponding to $V$ and $V_\sigma$ cancel.  The equality of the
 remaining terms is seen by \eqref{oct25.1}.
 \end{proof}

 The two-path result  is a special case of a more general theorem. 
 For a proof of the following see \cite[Proposition 9.6.2]{LLimic}.
 
 \begin{proposition}
 [Fomin's identity]  
Suppose $x_1,x_2,\ldots, x_k, y_1,y_2, \ldots, y_k$ are distinct points in $\p A$.
Then
\[ \sum (-1)^{\sgn(\sigma)} \, \hat H_A(\x,\y^\sigma)
   =  \det \left[H_{A}(x_i,y_j)\right]_{1 \leq i,j \leq k} \]
where the sum is over  all permutations of $\{1,\ldots,k\}$. 
\end{proposition}

We could alternatively write the right-hand side of the above
equation as
\[  \sum (-1)^{\sgn(\sigma)} \,   H_A(\x,\y^\sigma).\]
In the case of simple random walk in a simply connected domain,
topological constraints imply that $\hat H_A(x,y^\sigma)$
is non-zero for at most one permutation $\sigma$.

\begin{corollary} \label{fominsimple}
 Suppose  $A \in \simpleset$ and 
 $x_1,x_2,\ldots, x_k, y_k,y_{k-1},
 \ldots, y_1 $  are distinct points in order on  $\p A$.
  Then,
 \[   \hat H_A(\x,\y) =  \det \left[H_{A}(x_i,y_j)\right]_{1 \leq i,j \leq k} \]
where the sum is over  all permutations of $\{1,\ldots,k\}$. 
\end{corollary}

\begin{proof}  We know that $\hat H_A(\x,\y^\sigma) = 0$ for nontrivial
permutations $\sigma$.  \end{proof}

\begad
If $x_1,x_2,\ldots,x_n, y_1,y_2,\ldots,y_n$ are distinct points  
  on $\p A$, then there may be several permutations $\sigma$
  for which $\hat H_A(\x,\y^\sigma)  \neq  0$.  However,
  one can determine the value of the loop-erased quantities in
  terms of random walk determinants.  We give the idea here; see
  \cite{KW11} and $\cite{KPe}$ for more details.
  
  We will consider pairings of $[2n] = \{1,2,\ldots,2n\}$.
A planar
pairing $\pair$ is a partition into $n$ sets of cardinality $2$
such that nonintersecting
curves can be drawn in the upper half plane
$\Half$ connecting the points.
We write $x \lra y$ if $x$ and $y$ are paired.
There is a one-to-one correspondence between planar pairings
and  ``Dyck paths''; this is combinatorial terminology for 
one-dimensional random walk bridges that stay nonnegative, that is,
  functions $f:\{0,1,\ldots,2n\}
\rightarrow \{0,1,2,\ldots\}$ with $f(0) = f(2n) = 0$ and $f(j+1) -
f(j) = \pm 1$ for each $j$.
The correspondence is given by
\[   f_\pair(k) = \#\{j \leq k: j \mbox{ is paired with
a point} > k \}. \]
This defines a partial order on planar partitions: $\pair \preceq
\pair'$ if $f_\pair \leq f_{\pair'}$.  Given a planar pairing,
let $\x=\x_\pair$ denote the vector of left endpoints in
increasing order and $\y  = \y_\pair= (y_1,\ldots,y_n)$ where
$x_j \lra y_j$.  If $\sigma$ is a permutation of $[n]$, we write
$\pair_\sigma$ for the (not necessarily planar) 
pairing given by $x_j \lra y_{\sigma(j)}$.
   Fomin's identity implies that
\begin{equation}  \label{mar22.1}
    \det\left[H(\x,\y)\right]
  = \sum_{\sigma} (-1)^{\sigma} \hat H(\x,\y^\sigma).
  \end{equation}
Note that we can restrict the sum on the right-hand side
to permutations $\sigma$ such that $\pair_\sigma$ is
a planar pairing since $\hat H(\x,\y^\sigma)=0$ for the others.

The key observation, that we leave as an exercise, is
that if 
  $\pair_\sigma$ is a planar pairing, then
$\pair \preceq \pair_\sigma$.  
We can then write \eqref{mar22.1} as
\[       \det\left[H(\pair)\right] =
    \sum_{\pair'}  M_{\pair, \pair'} \, \hat H(\pair'), \]
where $M_{\pair,\pair'} \in \{0,\pm 1\}$, $M_{\pair,\pair} = 1$
 and $M_{\pair,\pair'} = 0$
unless $\pair \preceq \pair$.  Therefore, if we order
the pairings consistently with $\preceq$, 
 $M$ is an upper triangular
matrix with nonzero diagonal terms and we can invert giving
\[        \hat H  = M^{-1} [\det H].\]
  
 \endad

\begin{itemize}
\item If $X$ is any function on
$\paths_A(x,y)$ we write $\langle X \rangle_q$ for the integral
or ``expectation value''
\[   \langle X \rangle_q= \langle X;A,x,y\rangle_q
 = \sum_{\omega \in \paths_A(x,y)} X(\omega)\, q(\omega). \]  
\item 
If $\omega$ is a path and $\edge$ is a
directed edge, we let $Y_\edge(\omega)$
be the number of times that $\omega$ traverses the directed edge
$\edge$.  We also set $Y_\edge^-(\omega) = Y_{\edge^R}(\omega)$
be  the number of  
  traverses of the reversed edge, and note that
$Y_\edge - Y_\edge^-$ represents the number of ``signed''
traverses of $\edge$.
\item Note that if $q$ is symmetric and $z \in A$, then
\[   \langle Y_\edge - Y_\edge^- ;A,z,z\rangle_q =0\]
since the terms with $l$ and $l^R$ cancel.  
\item We   write
$I_z(\omega),I_e(\omega),I_\edge(\omega)$ for the indicator function that the loop
erasure $LE(\omega)$ contains the vertex $z$, the undirected
edge $e$, and the directed edge $\edge$, respectively.
\end{itemize}

\begin{proposition}  \label{prophallow2} Suppose $q$ is a symmetric integrable
weight
on  $A,$ $x,y \in \partial A$, and
$\edge = \ora{zw}  \in \dedges_A$.   Then
     \[  \langle Y_\edge - Y_\edge^-   ; A, x, y\rangle_q=    
     q_\edge \, F_\edge(A) \, \left[
     H_{A'}(x,z) \, H_{A'}(y,w) - H_{A'}(x,w)
      \, H_{A'}(y,z)\right],\]
   where $A' = A \setminus \{z,w\}$.
\end{proposition}

\begin{proof} Let $Y = Y_\edge, Y^-= Y^-_\edge, X = Y-Y^-$.
If $\omega$ does not visit both $z$ and $w$, then
$Y(\omega) = Y^-(\omega) = 0$; hence we will only
consider $\omega$ that visit both $z$ and $w$.
 Let $\xi,\tau$ be the first and last
indices $j$ with $\omega_j =z$ and $\xi',\tau'$
the corresponding quantities for $w$.
 
 Suppose $\xi < \xi', \tau' < \tau$.
In this case we can write $\omega = [\omega_0= x,
\ldots, \omega_n = y]$ uniquely as
$   \omega = \omega^- \oplus l \oplus \omega^+$
where 
\begin{equation}  \label{oct30.3}
\omega^-=[\omega_0,\ldots,\omega_\xi],\;\;\;
l = [\omega_\xi, \ldots,\omega_\tau],\;\;\;
 \omega^+ = [\omega_\tau, \ldots, \omega_n=y].
 \end{equation}
We can also consider  $\tilde \omega = \omega^-
\oplus l^R \oplus \omega^+$.  Since $\omega_{\tau + 1}
 \neq w$,  $X(\omega) = X(l) = - X(l^R) 
 =
 -X(\tilde \omega)$.
 Since $q$ is symmetric, $q(\omega) = q(\tilde \omega)$.
 Hence
 \[  \sum_{\xi < \xi', \tau' < \tau}
    q(\omega) \, X(\omega) = 0, \]
A similar argument shows that 
\[  \sum_{\xi' < \xi, \tau < \tau'}
    q(\omega) \, X(\omega) = 0.\]
  
   Suppose $\xi < \xi', \tau < \tau'$, and 
   decompose
 $\omega$ as in \eqref{oct30.3}.  Note that
 $X(\omega) = X(l) + 1\{\omega_{\tau+1} = w\}$.
 By comparing $l$ and $l^R$ as in the previous paragraph,
 we see that
 \[  \sum_{\xi<\xi', \tau < \tau'}
    q(\omega) \, X(\omega) = 
    \sum_{\xi < \xi', \tau < \tau'}
    q(\omega) \, 1\{\omega_{\tau+1} = w\}   .\]
 If $\omega_{\tau + 1} = w$, we can write
 \[   \omega = \omega^- \oplus l \oplus \edge
 \oplus l' \oplus \tilde \omega^+, \;\;\;\;l' = [\omega_{\tau+1},
 \ldots, \omega_{\tau'}], \;\;
   \tilde \omega^+ = [\omega_{\tau'},\ldots,\omega_n].\]
 By construction we see that
  $\omega_- \in \paths_{A'}(z,z),
  l \in \paths_A(z,z) , l' \in \paths_{A\setminus \{z\}}(w,w),$
  and $
  \tilde \omega^+ \in \paths_{A'}(w,y).$ 
 Therefore,
 \begin{eqnarray*}
  \sum_{\xi < \xi', \tau < \tau'}
    q(\omega) \, 1\{\omega_{\tau+1} = w\} &  = &
   H_{ A'}(x,z) \, H_{ A'}(y,w) \, q_\edge 
   \, G_A(z,z) \, G_{A\setminus \{z\}}(w,w)\\
   &  = &  H_{ A'}(x,z) \, H_{ A'}(y,w) \, q_\edge \,
      F_\edge(A) .
  \end{eqnarray*}
 A similar argument shows that
  \[ \sum_{\xi' < \xi, \tau' < \tau}
    q(\omega) \, X(\omega) = - H_{ A'}(x,w) \, H_{ A'}(y,z) \, q_\edge \,
      F_\edge(A) .\]

\end{proof}

\begin{proposition}  \label{prophallow} Suppose $q$ is a symmetric integrable
weight
on  $A,$ $x,y \in \partial A$, and
$\edge = \ora{zw}  \in \dedges_A$.   Then
\[  \langle Y_\edge  - Y_\edge^-  ; A, x, y\rangle_q
    =  \langle I_\edge  - I_{\edge^R}; A, x, y \rangle_q.\]
\end{proposition}

\begin{proof}  Let $X = Y_\edge - Y_\edge^-$.
Let $\omega \in \paths_A(x,y)$ and let $\eta =
LE(\omega) = [x_0 = x, x_1,\ldots,x_n = y]$.  Then we can write
\[ \omega = [x,x_1] \oplus l^1 \oplus [x_1,x_2]
 \oplus l^2 \oplus \cdots \oplus
 [x_{n-2},x_{n-1}] \oplus l^{n-1}
   \oplus [x_{n-1},y] , \]
 where $l^j \in \paths_{A \setminus \{x_0,\ldots,x_{j-1}\}}(x_j,x_j)$.
 By construction, we can see that
 \[     X(\omega)
    = I_\edge(\omega) 
   - I_{\edge^R}(\omega)  
    + \sum_{j=1}^{n-1} X(l^j).\]
 Using the fact that the terms with $l^j$ and $(l^j)^R$ cancel
 as in the previous proof, we see that
 \[   \sum_{\omega \in \paths_A(x,y) } q(\omega)
  \sum_{j=1}^{n-1} X(l^j)=0.\]
 Therefore,
 \[ \sum_{\omega \in \paths_A(x,y)}
  X(\omega) \, q(\omega) =
    \sum_{\omega \in \paths_A(x,y)}
  [I_\edge(\omega) -I_{\edge^R}(\omega)]\, q(\omega)  .\]
\end{proof}

The next proposition gives the probability that 
a two-dimensional 
loop-erased random walk uses an undirected edge
in terms of two quantities: the measure of the
set of loops with odd winding number
and boundary Poisson kernels   with respect to the signed
zipper measure.  

\begin{proposition} \label{saints.prop1}
Suppose $A \in \simpleset$; $x,y \in 
\p A$ are   positively ordered with respect to the zipper;
$p$ is simple random walk  with corresponding
zipper weight $q$; 
and $\eta \in \saws_A(x,y)$ is a nearest neighbor
SAW that contains the directed edge $\edge = \ora{01}$. 
Then
\[     \hat q(\eta) =  \hat p(\eta) \, \exp \left\{-2 m_p(\oddloops_A)
 \right\}, \]
where $\oddloops_A$ is the set of unordered loops
$\ell \in \loops_A$ that cross the zipper an odd
number of times.

In particular, if $e$ denotes the undirected edge
associated to $\edge$, 
\begin{eqnarray*}
      \langle I_e; A,x,y \rangle_p  & =  &
   \exp \left\{2 m_p(\oddloops_A)
 \right\}  \,  \langle I_{\edge}- I_{\edge^R} ; A,x,y\rangle_q\\
& = &  \exp \left\{2 m_p(\oddloops_A)
 \right\}  \, q_\edge \, F_\edge^q(A) \, \Delta^q_{A'}(x,y;0,1),
      \end{eqnarray*}
   where $A' = A \setminus \{0,1\}$
   and
 \[  \Delta^q_{A'}(x,y;0,1) =
   H_{A'}^q(x,0) \, H_{A'}^q(y,1) - H_{A'}^q(x,1)
      \, H_{A'}^q(y,0).\]
\end{proposition}

\begin{proof}
We know that 
\[     \hat p(\eta) = p(\eta) \, \exp \left\{\sum_{\ell 
\in \loops_A; \ell \cap \eta \neq \eset} m_p(\ell) \right\}, \]
\[     \hat q(\eta) = q(\eta) \, \exp \left\{\sum_{\ell 
\in \loops_A; \ell \cap \eta \neq \eset} m_q(\ell) \right\}. \]
Since $x,y$ are positively ordered, we have $p(\eta) = q(\eta)$.
Also $m_q(\ell) = - m_p(\ell)$ if $\ell \in \oddloops_A$ and
otherwise $m_q(\ell) = m_p(\ell)$.  For 
topological reasons, we see that if $\ell \in \oddloops_A$ 
and if  $\eta$ contains
$\{0,1\}$, then $\ell \cap \eta \neq \eset. $  This gives
the first equality, and by summing over $\eta$ we see that
\begin{equation}  \label{oct30.1}
     \langle I_\edge; A,x,y \rangle_p=
   \exp \left\{2 m_p(\oddloops_A)
 \right\}  \,  \langle I_{\edge}  ; A,x,y\rangle_q.
 \end{equation}

For the second, we give a similar argument for SAWs $\eta$
that 
contain $\ora{10}$.  The argument is the same except that 
$q(\eta) = - p(\eta)$ and hence
  \[     \hat q(\eta) =  -\hat p(\eta) \, \exp \left\{-2 m_p(\oddloops_A)
 \right\},\]
\begin{equation}  \label{oct30.2}
    \langle I_{\edge^R}; A,x,y \rangle_p=-
   \exp \left\{2 m_p(\oddloops_A)
 \right\}  \,  \langle I_{\edge^R}  ; A,x,y\rangle_q.
 \end{equation}
Adding \eqref{oct30.1} and \eqref{oct30.2} and using
$I_e = I_\edge + I_{\edge^R}$ gives the penultimate equality,
and the last follows from Propositions
\ref{prophallow2} and \ref{prophallow}.
\end{proof}

 \subsection{A crossing exponent in $\Z^2$}
 
 We will calculate a boundary exponent for simple random walk.
 We will first consider the $k=2$ case.  if $N,r$ are nonnegative
 integers, we set 
 \[   A_{N,r} = \left\{x+iy \in \Z + i \Z: 0 < x < rN, 0 < y < \pi   N \right\}. \]
 We will be considering the case with $r$ fixed and $N \rightarrow \infty$,
 in which case $N^{-1} A_{N,r}$ is an approximation of the
rectangle
 \[        D_r = \{x+iy \in \C: 0 < x < r, 0 < y < \pi\}. \]
 Let $0 < y_1 < y_2  < \pi $ and let
 \[  z_j = i y_j, \;\;\;\; w_j = i  y_j + r ,\;\;\;
     z_{j,N} = i   \lfloor y_jN \rfloor , \;\;\;\;  w_{j,N}
    =  \lceil rN \rceil  +  i  \lfloor y_jN \rfloor . \]
     \[   \z_N = (z_{1,N},z_{2,N}), \;\;\;\w_N = (w_{1,N},w_{2,N}),  \]
Let us first fix $r$ and let $N \rightarrow \infty$.     Fomin's identity
(see Corollary \ref{fominsimple}) implies that
\[     \hat H_{A_{N,r}}(\z_N,\w_N)
  =    H_{A_{N,r}}(\z_N,\w_N) -  H_{A_{N,r}}(\z_N,\w_N^\sigma), \]
  where $\sigma$ denotes the transposition on $\{1,2\}$.
  
  In the limit, random walk approaches Brownian motion.  For simple
  random walk and domains that are parallel to the coordinate
  axes, the convergence is very sharp.  Indeed, it can be shown that
  \[ \lim_{N \rightarrow \infty}
             N^{2} \, H_{A_{N,r}}(z_{j,N},w_{k,N}) 
               = h_{\p D_r}(iy_j  , r +iy_k  ) .\]
Here we use $h$ to denote the boundary Poisson kernel for Brownian motion.
More precisely, the Poisson kernel $h(\zeta) := h_{ D_r}(\zeta, r +iy_k )$
is the harmonic function on $D_r$ with boundary value the delta function
at $r+ iy_k$,and $h_{\p D_r}(iy_j  , r +iy_k  )= \p_x h(iy_j)$. 
Therefore,
\[  \lim_{N \rightarrow \infty}
   \frac{\hat H_{A_{N,r}} (\z_N,\w_N)}
   { H_{A_{N,r}} (\z_N,\w_N) } \hspace{2.5in}\]
\begin{equation}  \label{halloween.1}
\hspace{.3in}= \frac{ h_{\p D_r}(z_1,w_1)
       \,  h_{\p D_r}(z_2,w_2) - 
        h_{\p D_r}(z_1,w_2)\,  h_{\p D_r}(z_2,w_1) }
      { h_{\p D_r}(z_1,w_1)  \,  h_{\p D_r}(z_2,w_2) }.
      \end{equation}
We will now take the asymptotics of the right-hand side
 as $r \rightarrow \infty$. 
 
 The boundary Poisson kernel for Brownian motion can be
computed exactly using separation of variables
(see, e.g., \cite[Section 11.3]{BergM}):
\[  h_{  D_r}(x+iy, r + i \tilde y) =
   \frac 2 \pi \sum_{j=1}^\infty \sin(jy) \, \sin(j \tilde y)
     \, \frac{\sinh(jx)}{\sinh(jr)},\]
\[      h_{\p D_r}(iy, r + i \tilde y) =
   \frac 2 \pi \sum_{j=1}^\infty \sin(jy) \, \sin(j \tilde y)
     \, \frac{j}{\sinh(jr)}.\]  
 In particular, as $r \rightarrow \infty$, 
\[h_{\p D_r}(iy, r + i \tilde y)
\sim  \sum_{j=1}^\infty \sin(jy) \, \sin(j \tilde y)
     \,j \, e^{-jr}
  \sim \frac 2 \pi
      \, e^{-r} \, \sin y \, \sin \tilde y, , \]
and hence the denominator of the right-hand side of
\eqref{halloween.1} is asymptotic to
\[     \left(\frac{2}{\pi}\right)^2 \, e^{-2r}\, \sin^2 y_1  \, \sin^2  y_2 .\]
If we plug in the asymptotics for $h_{\p D_r}$
into the numerator we see that the
  $e^{-2r}$ term cancels, and  the numerator
is asymptotic to $  (2/\pi)^2 \,c\, e^{-3r}$ where
\[     c= c(y_1,y_2)
 =      \,2\,\sin^2 y_1 \, \sin^2(2y_2)
  + 2\, \sin^2 y_2 \, \sin^2(2y_1) \hspace{.3in}\]
  \[ \hspace{1.4in}  - 4 \, \sin y_1 \, \sin y_2
   \, \sin(2y_1) \, \sin (2y_2)  . \] 
   In particular, the ratio is asymptotic to $c(y_1,y_2) \, e^{-r}$.
   A similar argument works for $k$ paths, and we leave the calculation
as an exercise.  
   
\begin{exercise} Suppose that $0 < y_1 < y_2 < \ldots < y_n < \pi$.  
Show that there exists $c = c(y_1,\ldots,y_n) > 0$ such
that  as $r \rightarrow \infty$,
 \[   \det \left[  h_{\p D_r}(iy_j, r + i   y_k) \right] 
  \sim c\,(2/\pi)^n \,  e^{-\frac{n(n+1)r}{2}} 
 , \]
and hence
\begin{eqnarray*}
 \det \left[  h_{\p D_r}(iy_j, r + i  y_k) \right]  
  & \sim & c\,(2/\pi)^n \,  e^{-\frac{n(n+1)r}{2}} \,  \prod_{j=1}^n   h_{\p D_r}(iy_j, r + i   y_j) \\
  & \sim & c \, [\sin^2 y_1\, \cdots \, \sin^2 y_n]  \,  e^{-\frac{n(n-1)r}{2}}
 . 
 \end{eqnarray*}
  \end{exercise}

  \begad
  
  The exponent $\frac{n(n+1)}{2}$ is a {\em (chordal)
  crossing exponent} for
  loop-erased random walk.  It can also be computed directly 
  as a crossing exponent for 
  its continuous counterpart the chordal Schramm-Loewner evolution with 
  parameter $\kappa = 2$.  There are corresponding crossing
  exponents for all $\kappa$.
  
  \endad

  \begad
  
  We let $N \rightarrow \infty$ and then $r \rightarrow \infty$ to make
  the calculation easier.  In fact, one can use a finite Fourier series, which is
  really just a diagonalization of a matrix, to find the discrete Poisson kernel
  exactly in terms of a finite sum that is dominated by the initial terms. 
  See, e.g., \cite[Chapter 8]{LLimic}. This allows
  us to take $N,r$ to infinity at the same time as long as $r$ does not go too much
  faster than $N$.
  
  \endad

  \subsection{Green's function for loop-erased random walk in $\Z^2$}
  
  We will now give a very sharp estimate for the probability that
  loop-erased random walk goes through a particular edge.  Recall the
  definitions of $\simpleset, r_A, S_{A,a,b}$ from Section \ref{simpleex}.
  
  \begin{theorem}
  There exist  $c',u > 0$
  such that if $A \in \simpleset$ and $a,b \in 
  \partial_e A$,  then the probability that a loop-erased random walk
  from $a$ to $b$ uses the directed edge $\ora{01}$ equals
\begin{equation}  \label{saints.1}
         c' \, r_A^{-3/4} \, \left[S_{A,a,b}^{3} + O(r_A^{-u}) \right].
         \end{equation}
  \end{theorem}
  
  The error term $O(\cdot)$ is bounded uniformly over all $A,a,b$.
  Let us be  more precise.  The probability that a loop-erased
  random walk uses edge $\edge = \ora{01}$ is
  \[     P(\edge,A,a,b) =
      \frac{\langle I_\edge;A,a,b \rangle }{
            H_{A}(a,b)}.\]
where we have left implicit the simple random walk weight $p$.   Then we
can restate the theorem as saying there exists $C < \infty$ such that
\[       \left|\log P(\edge,A,a,b) - \log (c'\, r_A^{3/4} \, S^{ 3}_{A,a,b}  )\right|
   \leq  \frac{C}{r_A^u  \, S_{A,a,b}^3} . \]
   In particular, if $S_{A,a,b} \geq r_A^{-u/6}$,
   \[  \left|\log P - \log (c'\, r_A^{3/4} \, S^{ 3}_{A,a,b}  )\right|
   \leq  \frac{C}{r_A^{u/2}  } . \]
   We prefer to use the
simpler
form \eqref{saints.1}.
   
 Let us write $A' = A \setminus \{0,1\}$, $q$ for the zipper measure, and 
 \[    \Delta H_A^q(a,b) =   {\left|H_{A'}^q(a,0) \, H_{A'}^q(b,1)
     - H_{A'}^q(a,1) \, H_{A'}^q(b,0)\right|} . \]
 Using Proposition \ref{saints.prop1}, we see that
 \[   P(\edge,A,a,b) = \frac{F_\edge^q(A)}{4} \, \exp\left\{2 m_p(\oddloops_A)
 \right\}
  \, \frac{ \Delta H_A^q(a,b)}{H_A(a,b)} . \]
 The result will follow from three estimates: there exist  $c_1, c_3,u > 0, c_2
  \in \R$
such that
\begin{equation}  \label{estimate1}
   F_\edge^q(A) = c_1 + O(r_A^{-1/2}), 
   \end{equation}
\begin{equation}  \label{estimate2}          m_p(\oddloops_A) = \frac 18\,\log r_A + c_2 + O(r_A^{-u}),  \end{equation}
\begin{equation}  \label{estimate3}    \frac{ \Delta H_A^q(a,b)}{H_A^p(a,b)}  =
      c_3\, S_{A,a,b}^3 + O(r_A^{-u}).  \end{equation}
   
 The relation \eqref{estimate1} follows from
 $F_\edge^q(A) = G_A(0,0;q) \, G_{A\setminus \{0\}}(1,1;q)$
 and the following proposition.    
      
      \begin{proposition}  There exists $c_0, c_0' > 0$, $u > 0$ such that 
   \[  G_A(0,0;q) = c_0 + O(r_A^{-1/2}), \;\;\;\;
      G_{A\setminus \{0\}}(1,1;q)=  c_0' + O(r_A^{-1/2}). \]
 \end{proposition}

 \begin{proof}  We will do the first estimate; the second can be
 done similarly.  Let $L_A = \tilde \loops^1_0(A)$ denote
 the set of elementary loops in $A$ rooted at $0$
 and let $L = L_{\Z^2}$ be the set of elementary loops in $\Z^2$
 rooted at $0$.  Recall that 
 \[            f_A = \sum_{l \in L_A} q(l), \]
  and let
 \[        f = \sum_{l \in L} q(l). \]
 Since
 \[     \sum_{l \in L} |q(l)| = 1, \]
and there are both positive and negative terms, we can see that
$-1 < f < 1$.  It suffices to show that $f_A = f +O(r_A^{-1/2})$,
that is, there exists $c$ such that  
\[   
  \left|\sum_{l \in L \setminus L_A
 }  q(l)\right| 
    \leq c \, r_A^{-1/2}.\]Although we
   can estimate
  the absolute value of the sum by the sum of the
absolute values,  the latter sum does not decay fast enough for us.
We will have to take advantage of some cancellations in the sum.

 Let $K = \{x+iy \in \Z + i \Z: |x|, |y| <r_A/10\}$.
 Using 
 \eqref{koebe} we see that $\overline K \subset A$.  In particular any
 loop in $L \setminus L_A$ 
  can be decomposed
 as
 \[          l = \omega^- \oplus \omega^+ , \]
 where $\omega^-$ is $l$ stopped at the first visit to $\p K$.
 We can further decompose the walk as
 \[     l = \omega^1 \oplus \omega^2 \oplus \omega^+ , \]
 where $\omega^1$ is $\omega^-$ stopped at the last visit to
 $\{0,1\}$ before reaching $\p K$. 
 We now do a third decomposition.  
 
 Let $L_0$ denote the set of loops in $L \setminus L_A$
  as in the previous paragraph
 such that the last vertex of $\omega^1 $ is zero.   Let $L_0'$
 be the set of loops in $L_0$  such 
 that $\omega^2 \cap \{2,\ldots,k\} \neq \eset$,
 where $k=\lfloor r_A/10\rfloor$.   If $l \in L_0'$,
  we write
 \[     \omega^2 = \omega^3 \oplus \omega^4 , \]
 where $\omega^3$ is $\omega^2$ stopped at the first visit
 to $\{2,\ldots,k\}$.  Let $
 \tilde l = \omega^1 \oplus \tilde \omega^3 \oplus \omega^4
 \oplus \omega^+$
 where $\tilde \omega^3$  is  the reflection of
 $\omega^3$ about the real axis --- that is, the real jumps of $\tilde{\omega}^3$
 are the real jumps of $\omega^3$ but the imaginary jumps of 
  $\tilde{\omega}^3$ are the negative of the imaginary jumps of
  $\omega^3$.  Since $\omega^3$ does not use the edge $\{0,1\}$, we can
  see that $q(\tilde \omega^3) = - q(\omega^3)$ and hence
  \[     q(\tilde l )
   = -  q(l ). \]
   This gives
   \[   \sum_{l \in L_0'} q(l) = 0 . \]
 The measure of $\omega^2$ such that  $\omega^2 \cap \{2,\ldots,k\} = \eset$
 is $O(k^{-1/2})$ (see \cite[Section 5.3]{LLimic}).
 We therefore get
 \[    \left| \sum_{l \in L_0  } q(l)\right| \leq
    \sum_{l \in L_0 \setminus L_0'} |q(l)| \leq  \frac c {k^{1/2}}. \]
  
 If $L^1$ is the set of loops as in the previous paragraph
 such that the last vertex of $\omega^1 $ is $1$, we do a similar
 argument using $\{-k,-k+1,\ldots,-2\}$ to show that
  \[    \left| \sum_{l \in L_1  } q(l)\right| \leq
    \sum_{l \in L_0 \setminus L_1'} |q(l)| \leq  \frac c {k^{1/2}}. \]

\end{proof}

   \subsection{The estimate \eqref{estimate2}}
   
 We will study the loop measure of $\oddloops_A$, the set loops
 that cross the zipper an odd number of times.  We will first consider
 the case
 \[   A = C^n := C_{e^n} = \{z \in \Z^2: |z| < e^n\}, \]
and let $\oddloops^n = \oddloops_{C^n}$.  If we restrict to
the sets $C^n$, the estimate \eqref{estimate2} can be written
as\[    m(\oddloops^n) = \frac{n}{8} + c_0 + O(e^{-un}), \]
which follows immediately, if we can establish
\[     m(\oddloops^n \setminus \oddloops^{n-1})
    = \frac 18 + O(e^{-un}).\]
    
To establish this we consider the scaling loop of the random
walk loop measure, the {\em Brownian loop measure}.  The definition
is similar to that for random walk.  We will start with a measure
on rooted loops by giving the analog of $\tilde m$ from
Section \ref{loopmeasuresec}.  A (rooted) loop $\gamma: [0,t_\gamma]
\rightarrow \C$ is a continuous function with $\gamma(0) = \gamma(t
_\gamma)$.   One important probability measure on loops
is the {\em Brownian bridge measure}  $\mbridge$ defined as
the measure on Brownian paths $B_t, 0 \leq t \leq 1$ conditioned
so that $B_0 = B_1 = 0$.  (This is conditioning on an event
of probability zero so some care needs to be taken, but it
is well known how to make sense of this; indeed, there are numerous
equivalent constructions.)   

If we want to specify a loop $\gamma$, we can write a
triple  $(z,t_\gamma,\tilde \gamma)$ where $z$ is the root,
$t_\gamma$ is the time duration, and $\tilde \gamma$
is a loop rooted at $0$ of time duration one (obtained from
$\gamma$ by translation and Brownian scaling).  The rooted 
Brownian loop measure can be defined as the measure
on triples given by
\[      \left({\rm Area}\right) \times \left(
  \frac{1}{2\pi t^2} \, dt \right) \times \mbridge . \]
 The factor $1/2\pi t^2 $ should be read as $(1/2\pi t ) \cdot
 (1/t)$.  The factor $1/2\pi t $ is the ``probability that Brownian
 motion is at the origin at time $0$''; more precisely, it is
 the density at time $t$ evaluated at $z=0$.  The factor $1/t$
 is the analog of the $1/|l|$ factor in the definition of $\tilde m$.
 
 This gives the Brownian loop measure in all of $\C$.  For loops
 in a domain $D$ we restrict the measure to such loops.  This is
 an infinite,  $\sigma$-finite measure, because the measure of
 small loops blows up.  However, if $D$ is a bounded domain, and
 $\epsilon > 0$, the measure of loops in $D$ of diameter at
 least $\epsilon$ is finite (it goes to infinity as $\epsilon
 \downarrow 0$). 
 
 The amazing and useful fact about the Brownian loop measure
 is that it is conformally invariant, at least if viewed as a
 measure on unrooted loops (these are defined in the obvious way).
 In other words, if $D$ is a domain and $f: D \rightarrow f(D)$ is
 a conformal transfrormation then the image of the loop measure in $D$
 under $f$ is the same as the loop measure on $f(D)$. (One does need
 to worry about the parametrization --- there is a time change 
 which is the same as 
  that for the usual conformal invariance of Brownian motion.)
 In particular, if we consider the measure of loops in the disk of
 radius $e^{n}$ that are not contained in the disk of radius
 $e^{n-1}$ but that have winding number about $0$, then this value is
 independent of $n$.  Indeed, it can be computed (we will not do
 it here) and the answer is $1/8$. 
   
Although we will not give the details, we will
   describe how to show that the random walk loop measure
   converges to the Brownian measure
 We will couple the rooted random walk loop measure, giving measure
 $(2n)^{-1} \, 4^{-2n}$ to each loop of length $2n$ with the corresponding
 Brownian loop measure.  We will use the fact
 that for one dimension walks, the probability of being at
 the origin at time $n$ is
 \[p_n:=  \frac{1}{ 
 \sqrt{\pi n}} \, \left[1 - \frac 1{8n} + O(n^{-2})\right],\]
 and by a well known trick gives for two dimensional random walk
 \[  \Prob\{S_{2n} = 0\}=   p_n^2
  = \frac{1}{ \pi n} \, \left[1 - \frac 1{4n} + O(n^{-2})\right].\]
  Note that if
  \[   q_n =\int_{n-\frac 38}^{n+ \frac 58} \, \frac{dt}{2\pi t^2}  , \;\;\;\;\;
    q_n' := \frac{1}{2n} \, \Prob\{S_{2n} = 0\}
    , \]
    then
    \[  |q_n - q_n'| \leq c_1 \, n^{-4} . \]
   For each $(z,n)$ we let $(K_{z,n}, K'_{z,n})$ be coupled Poisson random
   variables with parameters $q_n,q_n'$, respectively, coupled so
   that $\Prob\{K_{z,n}  \neq K'_{z,n}\}  \leq |q_n - q_n'|$.
   We let $(K_{z,n},K'_{z,n}), z \in \Z^2, n \geq 1$ be independent.
We also have a coupling of Brownian bridge $B_t, 0 \leq t \leq n$
and random walk bridge $S_t, 0\leq t \leq 2n$ such that
\[    \Prob\{\max_{0 \leq t \leq n}|B_t - S_{2t}| \geq c_2 \, \log n\}
    \leq c_2 \, n^{-4}. \]
We then use this to construct the coupling.  Whenever a random walk and
Brownian pair $(z,n)$ occur we do the following:
\begin{itemize}
\item  Choose from the $(B,S)$ distribution for $n$,
\item  Let the random walk loop be $S + z$.
\item   Choose $t \in [n-\frac 38, n+ \frac 58]$ from density
$c \, nt^{-2}$.
\item Scale $B$ from time $n$ to time $t$ (this is not much of a change).  
\[      W_s =  (t/n)^{1/2} \, B_{sn/t}, 0 \leq s \leq t  . \] 
\item  Let the Brownian loop be $W + z + Y$ where $Y$ is a uniform
random variable on the square of side length one about the origin.
\end{itemize}

After the estimate is proved for $\oddloops^n$, it is done for more
general domains again using conformal invariance of the Brownian loop
measure.   
    
\subsection{The estimate \eqref{estimate3}}

We first consider the denominator $H_A(a,b) = H^p_A(a,b)$.  
In this setup, this was considered in \cite{LKozdron} where it was
shown that there is an absolute constant $c$ such that
\[     H_A(a,b) = c \, [\sin^{-2} \theta] \, H_A(0,a) \, H_A(0,b) 
  \, [1 + O(r_A^{-u})], \]
at least if $\sin \theta \geq r_A^{-u}$ (explicit  values
of $c$ and $u$ were given but we will not use them).  The way
to think of this result is that the term $H_A(a,b)$ has three
factors: one local factor at $a$ measuring the probability
of escaping the boundary there; a similar local factor at $b$;
and one global factor that is a conformal invariant.  Similarly,
$H(0,a)$ has the local factor at $a$ and a conformal invariant.
Given this, one needs to estimate the terms like
\[          \frac{ H_A^q(0,a)}{H_A^p(0,a)}. \]
For these terms the ``local factor'' at $a$ cancels and the
limit should be a conformally invariant
quantity about Brownian motion.

To see what the limit should be in our case, let us consider the
continuous case.  Let $D$ be a bounded domain containing the origin and
let $\lambda(t): 0 \leq t \leq 1$ be a ``zipper'', that is, a simple
curve with $\lambda(0) = 0, b:= \lambda(1)  \in \partial D, \lambda[0,
1) \subset D$. Let $a \in \partial D \setminus \{b\}$, and let
$\tilde D = D \setminus \lambda[0,1]$.  An example would be $\lambda$
the vertical line segment starting at $0$, going downward, and stopping
at the first visit to $\partial D$.  Let $f:D \rightarrow \Disk$
be the  conformal transformation with $f(0) = 0, f(a) = - 1$.
 Let $\tilde \lambda = f \circ \lambda$
which is a zipper in $\Disk$ from $0$ to $\partial \Disk$. 

If $z \in \tilde D \setminus \{0\}$, consider a curve $B_t, 0 \leq t \leq \tau$
where $B_0 = z, B_\tau = a$ and assume that $0 \not\in 
 B[0,\tau]$.  The example we have in mind is $B_t$ is an $h$-process
  that is, Brownian motion ``conditioned''  to leave $D$
  at $a$.   We want to consider $(-1)^{J}$ where $J$ is the number
  of times that $B$ crosses the zipper $\lambda$.  This does not quite
  make sense for curves such as Brownian motion since there are infinitely
  many crossings.  But we will make sense of this in terms of arguments.
 Define $\theta_t$ by $f(B_t) = e^{2i\theta_t}$.  Note that $\theta_0$ is
 only defined up to an additive multiple of $\pi$; we
  we make an arbitrary
 choice of $\theta_0$ but then  require $\theta_t$ to be continuous in $t$.
 This is well defined assuming the curve does not go through the origin.
In this case $\theta_\tau$ is well define and $\theta_\tau - \theta_0$
is independent of the arbitrary choice for $\theta_0$.      
We then define $J$ to be $+1$ if $\theta_\tau =   2k \pi$ for some
integer $k$ and define $J$ to be $-1$ if $\theta_\tau =
  (2k+1) \pi$ for some $k$.  We then set
   \[    g(z) = \E^z_a\left[(-1)^{J} \right] \]
  where we write $\E_a$ to denote expectations with respect to 
the
  $h$-process corresponding to  Brownian motion conditioned to leave
  $D$ at $a$.

 To compute $g$, we first note that $g$ is conformally invariant
 and so it suffices to compute it when $D = \Disk, a = -1$.  
 Let $l = [0,1)$ denote the radial line 
 is antipodal to $-1  $.  By symmetry we can see that
 $  g(z) = 0$ for $z \in l$, and hence by   the strong Markov
 property,   that $  g(z)$
 is the probability that an $h$-process in $\Disk$ toward $-1$
  starting at $z$
 reaches $-1$ without hitting the antipodal line $l$.
 This can be written as
 \[     g (z) =  \frac{h_{D}(z,-1)}{h_\Disk(z,-1)},
 \;\;\;\;  D = \Disk \setminus l.\] 
 Here $h$ denotes the Poisson kernel which we will normalize so that
 $h_\Disk(0,-1) = 1$.  As $z \rightarrow 0$,   $h_\Disk(z,-1) = 1
 +O(z)$.
To compute  $ h_{D}(z,-1)$ it is somewhat easier
to consider the upper half disk $D_+ = \Disk \cap \Half$.
    Note that $F(z) = z^2$
takes $D_+$ conformally onto $D$. A computation
shows that 
\[      h_{D_+}(z,i) = 4 \, \Im(z)  
 \, [1 + O(|z|)].\]
(One can see that it must be asymptotic to $c \, \Im(z)$ as
$z \rightarrow 0$ by considering the ``gambler's ruin'' estimate
for the imaginary part of the Brownian motion.)
The scaling rule for the Poisson kernel gives
\[       h_{D_+}(z,-1) = |F'(z)|^2 \, 
 h_{D}(z^2 , e^{2i\theta})
  = 2 \,  h_{D}(z^2 , e^{2i\theta}),\]
and hence
\[   g(re^{2i\psi}) =  h_D(re^{2i\psi},-1)\, [1 + O(r)]
 = 2 \,\sqrt r
 \, [\sin \psi] \,    \,
      [1 + O(\sqrt r)].\]
  
 The reader may note that the $\sin^3$ term in our results consists
 of two factors: we have a $\sin^2$ coming from $H^p_A(a,b)$ and then
 one extra $\sin$ coming from the ratio   $H^q_A/H^ p_A$.

\end{document}